\documentclass[oneside,reqno, twoside, 11pt]{amsart}
\usepackage[pdftex]{graphicx}
\usepackage{color}
\usepackage[T1]{fontenc}
\usepackage{lmodern}
\usepackage[english]{babel}
\usepackage{mathrsfs}  

\usepackage{upgreek}

\usepackage{amsfonts}
\usepackage{extpfeil}
\usepackage{verbatim}

       \usepackage{BOONDOX-cal}

\usepackage{microtype}
\usepackage[new]{old-arrows}
\usepackage{amsmath}
\usepackage[all, cmtip]{xy}
\usepackage{tikz-cd}

\usepackage{amsthm}
\usepackage{mathtools}

\usepackage{bbm}
\usepackage{paralist}
\usepackage[T1]{fontenc}
\usepackage[colorlinks,citecolor=magenta,pagebackref=true,urlcolor=magenta, bookmarksdepth=3]{hyperref}
\usepackage[nodisplayskipstretch]{setspace}
\usepackage{breqn}

\usepackage{nicefrac}
\usepackage[font=small,labelfont=bf]{caption}
\usepackage[labelformat=simple]{subcaption}

\usepackage{epsfig}

\DeclareFontFamily{OT1}{pzc}{}
\DeclareFontShape{OT1}{pzc}{m}{it}{<-> s * [1.10] pzcmi7t}{}
\DeclareMathAlphabet{\mathpzc}{OT1}{pzc}{m}{it}

\makeatletter
\newtheorem*{rep@theorem}{\rep@title}
\newcommand{\newreptheorem}[2]{%
	\newenvironment{rep#1}[1]{%
		\def\rep@title{#2~\ref{##1}}%
		\begin{rep@theorem}}%
		{\end{rep@theorem}}}

\makeatother
\theoremstyle{plain}
\newtheorem*{thm*}{Theorem}
\newtheorem{thm}{Theorem}[section]
\newtheorem{mthm}{\bf Theorem}%[subsection]

\newtheorem{mcor}[mthm]{\bf Corollary}

\newreptheorem{mthm}{Main Theorem}
\newreptheorem{mcor}{Corollary}
\newreptheorem{mlem}{Lemma}
\newreptheorem{thm}{Theorem}

\newtheorem{cor}[thm]{Corollary}
\newtheorem{lem}[thm]{Lemma}

\newtheorem*{lem*}{Lemma}
\newtheorem{prp}[thm]{Proposition}

\newtheorem{obs}[thm]{Observation}

\theoremstyle{definition}

\newtheorem{ex}[thm]{Example}

\newtheorem{rem}[thm]{Remark}

\newcommand{\longtwoheadrightarrow}{\longrightarrow\hspace{-1.2em}\rightarrow\hspace{.2em}}

\newcommand{\ann}{\mathrm{ann}}

\newcommand{\supp}{\mathrm{supp}\hspace{1mm}}

 \newcommand{\R}{\mathbb{R}}

\newcommand{\RR}{\varPsi}

\newcommand{\tet}{\vartheta}

\newcommand{\st}{\mr{st}}

\newcommand{\sbseq}{\subseteq}
\newcommand{\spseq}{\supseteq}

\newcommand{\vanish}[1]{}

\def\V{{\bf V}}

\def\im{\mathrm{im}}%\hspace{1mm}}
\def\ker{\mathrm{ker}}%\hspace{1mm}}
\def\kerc{\mathrm{ker}^{\circ}\hspace{-0.75mm}}
\def\coker{\mathrm{coker}}
\def\sbs\subset
\def\sbseq{\subseteq}

\def\langle{\left<}
\def\rangle{\right>}

\def\({\left(}
\def\){\right)}
\def\no={\,{\,|\!\!\!\!\!=\,\,}}

\def\no={\,{\,|\!\!\!\!\!=\,\,}}
\def\sbseq{\subseteq}

\def\sbseq{\subseteq}
\def\sbs\subset
\def\spseq{\supseteq}

\newcommand{\xqedhere}[2]{%
	\rlap{\hbox to#1{\hfil\llap{\ensuremath{#2}}}}}

\newcommand\Defn[1]{\textbf{#1}}
\newcommand{\cm}[1]{}
\newcommand\mc[1]{\mathcal{#1}}

\newcommand\mbf[1]{\mathbf{#1}}

\newcommand\mr[1]{\mathrm{#1}}

\newcommand{\bigslant}[2]{{\raisebox{.3em}{$#1$} \Big/ \raisebox{-.3em}{$#2$}}}

\newcommand\x{\mathbf{x}}

\DeclareMathOperator{\lk}{lk}
\DeclareMathOperator{\susp}{susp}

\DeclareMathOperator{\Lk}{lk}

\DeclareMathOperator{\St}{st}

\newcommand{\KK}{\mathcal{K}}
\newcommand{\MW}{\mathcal{M}}

\title[Beyond Positivity]{Combinatorial Lefschetz theorems beyond positivity}
\author{Karim Adiprasito}

\address{\emph{Karim Adiprasito}, Einstein Institute of Mathematics, Hebrew University of Jerusalem, Jerusalem, Israel \emph{and} Department of Mathematics, KTH Stockholm, Stockholm, Sweden}
\email{adiprasito@math.huji.ac.il}

\date{01.07.2019}

%Primary  14M25, 05C38 ; Secondary 32S50, 52C25,  13F55}

\begin{document}
	
	\begin{abstract}
		Consider a simplicial complex that allows for an embedding into $\mathbb{R}^d$. How many faces of dimension $\nicefrac{d}{2}$ or higher can it have? How dense can they be?
		
		This basic question goes back to Descartes' "Lost Theorem" and Euler's work on polyhedra. Using it and other fundamental combinatorial problems, we introduce a version of the K\"ahler package beyond positivity, allowing us to prove the hard Lefschetz theorem for toric varieties (and beyond) even when the ample cone is empty. A particular focus lies on replacing the Hodge-Riemann relations by a non-degeneracy relation at torus-invariant subspaces, allowing us to state and prove a generalization of theorems of Hall and Laman in the setting of toric varieties and, more generally, the face rings of Hochster, Reisner and Stanley. This has several applications to quantitative and combinatorial topology, among them the following:
		\begin{compactenum}
			\item We fully characterize the possible face numbers of simplicial rational homology spheres, resolving the $g$-conjecture of McMullen in full generality and generalizing Stanley's earlier proof for simplicial polytopes.  The same methods also verify a conjecture of K\"uhnel: if $M$ is a triangulated closed $(d-1)$-manifold on $n$ vertices, then
			\[\binom{d+1}{j}\mathrm{b}_{j-1}(M)\ \le \ \binom{n-d+j-2}{j}\ \quad \text{for}\ 1\le j\le \frac{d}{2}.\] 
			\item We prove that for a simplicial complex that embeds into $\mathbb{R}^{2d}$, the number of $d$-dimensional simplices exceeds the number of $(d-1)$-dimensional simplices by a factor of at most $d+2$. This generalizes a result going back to Descartes and Euler, and resolves the Gr\"unbaum-Kalai-Sarkaria conjecture. A consequence of this is a high-dimensional version of the celebrated crossing number inequality of  Ajtai, Chv{\'a}tal, Leighton, Newborn and Szemer{\'e}di: For a map of a simplicial complex $\varDelta$ into $\R^{2d}$, the number of pairwise intersections of $d$-simplices is at least
		\[\frac{f_d^{d+2}(\varDelta)}{(d+3)^{d+2}f_{d-1}^{d+1}(\varDelta)}\]
		provided $f_d(\varDelta)> (d+3)f_{d-1}(\varDelta)$.
		\end{compactenum}
	\end{abstract}
	
	\maketitle
	
	\newcommand{\AR}{\mathcal{A}}
	\newcommand{\BR}{\mathcal{B}}
	\newcommand{\CR}{\mathcal{C}}
	\newcommand{\Mu}{M}
	\newcommand{\Soc}{\mathcal{S}\hspace{-1mm}\mathcal{o}\hspace{-1mm}\mathcal{c}}
	\newcommand{\Socl}{{\Soc^\circ}}

\section{Introduction}
	
\subsection{The approach via Hodge theory}	The hard Lefschetz theorem, in almost all cases that we know, is connected to rigid algebro-geometric properties. Most often, it comes with a notion of an ample class, which not only induces the Lefschetz theorem but the induced bilinear form satisfies the Hodge-Riemann relations as well, which give us finer information about its signature (see~\cite{Lazarsfeld, Voisin}). 

Even in the few cases that we have the hard Lefschetz without the Hodge-Riemann relations, they are often at least conjecturally present in some form, as for instance in the case of Grothendieck's standard conjectures and Deligne's proof of the hard Lefschetz standard conjecture, see~\cite{Deligne}. This connection is deep and while we understand Lefschetz theorems even for singular varieties, to this day, we have no way to understand the Lefschetz theorem without such a rigid atmosphere for it to live in.
	
	The goal of this paper is to provide a different criterion for varieties to satisfy the hard Lefschetz property that goes beyond positivity, and abandons the Hodge-Riemann relations entirely (but not the associated bilinear form); instead of finding Lefschetz elements in the ample cone of a variety, we give general position criteria for an element in the first cohomology group to be Lefschetz. The price we pay for this achievement is that the variety itself has to be sufficiently "generic". 
	
\subsection{Beyond positivity and the Hall-Laman relations}	We therefore turn to toric varieties, which allow for a sensible notion of genericity without sacrificing all properties of the variety, most importantly, without changing its Betti vector. Specifically, we consider varieties with a fixed equivariant cohomology ring, and allow variation over the Artinian reduction, that is, the variation over the torus action. The main result can be summarized as follows:
	\enlargethispage{8mm}
	\begin{mthm}\label{mthm:gl}
		Consider a PL $(d-1)$-sphere $\varSigma$, and the associated graded commutative ring $\R[\varSigma]$. Then
		there exists an open dense subset of the Artinian reductions $\mathcal{R}$ of $\R[\varSigma]$ and an open dense subset $\mathcal{L} \subset \AR^1(\varSigma)$, where $\AR(\varSigma)\in \mc{R}$, such that for every $k\le\nicefrac{d}{2}$, we have:
		\begin{compactenum}[\bf (1)] 
			\item \emph{Generic Lefschetz property:} For every $\AR(\varSigma)\in \mathcal{R}$ and every $\ell \in \mathcal{L}$, we have an isomorphism 
			\[\AR^k(\varSigma)\ \xrightarrow{\ \cdot \ell^{d-2k} \ }\ \AR^{d-k}(\varSigma). \]
			\item \emph{Hall-Laman relations:} 
			The \Defn{Hodge-Riemann bilinear form}  
			\[\begin{array}{rccccc}
			\mr{Q}_{\ell,k}:&\AR^k(\varSigma)& \times &\AR^k(\varSigma) & \longrightarrow &\ \AR^d(\varSigma)\cong \R \\
			&a		&	& b& {\xmapsto{\ \ \ \ }} &\ \mr{deg}(ab\ell^{d-2k})
			\end{array}\]
			is nondegenerate when restricted to any squarefree monomial ideal in $\AR(\varSigma)$, as well as the annihilator of any squarefree monomial ideal.
		\end{compactenum}	
	\end{mthm}

For future reference, and as used in the main theorem, we shall say generic if we mean generic in the strongest sense: the properties are satisfied for an open, dense subset of the allowed space of configurations. We shall extend this result to homology spheres and homology manifolds in later sections, but for cleanness of the presentation restrict to the case of PL spheres first. In particular, we do not rely on the Pachner theorem for our proof.

\subsection{Genericity of the Poincar\'e pairing}

A critical auxiliary result on the way is the following genericity criterion for the Poincar\'e pairing.

	\begin{mthm}[Biased Poincar\'e duality]\label{mthm:opd}
		Consider a PL $(d-1)$-sphere $\varSigma$, and the associated graded commutative ring $\R[\varSigma]$. Then
		there exists an open dense subset of the Artinian reductions $\mathcal{R}$ of $\R[\varSigma]$ such that for every $k\le\nicefrac{d}{2}$ the perfect pairing
			\[\begin{array}{rccccc}
\AR^k(\varSigma)& \times &\AR^{d-k}(\varSigma)& \longrightarrow &\ \AR^d(\varSigma)\cong \R \\
			a		&	& b& {\xmapsto{\ \ \ \ }} &\ \mr{deg}(ab)
			\end{array}\]
		is non-degenerate in the first factor when restricted to any squarefree monomial ideal $\mc{I}$.
	\end{mthm}

	We shall see that biased Poincar\'e duality and the Hall-Laman relations	 are not related to regularity  of toric varieties, but rather a generic property of the same. In particular, even smooth varieties do not satisfy these refining properties in general (see Example~\ref{ex:smooth}), but another variety with the same equivariant cohomology does. Our approach to the Lefschetz theorem is therefore somewhat transversal to the classical proofs of the Lefschetz theorem. 	
	
Biased Poincar\'e duality can perhaps be visualized most easily via the following geometric metaphor: every element $x$ in $\mc{I}^k$ pairs with some element $y$ in $\AR^{d-k}(\varSigma)$. We wish to modify $y$ to an element $y'\in \mc{I}^{d-k}$ so that
\[ \mr{deg}(xy')\ =\  \mr{deg}(xy)\ \neq\ 0.\]
Ideally, one could attempt to achieve this by pushing along a Lefschetz pencil, or a suitable Morse function as used by Andreotti and Frankel \cite{AndreottiFrankel, Milnor}, indicating that the question may be associated to the Lefschetz theorem on the simplicial complex associated to $\R[\varSigma]/\mc{I}$. We will forgo the geometric intuition, and work with the Lefschetz theorem directly.
		
\subsection{Why genericity? Limits of positivity}	To a discrete geometer as well as an algebraic geometer, the reason we need to come up with the Hall-Laman relations and the generic Lefschetz theorem lies in the limited availability of positivity in this context. 

To a discrete geometer, this is most familiar under the name of convexity, appearing in the Lefschetz theorem and in particular McMullen's proof in the form of the theory of mixed volumes of convex bodies. To an algebraic geometer, it comes in the form of K\"ahler structures that smooth projective varieties enjoy.

And indeed, aside from more exotic combinatorial constructions where the Lefschetz theorem holds for more trivial combinatorial reasons (see also our discussion of Laman's theorem), the first part of Theorem~\ref{mthm:gl} is known mainly in the case when $\varSigma$ is combinatorially equivalent to the boundary of a polytope, in which case we can conclude the Lefschetz theorem from the hard Lefschetz theorem for rationally smooth toric varieties~\cite{BBDG} or using the specialized approach of McMullen~\cite{McMullenInvent}. This allows us to treat all spheres of dimension two and lower by Steinitz' theorem.

Unfortunately, already in dimension three PL spheres rarely generate fans, not to even speak of fans that support an ample class~\cite{Alon, GP, KalaiM, PZ, Rudin}. 

And what is more, we shall extend the hard Lefschetz theorem even to general homology spheres and then further to closed orientable manifolds in Section~\ref{sec:mani}, which are beyond any reach of the only combinatorial program available, the program via local moves (see Section~\ref{sec:local-move}). 

	The generic Lefschetz theorem is the classical version conjectured usually in connection to the $g$-conjecture, see also~\cite{Swartzthirty}. The Hall-Laman relations instead replace the Hodge-Riemann relations in their role towards the Lefschetz theorem for ample classes. We will motivate their name when we discuss the Gr\"unbaum-Kalai-Sarkaria conjecture in Section~\ref{sec:Laman}.
	
\subsection{Key ideas}	The proof relies on several new techniques towards measuring and exploiting genericity in intersection rings, in particular with regard to the Poincar\'e pairing. The key ideas can be summarized as follows:
	\begin{compactenum}[(1)]
		\item To replace the classical signature conditions on the Hodge-Riemann bilinear form, we introduce Hall-Laman relations of the Lefschetz theorem that posits the Hodge-Riemann form restricts well to certain ideals. We introduce it in Section~\ref{sec:pairings}.
		\item Section~\ref{sec:perturbation} provides the critical perturbation lemma that details how one obtains the Lefschetz theorems from properties of the pairing. It applies immediately to prove the Lefschetz theorem for nicely decomposable spheres.
		\item Section~\ref{sec:railway} introduces the notions of railways in manifolds, allowing us to do surgery on a manifold to simplify the objects we need to prove the Lefschetz theorem on.
	\end{compactenum}
	
	More informally: We will argue that kernels of a generic Lefschetz theorem are associated with subcomplexes of the simplicial sphere $\varSigma$, in the sense that they are minimally supported on that subcomplex.
	
	 We then use the Hall-Laman relations and the perturbation lemma to show that the kernel can be decreased in size of support, first in the case of nicely decomposable spheres, and then using railways for general manifolds. As we will see, the Hall-Laman relations are naturally proved first, with the Lefschetz theorem being more of an afterthought, and with the case $d=2k$ of the Hall-Laman relations holding special importance.
	
\subsection{Applications to toric varieties and combinatorics } The Lefschetz theorem is therefore as announced valid for generic Artinian reductions. In particular, the more algebro-geometrically inclined reader may consult the following corollary for easier visualization, obtained from Theorem~\ref{mthm:gl} as open dense Artinian reductions can be chosen rationally.
	
	\begin{mcor}
		Consider $\mathfrak{F}$ a complete simplicial fan in $\R^d$. Then, after perturbing the rays of $\mathfrak{F}$ to a suitable rational fan $\mathfrak{F}'$, the Chow ring of the toric variety $X_{\mathfrak{F}'}$ satisfies the hard Lefschetz property with respect to a generic Picard divisor, while the equivariant Chow ring remains unchanged from $X_{\mathfrak{F}}$ to $X_{\mathfrak{F}'}$.
		
		Moreover, we can find such a perturbation in a way that the Hodge-Riemann pairing with respect to this generic Picard divisor is non-degenerate when restricted to the ideal in the Chow ring generated by any open subset, that is, any order filter, of the torus-orbit stratification.
	\end{mcor}
	
	Notably, the requirement that there exists an ample class has vanished from the assumptions. The second part, the Hall-Laman relations, are more than a bonus to this fact, but play a similarly crucial role to the generic Lefschetz theorem as the Hodge-Riemann relations play for the proof of the classical hard Lefschetz theorem. 
		
	Consequences of Theorem~\ref{mthm:gl} are legion, too many to collect them all here individually. Among them are the following that we explain briefly. We stand on shoulders of giants here, as many of these results were already established assuming Theorem~\ref{mthm:gl}(1).
	\begin{compactenum}[(1)]
		\item \emph{McMullen's $g$-conjecture, see \cite{zbMATH03333960}:} Chief among the applications of Theorem~\ref{mthm:gl}(1) is the characterization of face vectors of triangulated manifolds, following and resolving the ingenious $g$-conjecture of McMullen.
	
		In particular, Theorem~\ref{mthm:gl}(1) proves that the set of $f$-vectors, that is, the number of vertices, edges, two-dimensional faces etc.\ of a PL sphere is also the $f$-vector of a simplicial polytope. This is achieved by realizing, using that theorem, that the $f$-vector is but a linear transformation of the primitive Betti numbers 
		\[g_i(\varSigma)\ =\ \dim \bigslant{\AR^i(\varSigma)}{\ell\AR^{i-1}(\varSigma)}\]
		of $\AR^\ast(\varSigma)$. This implication was first observed by Stanley, who then used the hard Lefschetz theorem for rationally smooth toric varieties to demonstrate the $g$-conjecture in the case of spheres arising as boundaries of simplicial polytopes~\cite{StanleyHL}. These Betti numbers are then characterized as being the Betti vectors of commutative graded algebras generated in degree one, the constructive side being provided by Billera and Lee~\cite{BL}.  The later generalizations of Theorem~\ref{mthm:gl} found in
 Section~\ref{sec:mani} extend this to characterize the $f$-vectors of rational homology spheres.
		
		Going beyond spheres, Theorem~\ref{mthm:gl}(1) provides necessary conditions on the possible face vectors of triangulations of a fixed rational homology manifold $\Mu$, see also~\cite{KN, Swartzthirty} for details. In particular, it also applies to bound the complexity of triangulated manifolds, and resolves a conjecture of K\"uhnel \cite{Kuhnel}: If $\Mu$ is a triangulated $(d-1)$-dimensional closed rational homology manifold on $n$ vertices, then in terms of its rational Betti numbers, we have
		\[\binom{d+1}{j}\mr{b}_{j-1}(\Mu)\ \le \ \binom{n-d+j-2}{j}\ \quad \text{for}\ 1\le j\le \frac{d}{2}\] 
see also~\cite[Section 4.3]{KN}.		
		\item \emph{Complexity measures for connectivity:} Given a (simplicial rational homology) sphere~$\varSigma$ of dimension $d-1$, and a nonnegative integer $k$ less than~$d$, one can measure the "complexity" of the triangulation in dimension $k-1$ in several ways. One especially fruitful invariant for this complexity is
		$\widetilde{H}_{k-1}(\varSigma_{|W})$ for $W$ a subset of vertices $\varSigma^{(0)}$ of $\varSigma$ and $\varSigma_{|W}$ the induced subcomplex on $W$. For a numerical measure, one often uses either the $1$-norm
		\[|(\varSigma)|_{k-1,1,m}\ \coloneqq \ \sum_{\substack{W\subset \varSigma^{(0)}\\ |W|=m } } \mr{b}_{k-1}(\varSigma_{|W})\]
		or the $\infty$-norm
		\[|(\varSigma)|_{k-1,\infty}\ \coloneqq \ \max_{W\subset \varSigma^{(0)} } \mr{b}_{k-1}(\varSigma_{|W}).\]
		Following~\cite{MiglioreNagel} and using Theorem~\ref{mthm:gl}(1) (and its generalization Theorem~\ref{mthm:h}), it follows that $|(\varSigma)|_{k-1,1,m}$ is maximized by the Billera-Lee polytopes among all triangulated spheres with the same $f$-vector. This answers a question of Codenotti, Santos and Spreer \cite{CSS}.

		Moreover, it follows from~\cite{AdiprasitoTC} and Theorem~\ref{mthm:gl}(1) that for $k\le \nicefrac{d}{2}$, we have a tight bound
		\[|(\varSigma)|_{k-1,\infty}\ \le \ g_k(\varSigma),\]
		where $g_k$ denotes the $k$-th primitive Betti number of $\AR(\varSigma)
		$. The case of arbitrary $k$ is treated by exploiting Alexander duality in the sphere $\varSigma$. %Global combinatorial consequences of this are obtained if $g_k$ vanishes, for instance the generalized lower bound theorem of Murai-Nevo~\cite{MN} for general simplicial homology spheres.
			In addition, there exists a subset $\mc{E}\subset \varSigma^{(0)}$, \[|\mc{E}|\ \le\ ((k+1)g_{k+1}+(d+ 1-k)g_k)(\varSigma)\] such that $\mr{b}_{k-1}(\varSigma_{W\cup \mc{E}})=0$ for all $W\subset \varSigma^{(0)}$. Hence, nontrivial homology classes in $H_{k-1}(\varSigma_{W})$ are not only boundaries in $\varSigma$, but one can find them to be boundaries of chains with at most \[((k+1)g_{k+1}+(d+ 1-k)g_k)(\varSigma)\] additional vertices.
The above results still apply to measure the complexity of triangulated manifolds $\Mu$, but now measuring not the induced homology of subcomplexes, but the kernel of
			\[\widetilde{H}_{k-1}(\Mu_{|W})\ \longrightarrow\ \widetilde{H}_{k-1}(\Mu).\]
		
		\item \emph{Gr\"unbaum-Kalai-Sarkaria conjecture, see~\cite{Gbhd}:} 
		The probably most widely studied application is to a conjecture of Gr\"unbaum, Kalai and Sarkaria: 
		As Kalai observed in \cite{KDM}, Theorem~\ref{mthm:gl}(1) implies that if $\varDelta$ is a simplicial complex of dimension $d$ that allows a PL embedding into $\R^{2d}$ then 
		\begin{equation}\label{eq:grb}
		f_d(\varDelta)\ \le \ (d+2)f_{d-1}(\varDelta).
		\end{equation}	
		We will give a simpler, self-contained proof of this implication in Section~\ref{sec:GKS}. Kalai also observed that this bound is essentially sharp, only differing to constructions in an additive error depending only on $d$, see also Remark~\ref{rem:kkk}.
		
		The case $d=1$ (when $\varDelta$ is a simple graph) is a consequence of what is now called Euler's formula, postulated first by Descartes and Euler \cite{descartes, euler}, though the first formal proof by modern standards seems to be due to Legendre \cite{Legendre}. They did not however seem to make the connection to Inequality~\eqref{eq:grb}.
		
		However, already for $d=2$, only little improvement was known over the trivial bound of $f_2(\varDelta)\ \le \ f_0(\varDelta)f_{1}(\varDelta)$: Dey~\cite{Dey} proved that 
		\[f_d(\varDelta)\ \le \ Cf^{d+1-\frac{1}{3^{d}}}_{0}(\varDelta),\]
		a bound that stood essentially uncontested for 25 years. Parsa~\cite{Parsa} (see also \cite{Skopenkov}) recently improved this to \[f_d(\varDelta)\ \le \ Cf^{d+1-\frac{1}{3^{d-1}}}_{0}(\varDelta)\]
		but the topological techniques employed are seemingly limited in potential, see also Gundert's Diplomarbeit~\cite{Gundert} for an excellent survey. 
		
		Due to this, several algebraic techniques towards the Gr\"unbaum-Kalai-Sarkaria conjecture were proposed, most notably the proposal of Kalai and Sarkaria using algebraic shifting \cite[Section~5.2]{KalaiS}. Unfortunately, a proof of this program was elusive, though Theorem~\ref{mthm:gl}(1) vindicates their approach by verifying their program, as Kalai noticed in the same manuscript.
		
		Our Inequality~\eqref{eq:grb} implies the stronger and optimal
		\[f_d(\varDelta)\ \le \ (d+2) \binom{f_{0}(\varDelta)}{d}.\]
		This has various consequences in discrete geometry (see also \cite[Section 3]{UW}). For instance, if $\varGamma$ is a simplicial complex and $\varGamma \rightarrow \R^{2d}$ is a PL map, then we have for the number of pairwise intersections of $d$-simplices, the $d$-th crossing number $\mr{cr}_d(\varGamma)$, the inequality
		\[\mr{cr}_d(\varGamma)\ \ge\ \frac{f_d^{d+2}(\varGamma)}{(d+3)^{d+2}f_{d-1}^{d+1}(\varGamma)}\]
		if $f_d(\varGamma)> (d+3)f_{d-1}(\varGamma)$, obtaining therefore a  generalization of the influential crossing lemma of Ajtai, Chv{\'a}tal, Leighton, Newborn and Szemer{\'e}di~\cite{ajtai, leighton}. Compare this also Gundert's work on the subject \cite{gundert2013}, who relied on strong additional assumptions on the spectrum of the the Laplacian of $\varGamma$ and deep work of Fox, Gromov, Lafforgue, Naor and Pach \cite{FGLNP}. 
		
\begin{rem}
We in fact prove the Gr\"unbaum-Kalai-Sarkaria conjecture for simplicial complexes $\varDelta$ which are embeddable as a subcomplex into a triangulated rational homology sphere of dimension~$2d$, and more generally embeddings as subcomplexes into~$2d$-dimensional triangulated manifolds $\Mu$, where we obtain a bound
		\[f_d(\varDelta)\ \le \ (d+2)f_{d-1}(\varDelta) + {\binom{2d+1}{d}} \mr{b}_d (\Mu).\]
This is a larger class of embeddings, but still not all, see for instance~\cite{DV, Rushing}. The case of arbitrary topological embeddings therefore remains open.
\end{rem}	
	\end{compactenum}

%\enlargethispage{10mm}
	
\subsection{Future directions}	Even more interesting, the generic Lefschetz theorem raises several questions.
	\begin{compactitem}[$\circ$]
		\item One can equivalently ask whether polyhedral spheres satisfy the analogue of the Lefschetz property as well, provided an appropriate construction of intersection cohomology for this case. Unfortunately, realizing the associated intersection cohomology in the abstract case seems to require extra ingredients.
		\item Are there analogues of the generic Lefschetz theorem in classical algebraic geometry? Is there a geometric proof? This could use theory of K\"ahler currents, but clearly requires more beyond that as Section~\ref{sec:counterexHall} demonstrates. Arbitrary algebraic cycles can be deformed in various ways, of course, but lose much of their properties in the transition. Does a generic fibre say anything about the Lefschetz property at the limit?
		\item There are other interesting spaces that could, generically, behave like algebraic varieties with respect to the Lefschetz theorems, for instance quotients of real reductive groups $G$ by a maximal compact $K$ and a "generic" cocompact $\varGamma$, offering a route to some of Venkatesh's conjectures in the area~\cite{Venkatesh}.
		\item Another possible area of interest and source of interesting deformations could be Hrushovski's theory of Frobenius automorphisms and difference varieties \cite{Hrushovski}, but for this a cohomology theory remains to be developed. See also \cite{giabicani} for recent developments concerning intersection theory of difference varieties.
		\item A crucial part of our approach lies in the fact that the underlying cohomology rings are commutative and generated in degree one. As we shall see, not all such Gorenstein rings satisfy the Lefschetz property, but perhaps generic ones? The answer is easily shown to be yes without qualification~\cite{Watanabe} but we are interested in fixing certain interesting invariants, such as the equivariant cohomology. On the other hand, Boij has constructed "generic" Gorenstein algebras that do not even have unimodal Betti numbers~\cite{Boij}. We lack an interesting conjecture in this area.
	\end{compactitem}
	
	The proof introduces several new ideas, and the development of techniques requires some space. To help the reader, we offer some intermediate results along the way to mark milestones of what we have learned up to that point. These avocations provided in this paper are weaker than the final results but offer some helpful camp and rejuvenation for the reader, and in some cases already offer substantial improvement on previous results.
	
	\textbf{Acknowledgements:} The research leading to these results has received funding from the European Research Council under the European
Union's Seventh Framework Programme ERC Grant agreement ERC StG 716424 - CASe and the Israel Science Foundation under ISF Grant 1050/16. 
The main ideas for this paper were obtained strolling in the nature walk of the Hebrew University of Jerusalem, the apothecary garden in Leipzig, and the National Institute for Forest Science in Seoul. 
	
Tremendous thanks to June Huh, Gil Kalai and David Kazhdan for delightful conversations and insightful comments. Thanks to Gil especially for bribing me with many coffees to explain critical steps, and to David for making me present the proof in the Kazhdan Sunday Seminar, whose attendants I thank for their patience and questions, which helped shape the polished version of this paper. Special thanks to June Huh for asking many questions that helped improve and shape the exposition of this paper. 

Special thanks also  to Bella Novik for many helpful and invaluable remarks, and to Johanna Steinmeyer who proofread early versions for typos, as well as providing some of the pictures. Thanks to Anna Gundert for making her thesis and survey on the Gr\"unbaum-Kalai-Sarkaria conjecture publicly available. Thanks to Lou Billera, Ehud Hrushovski, Eric Katz, Satoshi Murai, Pavel Pat\'ak, Paco Santos, Martin Tancer and Shmuel Weinberger for enlightening conversations and helpful comments, and Corrado de Concini, Misha Gromov, Leonid Gurvits and Maxim Kontsevich for comments leading to this updated version. 	
	
%, and Lou Billera, Uli Wagner and Shmuel Weinberger for enlightening conversations.
	
% I am grateful to Shmuel Weinberger, Yuval Peled, Philippe Michel and Uli Wagner for invaluable conversations.
%Special thanks to Mikhail Burens, David Kazhdan, Gil Kalai, Eric Katz and June Huh for being sounding boards for the ideas, and Johanna Steinmeyer and Isabella Novik for proofreading early versions of this paper. 

	\section{Design of the paper}
	
	Generally, proofs of Lefschetz theorems are delayed to the very end of the paper, with the exception of Theorem~\ref{thm:vd}.
	
	Section~\ref{sec:basics} recalls some of the basics for the rings we use inspired from Chow rings of toric varieties, as well as some classical and new facts concerning these. Section~\ref{sec:Laman} will recall a basic classical case of a generic Lefschetz theorem, and its relations to an expansive numerical property of ideals in face rings. Section~\ref{sec:pairings} discusses a first general position property in the intersection ring that will accompany us throughout the paper, as well as an important observation: the Lefschetz property on hypersurfaces, which allows for a true induction on dimension, is related to a crucial genericity of the Poincar\'e pairing. 
	
	Section~\ref{sec:perturbation} shows that the general position condition in turn can be used to show the Lefschetz theorem, by introducing a crucial perturbation lemma. We apply it to prove the hard Lefschetz theorem for derived subdivisions of shellable spheres, or more generally the class of $L$-decomposable spheres, in a slight variation on the concept of vertex decomposable spheres defined by Billera and Provan. Up to this point, our reasoning is purely algebraic and combinatorial. 

Section~\ref{sec:railway} provides the critical idea to translate this to general spheres and manifolds, initially for the middle Lefschetz property, the so called railway construction. The proof for PL spheres is wrapped up in Section~\ref{sec:induct} for the middle isomorphism, and Section~\ref{sec:beymid} for the general isomorphism and the general Hall-Laman relations.
	
Finally, we extend the Lefschetz theorem from intersection rings of spheres to arbitrary manifolds in Section~\ref{sec:mani}, and discuss the case of homology manifolds. This is irrelevant to the combinatorial applications we outlined, but may be of relevance in future applications. For cleanness of the presentation, we choose to delay this aspect to the end. %The appendix finally gives an alternative proof of the generic hard Lefschetz theorem for homology spheres and manifolds by appealing to handlebody decompositions of cobordisms.
	
		\setcounter{tocdepth}{1}
	\tableofcontents
	
	\section{Basic Notions}\label{sec:basics}
	
	We set up some of the basic objects, and refer to~\cite{Lee, Stanley96} for a more comprehensive introduction. While many of these notions have their origin in toric geometry, we immediately define them in a more general combinatorial setting, with the additional benefit of selfcontainment, and refer the interested reader to \cite{CoxLittleSchenck} for the more classical aspects.
	
\subsection{Face rings and stress spaces} 
If $\varDelta$ is an abstract simplicial complex on groundset $[n]\coloneqq \{1,\cdots,n\}$, let $I_\varDelta\coloneqq \langle \x^{\mbf{a}}:\ \supp(\mbf{a})\notin\varDelta\rangle$ denote the nonface ideal in $\R[\x]$, where $\R[\x]=\R[x_1,\cdots,x_n]$. Let $\R^\ast[\varDelta]\coloneqq \R[\x]/I_\varDelta$ denote the face ring of $\varDelta$. A collection of linear forms $\Theta=(\theta_1,\cdots,\theta_l)$ in the polynomial ring $\R[\x]$ is a \Defn{partial linear system of parameters} if \[\dim_{\mr{Krull}} \bigslant{\R^\ast[\varDelta]}{\Theta \R^\ast[\varDelta]}\ =\ \dim_{\mr{Krull}} \R^\ast[\varDelta]-l,\] for $\dim_{\mr{Krull}}$ the Krull dimension. If $l=\dim_{\mr{Krull}} \R^\ast[\varDelta] = \dim \varDelta +1$, then $\Theta$ is simply a \Defn{linear system of parameters}, and the corresponding quotient ${\R^\ast[\varDelta]}/{\Theta \R^\ast[\varDelta]}$ is called a \Defn{Artinian reduction} of $\R^\ast[\varDelta]$. 
	
	The dual action of $\R^\ast[\x]$ acting on itself by multiplication is the action of the polynomial ring on its dual $\R_\ast[\x]$ by partial differentials, where naturally every variable in a polynomial $p=p(\x)$ is replaced with a corresponding partial differential $p(\nicefrac{\mr{d}}{\mr{d}\x})$. 
	In details, we define the \Defn{stress space} of $\varDelta$ as 
	\[{\R_\ast}[\varDelta]\ \coloneqq \ \ker\left[I_\varDelta: \R[\x]\ \longrightarrow\ \R[\x]\right]\]
	and by minding also the linear system of parameters
	\begin{equation} \label{eq:xx}
	{\R_\ast}[\varDelta;\Theta]\ \coloneqq \ \ker\left[\Theta: {\R_\ast}[\varDelta]\longrightarrow {\R_\ast}[\x] \right].
	\end{equation}
	The elements of this space are called stresses; the spaces themselves are very familiar in toric geometry and coincide simply with the top homology groups of the Ishida complex~\cite{MR951199, Oda}, see Section~\ref{sec:ish}.
	
	Note that ${\R_\ast}[\varDelta;\Theta]$ and the face ring \[\R^{\ast}[\varDelta,\Theta]\ \coloneqq\ \R^\ast[\varDelta]/\Theta \R^\ast[\varDelta]\] are isomorphic as graded vector spaces, and are naturally dual to each other \cite{Weil}:
	 The dual $c(\nicefrac{\mr{d}}{\mr{d} \x})$ of a linear form $c=c(\x)$ acts on the linear stress space, and therefore defines an action \[\R^\ast[\varDelta;\Theta]\times {\R_\ast}[\varDelta;\Theta]\ \longrightarrow\ {\R_\ast}[\varDelta;\Theta].\] 
	 We will often call this the \Defn{Weil duality} to distinguish it from other forms of duality present in this paper.
	
	\subsection{The relative case} A \Defn{relative simplicial complex} $\RR=(\varDelta, \varGamma)$ is a pair of simplicial complexes $\varDelta, \varGamma$ with $\varGamma \subset \varDelta$.
	If $\RR=(\varDelta,\varGamma)$ is a relative simplicial complex, then we can define the \Defn{relative face module}
	\[\R^\ast[\RR]\ \coloneqq \ \bigslant{I_\varGamma}{I_\varDelta}
	\]
	and its reduction $\R^\ast[\RR]/\Theta \R^\ast[\RR]$.
	
It is not hard to modify this definition to obtain a definition of \Defn{relative stress spaces} by
\[{\R_\ast}[\RR]\ \coloneqq\ \coker\left[{\R_\ast}[\varGamma]\ \longrightarrow\ {\R_\ast}[\varDelta]\right]\]
	and
	\[{\R_\ast}[\RR;\Theta]\ \coloneqq \ \ker\left[\Theta:{\R_\ast}[\RR]\longrightarrow {\R_\ast}[\RR]\right].\]
	This will be occasionally useful. We obtain a pairing
	\[\R^\ast[\RR]\times {\R_\ast}[\RR]\ \longrightarrow\ {\R_\ast}[\varDelta].\]
	This restricts to a perfect pairing
	\[\R^k[\RR,\Theta]\times {\R_k}[\RR,\Theta]\ \longrightarrow\ {\R_0}[\varDelta]\ \cong\ \R.\]
	
	\subsection{Coordinates and properness} Observe finally that $\Theta$ induces a map $\varDelta^{(0)}\rightarrow \R^l$ by associating to the vertices of $\varDelta$ the coordinates $\V_\varDelta=({v_1},\cdots, {v_n}) \in \R^{l\times n}$, where $\V_\varDelta\x = \Theta$. Hence, as is standard to do when considering stress spaces, we identify a pair $(\varDelta; \Theta)$ with a \Defn{geometric simplicial complex}, that is, a simplicial complex with a map of the vertices to $\R^l$. The differentials given by $\V_\varDelta$ are therefore $\V_\varDelta\nabla$, where $\nabla$ is the gradient. 
	
	Conversely, the canonical stress spaces and reduced face rings, respectively, of a geometric simplicial complex are those induced by the linear system of parameters given by the geometric realization. 
	The stress space or face ring of a \Defn{geometric simplicial complex} is considered with respect to its natural system of parameters induced by the coordinates.
	
	A geometric simplicial complex in $\R^d$ is \Defn{proper} if the image of every $k$-face, with $k<d$, linearly spans a subspace of dimension $k+1$.  A sequence of linear forms is a (partial) linear system of parameters if the associated coordinatization is proper. For the results of our paper, we always think of every simplicial complex as geometric and proper, that is, as coming with a proper coordinatization in a vector space over~$\R$, and shall generally assume $(d-1)$-dimensional complexes to be realized in $\R^d$ unless otherwise stated, so that the associated collection of coordinatizing linear forms is a linear system of parameters.
	
	 We are most interested in the cases when the coordinates give rise to an Artinian reduction, that is, the coordinates induce a linear system of parameters, justifying our notation.
 All geometric simplicial complexes are therefore assumed to be proper. For geometric simplicial complexes, we shall use the notation
\[\AR_\ast(\RR)\ \coloneqq\ {\R_\ast}[\RR;\V_\RR \x] \quad \text{resp.}\quad \AR^\ast(\RR)\ \coloneqq\ {\R^\ast}[\RR;\V_\RR \x]\]

	\subsection{Minkowski weights} We close with a useful notion related to simplicial stresses: The space of squarefree coefficients of $\AR_\ast(\RR)$ is also called the space of \Defn{Minkowski weights}~\cite{FS}, denoted by~$\MW(\RR)$.
	\begin{prp}[cf.~\cite{Lee}]\label{prp:mw}
		Consider a relative simplicial $d$-complex $\RR$ in $\R^d$. Then
		\[\upvarrho:\AR_\ast(\RR)\ \longrightarrow\ \MW(\RR),\]
		the map restricting a stress to its squarefree terms, is injective (and therefore an isomorphism). Moreover, $\MW_k(\RR)$ is the space of weights on $(k-1)$-faces that satisfy the \Defn{Minkowski balancing condition}: for every $(k-2)$-face $\tau$ of $\RR$,
		\[\sum_{\substack{\sigma\in\RR^{(k-1)}\\ \sigma\supset\tau}} c(\sigma)v_{\sigma{\setminus}\tau}\ =\ 0 \mod \mr{span}\,( \tau),\]
		that is, the sum on the left-hand side lies in the linear span of $\tau$.
	\end{prp}
	Here, we use $\AR_\ast(\RR)$ to denote the stress space of the geometric simplicial complex $\RR$ as convened upon, and $\RR^{(i)}$ to denote the $i$-dimensional faces of $\RR$.
The action of the face ring on stress spaces extends to Minkowski weights.
	
	\subsection{The cone lemmas} A crucial ingredient for the inductive structure is given by pullbacks to prime divisors.
	Recall that the  \Defn{star} and \Defn{link} of a face $\sigma$ in $\varDelta$ are
	the subcomplexes \[\St_\sigma \varDelta\ \coloneqq \ \{\tau:\exists \tau'\supset \tau,\ \sigma\subset
	\tau'\in \varDelta\}\ \
	\text{and}\ \ \Lk_\sigma \varDelta\ \coloneqq \ \{\tau{\setminus} \sigma: \sigma\subset
	\tau\in \varDelta\}.\]
	For geometric simplicial complexes $\varDelta$, we shall think of the star of a face as a geometric subcomplex of $\varDelta$, and the link of a face $\sigma$ as the geometric simplicial complex obtained by the orthogonal projection to $\mr{span}(\sigma)^\bot$.
	Let us denote the \Defn{deletion} of $\sigma$ by $\varDelta-\sigma$, the maximal subcomplex of $\varDelta$ that does not contain $\sigma$.
	Let 
	\[\St_\sigma^\circ \varDelta\ \coloneqq\  (\St_\sigma \varDelta,\St_\sigma \varDelta-\sigma).\]
	We have the following two elementary lemmas.
	
	\begin{lem}[Cone lemma I, see {\cite[Thm. 7]{Lee}}]
		For any vertex $v\in \varDelta$, where $\varDelta$ is a geometric simplicial complex in $\R^d$, and any integer $k$, we have an isomorphism
		\[\AR_k({\Lk}_v \varDelta)\ \cong\ \AR_{k}(\St_v \varDelta).\]
	\end{lem}
		
	\begin{lem}[Cone lemma II, see {\cite[Lem. 3.3]{AdiprasitoTC}}]
		In the situation of the first cone lemma we have a natural isomorphism
		\[x_v:\AR_{k+1}(\St_v^\circ \varDelta)\ \longrightarrow\ \AR_{k}(\St_v \varDelta).\]
	\end{lem}

	\subsection{Partition of unity and Poincar\'e duality}\label{sec:part}
	
	There is another fact that provides a good intuition to what we do, hidden in the following derivation of Poincar\'e duality:
	
	\begin{lem}[Partition of unity]\label{lem:partyyyyy}
		Consider a Cohen--Macaulay $(d-1)$-complex $\varDelta$ in $\R^{d}$. Then, for every $k<d$, we have a surjection
		\[\bigoplus_{v\in \varDelta^{(0)}} \AR_k(\St_v \varDelta)\ \longtwoheadrightarrow\ \AR_k(\varDelta).\]
	\end{lem}
Recall that  $\cdot^{(i)}$ denotes the set of $i$-dimensional faces of a simplicial complex.

	As a general convention, all our simplicial complexes (unless they are relative) have the empty set as a face, so that homology and cohomology of simplicial complexes is naturally always reduced.
	
	\begin{proof}
		We may equivalently prove an injection
		\begin{equation}\label{eq:cohomology} \AR^k(\varDelta)\ \longhookrightarrow\ \bigoplus_{v\in \varDelta^{(0)}} \AR^k(\St_v \varDelta)\end{equation}
		for $k< d$.
		
		Let $\Theta$ denote the coordinate ideal of parameters for $\varDelta$. For the proof, we consider the Koszul complex $K^\bullet\coloneqq K^\bullet(\Theta)$
		and the chain complex $\widetilde{\mc{P}}^\bullet=\widetilde{\mc{P}}^\bullet(\varDelta)$ defined as
		\[0\ \longrightarrow\ \R^\ast[\varDelta]\ \longrightarrow\ \bigoplus_{v\in \varDelta^{(0)}} \R^\ast[\St_v \varDelta]\ \longrightarrow\ \cdots\ \longrightarrow\  \bigoplus_{F \in \varDelta^{(d-1)}} 
		\R^\ast[\St_F \varDelta] \  \longrightarrow\ 0\]
with a natural choice of maps for the boundary operator. Indeed, this choice of map can be found easily once we realize that the degree zero component is naturally the \u{C}ech complex of $\varDelta$ covered by $\st_v  \varDelta$. With this, we obtain a complex as desired, and the complex is exact in positive degree.
		Let ${\mc{P}}^\bullet={\mc{P}}^\bullet(\varDelta)$ denote the complex
		\[0\ \longrightarrow\ \AR^\ast({\varDelta})\ \longrightarrow\ \bigoplus_{v\in \varDelta^{(0)}} \AR^\ast(\St_v {\varDelta})\ \longrightarrow\ \cdots\ \longrightarrow\  \bigoplus_{F \in \varDelta^{(d-1)}} 
		\AR^\ast(\St_F {\varDelta}) \  \longrightarrow\ 0.\]
		Considering the double complex $K^\bullet \otimes \widetilde{\mc{P}}^\bullet$ and computing the total complex using Reisner's theorem that the realization induces a regular linear system of parameters (see~\cite{Stanley96}), we get
		\[H^k(\mc{P}^\bullet)\ \cong\ H^k(\mathrm{tot}(K^\bullet \otimes \widetilde{\mc{P}}^\bullet)) \ \cong \ (H^{d-1})^{\binom{d}{k}}(\varDelta)(-k),\] where $(+j)$ denotes a shift in degree by $j$, so that $H^{d}(\mc{P}^\bullet)\cong(H^{d-1}(\varDelta))(-d)$, which is nontrivial only in degree $d$.
		But the degree $k$ component of
		$H^{d}(\mc{P}^\bullet)$
		is isomorphic to the kernel of the Map~\eqref{eq:cohomology}, as desired.
	\end{proof}
	
	An important corollary (and equivalent restatement for spheres) of this fact is Poincar\'e duality for simplicial homology spheres, proved initially by Gr\"abe~\cite{Grabe} using somewhat deeper techniques from commutative algebra; note though that the preceding statement is more general in that it applies beyond the case of spheres to Cohen--Macaulay complexes that do not have a fundamental class.

	\begin{thm}\label{thm:pd}
		Let $\varSigma$ be a  $(d-1)$-dimensional simplicial homology sphere in~$\R^d$. Then $\AR^\ast(\varSigma)$ is a Poincar\'e duality algebra.
	\end{thm}

 For this, we first notice: \begin{lem}\label{lem:cl}
		In any closed manifold $\Mu$ of dimension $d-1$ in $\R^d$, we have an injection
		\[\AR^\ast(\st_v^\circ \Mu)\ \longhookrightarrow\  \AR^\ast(\Mu) \]
		for any vertex $v$ of $\Mu$. 
\end{lem}
\begin{proof}
This follows immediately from Schenzel's calculations
\cite[Section 4.1 and Lemma 6.3.4]{Schenzel82} and a basic extension to the relative case made explicit in~\cite[Proposition 4.2 and Theorem 4.8]{Adiprasitorelative}). Apply the last two results directly to the short exact sequence
		\[0\ \longrightarrow\ \R^\ast[\st_v^\circ \Mu]\ \longrightarrow\ \R^\ast[\Mu]\ \longrightarrow\ \R^\ast[\Mu-v]\ \longrightarrow\ 0,\]
computing the Artinian reduction using the above facts and by exploiting that the inclusion of manifolds \[\Mu-v\ \longhookrightarrow\ \Mu\] induces an surjection in cohomology
\[H^\ast(\Mu)\ \longtwoheadrightarrow\ H^\ast(\Mu-v). \qedhere\]
		\end{proof}		
	
	\begin{proof}[\textbf{Proof of Theorem~\ref{thm:pd}}]
		It suffices to observe that the annihilator of the irrelevant ideal in $\AR^\ast(\varSigma)$ is concentrated in degree $d$.
		 But then partition of unity implies for all $k<d$ an injection 
		\[\AR^k(\varSigma)\ \longhookrightarrow\ \bigoplus_{v\in \varSigma^{(0)}} \AR^k(\St_v \varSigma)\ \cong\ \bigoplus_{v\in \varSigma^{(0)}} \AR^{k+1}(\St_v^\circ \varSigma)\ \longhookrightarrow \bigoplus_{v\in \Mu^{(0)}}\AR^{k+1}(\varSigma) \]
		where the middle map is induced by the second cone lemma for each vertex $v$ of $\varSigma$.
	\end{proof}
	
	\subsection{Manifolds and socles}\label{sec:part2}
	
We will simply write manifolds to mean orientable (over~$\mathbb{Q}$), closed, connected combinatorial (also known as PL) manifolds.
	 If an object satisfies all these conditions except closedness, we will call it a manifold with boundary. Sometimes we will also consider rational homology manifolds, and make clear that the result extends immediately to this case. We will denote this by the adjective "rational".
	
	  For instance, rational spheres will be rational homology manifolds that are rational homology equivalent to the standard spheres.	  
% A combinatorial sphere is simply a combinatorial manifold homeomorphic to the sphere. 
 A combinatorial rational homology sphere is a combinatorial manifold rationally homology equivalent to the standard sphere.
 
   A PL sphere finally is a sphere PL homeomorphic to the standard sphere. This is the classical definition that may be more familiar. For us, it is enough to ensure that the link of every face is homeomorphic to a sphere of the appropriate dimension. Modulo the smooth Poincar\'e conjecture in dimension $4$, these two notions are equivalent, but for us either notion works. Balls are distinguished analogously.

	For triangulated rational spheres, we will be interested mostly in the intersection rings~$\AR^\ast$, as the above section shows that this results in Poincar\'e duality algebras already. For rational manifolds, we will have to go beyond that because of our desire to work strongly with the Poincar\'e pairing. For a simplicial $(d-1)$-manifold $\Mu$ that is realized in $\R^d$, \[\Soc (\Mu)\subset \AR^\ast(\Mu)\] shall denote the annihilator of the irrelevant ideal in $\AR^\ast(\Mu)$, also known as the \Defn{socle} of~$\AR^\ast(\Mu)$. If $\Mu$ is closed, orientable and connected, then $\Soc^d$ is generated by one element, the fundamental class. We will therefore denote by $\Socl(\Mu)$ the restriction of $\Soc(\Mu)$ to degrees $\le d-1$. We call this the \Defn{interior socle}, or the space of \Defn{mobile cycles}. Note that with this definition, we have immediately
\begin{lem}\label{lem:cl2}
For closed $(d-1)$-manifolds $\Mu$ in $\R^d$	
we have
\begin{equation*}
\Soc(M)\ \cong\  H^{d}(\mc{P}^\bullet(\Mu))
\end{equation*}
\end{lem}

\begin{proof}
This is immediate from Lemma~\ref{lem:cl}. 
\end{proof}

We have the following useful observation of Novik and Swartz:
	
	\begin{prp}[\cite{NS2}]\label{ns}
		For a triangulated $(d-1)$-dimensional rational manifold $\Mu$  in~$\R^d$, 
		\[\bigslant{\AR^\ast(\Mu)}{\Socl(\Mu)}\]
		is a Poincar\'e duality algebra. Equivalently, it satisfies partition of unity as stated above, that is, 
		\[\bigslant{\AR^\ast(\Mu)}{\Socl(\Mu)}\ \longhookrightarrow\ \bigoplus_{v\in \varDelta^{(0)}} \AR^\ast(\st_v\Mu) .\]
	\end{prp}
	
	While this fact on its own is interesting, trying to prove partition of unity for ${\AR^\ast(\Mu)}$ reveals an even more remarkable fact, as it reveals that the socles can be naturally identified using the spectral sequence coming from the two filtrations of the double complex in the proof of Lemma~\ref{lem:partyyyyy} above. Computing the second page of both filtrations, and using the fact that stars of nontrivial faces are Cohen-Macaulay, we obtain:
	\begin{prp}[\cite{NS}]\label{prp:part}
		For a triangulated $(d-1)$-dimensional rational manifold $\Mu$ (possibly with boundary) in $\R^d$, we have an exact sequence
		\begin{equation}\label{eq:homol}
		0\ \longrightarrow\ (H^{k-1})^{\binom{d}{k}}(\Mu)\ \longrightarrow \AR^k(\Mu)\ \longrightarrow\ \bigoplus_{v\in \Mu^{(0)}} \AR^k(\St_v \Mu).
		\end{equation}
		More generally, the degree $(k-i)$-component of 
		$H^{d-i}(\mc{P}^\bullet(\Mu))$
		is isomorphic to $(H^{k-1})^{\binom{d}{k-i}}(\Mu).$
	\end{prp}
We refer to Novik-Swartz~\cite{NS} for a derivation of Proposition~\ref{prp:part} using more commutative algebra (though their argument needs a minor adaption for manifolds with boundary). Historically, the use of the partition complex, and basic versions of Proposition~\ref{prp:part}, can be traced back to Kalai~\cite{KalaiRig}, see also \cite[Section~2.4]{Gromov}, though of course if we see it in context of \u{C}ech complexes the history of course becomes much richer. 
	
	\subsection{The combinatorial Ishida complex and the map $\mathrm{hom}$:}\label{sec:ish}
	It is possible to refine the understanding of this isomorphism as follows: we note that the map is given in terms of an elementary map of weighted simplicial homology.
	
	 To this end, it is useful to recall the combinatorial Ishida complex of Tay and Whiteley~\cite{Ishida, TW}. For a vector space~$V$, let us denote the $j$-th exterior power by $V^{[j]}$. If $\sigma$ is an $i$-tensor, then multiplication with it in the exterior algebra induces a natural map \[V^{[j]}\rightarrow V^{[j+i]}.\] Consider a geometric complex $\varDelta$ in $V$, and a non-negative integer~$k\le \dim V$, we then define the stress complex (or combinatorial Ishida complex) $I_\bullet[k]$ as
	\[0\ \rightarrow\ \bigoplus_{\sigma \in \varDelta^{(k-1)}} \bigslant{V^{[0]}}{\ker \sigma}\ \longrightarrow\ \bigoplus_{\tau \in \varDelta^{(k-2)}} \bigslant{V^{[1]}}{\ker \tau}\ \rightarrow\ \cdots\]
	where we identify $\sigma$ with its $|\sigma|$-tensor and
	\[\partial c[\sigma]\ = \ \sum_{x\in \sigma^{(0)}} cx [\sigma{\setminus}x].\]
	The top cohomology of this complex is isomorphic to $\AR_k(\varDelta)$~\cite{Oda}. Moreover, this homology provides a natural interpretation of stresses as simplicial cycles.
	For instance, we obtain immediately that for a complex $\varDelta$ in $\R^k$, the complex $I_\bullet[k]$ naturally defines a simplicial homology theory, with a boundary operator perturbed by a weighting from the standard one, so we have a natural map
	\[\mr{hom}:\AR_k(\varDelta)\ = \ Z_{k-1}(\varDelta;\R)\ \longrightarrow\ H_{k-1}(\varDelta;\R).\]
	Naturally, if $\varDelta$ is realized in $\R^{[d]}$, then we obtain a natural map as composition
	\begin{equation}\label{eq:maphom}
	\mr{hom}:\AR_k(\varDelta)\ \longhookrightarrow \ (Z_{k-1})^{\binom{[d]}{k}}(\varDelta;\R)\ \longrightarrow\ (H_{k-1})^{\binom{[d]}{k}}(\varDelta;\R).
	\end{equation}
	This is indeed the map inducing Isomorphism~\eqref{eq:homol}, as
	the cokernel of \[\bigoplus_{v\in \varDelta^{(0)}} \AR_i(\St_v \varDelta)\ \longtwoheadrightarrow\ \AR_i(\varDelta)\] is naturally the cokernel associated to the \u{C}ech filtration of $\varDelta$ induced by stars of vertices. We will later see that it is useful to pay close attention to this map.

	As a general rule, we will use \[\BR^{\ast}(\Mu)\ \coloneqq\  \bigslant{\AR^{\ast}(\Mu)}{\kerc\left[\AR^\ast(\Mu)\ \rightarrow\ \bigoplus_{v\in \Mu^{(0)}} \AR^\ast(\st_v \Mu) \right]}\]
	to denote the reduction of $\AR^\ast(\Mu)$ of a manifold $\Mu$ (possibly with boundary) to the associated Gorenstein ring,
	where
\[\kerc\left[\AR^\ast(\Mu)\ \rightarrow\ \bigoplus_{v\in \Mu^{(0)}} \AR^\ast(\st_v \Mu) \right]\ =\ (H^{d})^\circ(\mc{P}^\bullet(\Mu))	\]
is the restriction of	$H^{d}(\mc{P}^\bullet(\Mu))$, to elements of degree at most $\dim \Mu$.

With this notation, we obtain also immediately the following version of Lemma~\ref{lem:cl}.
\begin{cor}\label{lem:cl3}
		In any closed manifold $\Mu$ of dimension $d-1$ in $\R^d$, we have an injection
		\[\AR^\ast(\st_v^\circ \Mu)\ \longhookrightarrow\  \BR^\ast(\Mu) \]
		for any vertex $v$ of $\Mu$. 
\end{cor}

Hence, we also obtain a new proof of Propositon~\ref{ns}.
\begin{proof}
Consider the following commutative diagram:
\[\begin{tikzcd}
0 \arrow{r}{} &(H^{d})^\circ(\mc{P}^\bullet(\Mu-v)) \arrow{r}{}  & \AR^{\ast}(\Mu-v) \arrow{r}{}& \BR^{\ast}(\Mu-v) \arrow{r}{}& 0 \\
0 \arrow{r}{} &(H^{d})^\circ(\mc{P}^\bullet(\Mu)) \arrow{r}{} \arrow{u}{\wr} & \AR^{\ast}(\Mu) \arrow{r}{}\arrow{u}{}& \BR^{\ast}(\Mu) \arrow{r}{}\arrow{u}{}& 0 \\
 &0 \arrow{r}{} \arrow{u}{}&  \AR^{\ast}(\st_v^\circ \Mu) \arrow{r}{}\arrow{u}{} &  \AR^{\ast}(\st_v^\circ \Mu) \arrow{r}{} \arrow{u}{}& 0 
\end{tikzcd}
\]
where we use Proposition~\ref{prp:part} for the horizontal sequences and the fact that the left top horizontal map is an isomorphism.
\end{proof}	
	 
We denote by $\BR_\ast(\Mu)$ the corresponding subspace of the stress space~$\AR_\ast(\Mu)$. Note that for closed manifolds, Lemma~\ref{lem:cl2} implies that $\BR^{\ast}(\Mu)$ is the quotient of $\AR^{\ast}(\Mu)$ by its interior socle.
	 
	  We will generally define either define the Weil primal or the Weil dual of a space if the definition of the other is clear, and use both.
	 
	\section{The Gr\"unbaum-Kalai-Sarkaria conjecture and Laman's theorem}\label{sec:Laman}
	
	The Gr\"unbaum-Kalai-Sarkaria conjecture is not only solved by our work, but it also provides critical intuition towards our research, which is why we reserve a special section towards it here.
	
	\subsection{Laman's theorem}
	We discuss now the "original" generic Lefschetz theorem, known as Laman's rigidity criterion for planar graphs. Recall that a \Defn{Laman graph} is a graph $G$ on $n$ vertices and $2n-3$ edges such that every subgraph on $x$ vertices has at most $2x-3$ edges.
	
	\begin{thm}[Laman,~\cite{Laman}]
		Consider a Laman graph $G$. Then a generic realization of $G$ in $\R^2$ with straight edges is infinitesimally rigid, that is, every infinitesimal deformation of $G$ that preserves edge-lengths in the first order is a restriction of a global isometry.
	\end{thm} 
	
	Let us restate Laman's theorem in another form more familiar to us:
	
	\begin{thm}\label{thm:laman}
		For any generic embedding of a Laman graph $G$ into $\R^2$, we have an isomorphism
		\[\AR^1(G)\ \xrightarrow{\ \cdot \ell\ }\ \AR^2(G),\]
		where $\ell$ denotes any generic element of $\AR^1(G)$.
		Conversely, any graph on $v$ vertices and $2v-3$ edges admits such an embedding if and only if it is a Laman graph.
	\end{thm}
	
In the latter form, Laman's rigidity theorem looks suspiciously like a Lefschetz theorem. Its proof follows a very nice idea; we call it the "decomposition principle", named after the decomposition theorem for maps between algebraic varieties \cite{BBDG}.
	
	\begin{proof}[{Proof Sketch}]
		Every Laman graph can be shown to be obtained from a single edge by the Henneberg moves.
		\begin{figure}[h!tb]
			\begin{center}
				\includegraphics[scale = 0.7]{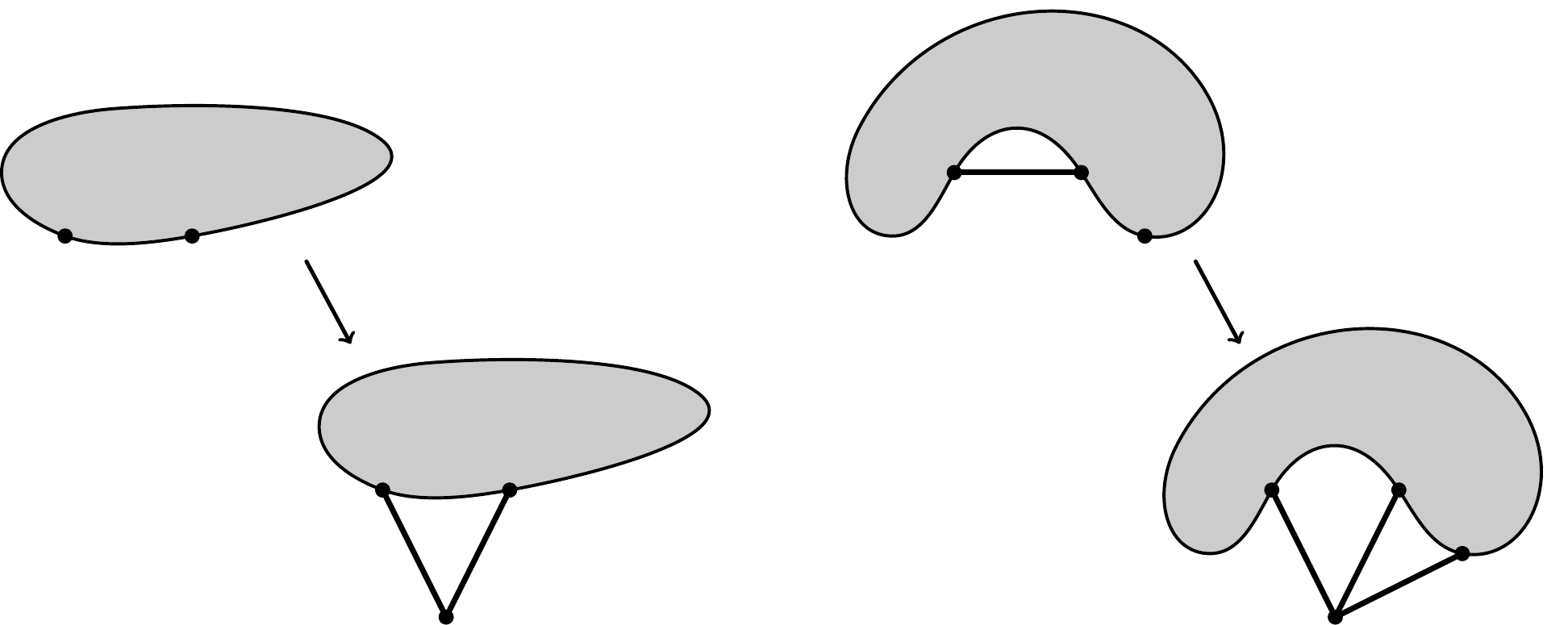}
				\caption{The two Henneberg moves.}
				\label{henneberg}
			\end{center}
		\end{figure}
		These moves preserve the Lefschetz theorem, as seen by an appropriate form of the decomposition theorem: one constructs an embedding of algebras
		\[\AR^\ast(G)\ \longhookrightarrow\ \AR^\ast(G')\]
 encoding the Henneberg move of $G$ to $G'$, give an appropriate basis for the cokernel and prove the Lefschetz theorem on the decomposition, and in particular conclude rigidity for $G'$ from rigidity for $G$.
	\end{proof}

\subsection{Limits of refinement and decomposition}\label{sec:local-move}

While the Laman theorem is a very nice generic Lefschetz theorem, it turns out that that it relies on a principle that is unfortunately of limited use in proving the Lefschetz theorem in general.

This program is also the the easier part of the classical deformation approach that also McMullen used to prove the hard Lefschetz theorem for simple polytopes~\cite{McMullenInvent}: One constructs a local refinement move as, for instance, stellar subdivisions, Whiteley's edge splittings \cite{Whiteley} or the aforementioned Henneberg moves that descends to an embedding of algebras
	\[\AR^\ast\ \longhookrightarrow\ \widetilde{\AR}^\ast.\]
	By splitting the latter cleverly (using different forms of the decomposition theorem) one can then often conclude the Lefschetz theorem for the refinement from the Lefschetz theorem on the preimage. 
	
	On the other hand, one should not underestimate this approach. Indeed, the only cases of Theorem~\ref{mthm:gl}(1) beyond polytopes rely on spheres obtained from the simplex by local refinement moves. See for instance~\cite{MuraiHL}, although the use of this beautiful "decomposition principle" is obscured by the step of algebraic shifting, that is unnecessary for the proof but reduces the splitting to a more pleasant combinatorial one (though we lose sight of the Poincar\'e pairing). Additionally, much of the reasoning of \cite{AHK} can be simplified using this technique, as the proof of the Hodge-Riemann relations for matroids only ever uses the blowup part of the McMullen program.
	
	Unfortunately, these kinds of proofs are not useful for us as they can only treat blowups (and local similar moves implying an embedding of algebras as above), and not all spheres of dimension three and higher can be generated by such local refinement moves.  This is especially drastic in dimension five and higher, where such a result would make the decision problem of distinguishing an arbitrary triangulated five-manifold from a sphere treatable, which contradicts the work of Markov and Novikov.
	
	Indeed, even if we allowed local moves in general, it would force us to restrict to something like bistellar transformations~\cite{Pachner}, and in particular PL spheres, and we would need to incorporate blowdowns (and analogous moves). McMullen got around this issue by proving ingeniously that the Hodge-Riemann relations in lower dimensions imply the preservation of the hard Lefschetz theorem, meaning he only had to take care of the Hodge-Riemann relations in the blowdown, which now was comparatively simple as the hard Lefschetz guaranteed the Hodge-Riemann form did not degenerate in the deformation. But alas, we do not have the Hodge-Riemann relations without ample classes.
	
	Our main interest in Laman's theorem is different: instead of the numerical condition for the faces of simplicial complexes, we want to focus on an algebraic condition for ideals.
	
	\subsection{Numerical consequences of the Lefschetz theorem} 
	Let $\varSigma$ denote a simplicial (rational) sphere and let $\varDelta$ denote a subcomplex of $\varSigma$.
	We define
	\[\upkappa_i(\varSigma,\varDelta)\ \coloneqq \ \dim \KK_i(\varSigma,\varDelta),\]
	where
	\[\KK_\ast(\varSigma,\varDelta)\ \coloneqq \ \coker[\AR_\ast(\varDelta)\rightarrow \AR_\ast(\varSigma)].\] 
	Note that these are just the (duals of) squarefree monomial ideals generated by the monomials supported outside $\varDelta$ referenced in the Hall-Laman relations.
	
	The monotonicity property of the Laman theorem is a consequence of the Lefschetz isomorphism, and this has several applications we can distract ourselves with.
	
	\begin{obs}\label{obs:middle_monotonicity}
		Consider $\varSigma$ a $2k$-dimensional rational sphere in~$\R^{2k+1}$. If $\varSigma$ satisfies the  hard Lefschetz property, then 
		for every subcomplex $\varDelta$ of~$\varSigma$, we have 
		\begin{equation}\label{eq:mid}
		\upkappa_k(\varSigma,\varDelta)\ \le \upkappa_{k+1}(\varSigma,\varDelta) .
		\end{equation}
	\end{obs}

This is a trivial consequence of the commutative square
\[\begin{tikzcd}
 \AR_{k+1}(\varSigma) \arrow{r}{\ \cdot \ell\ } \arrow[two heads]{d}{} & \AR_{k}(\varSigma) \arrow[two heads]{d}{} \\
\KK_{k+1}(\varSigma,\varDelta) \arrow{r}{\ \cdot \ell\ } & \KK_k(\varSigma,\varDelta)
\end{tikzcd}
\]
so that an isomorphism (and in particular surjection) of the top map implies a surjection of the bottom map.
	
	\subsection{Avocation: A Hall type problem and expansion of ideals}\label{sec:counterexHall}
	This is certainly a more appropriate algebraic version of the Laman condition, as it takes into account the ring instead of the face number. The main objective could be to show that the converse holds as well, that is, to show that if Inequality~\eqref{eq:mid} holds for all $\varDelta \subset \varSigma$, then the Lefschetz property holds for $\varSigma$, that is,
	\[
	\AR^k(\varSigma)\ \xrightarrow{\ \cdot\ell\ } \AR^{k+1}(\varSigma)
	\]
	is an isomorphism for some $\ell$ in $\AR^1$. This would be a fantastic full realization of Hall's marriage theorem~\cite{Hall} for intersection rings of toric varieties (that is, expansion implies existence of an injective morphism). Let us recall the classical version:
	\begin{thm}[Hall's marriage theorem]
		Assume we have two finite sets $X$ and $Y$, and a relation between them so that every subset $A$ of $X$ has at least $|A|$ elements in $Y$ related to it. Then
		the relation supports an injective morphism.
	\end{thm}
	It is natural to see Laman's theorem in this context: it postulates the expansion of ideals, that is, it posits that any monomial ideal in a Laman graph is larger (with respect to dimension) in degree two than it is in degree one. One could hope for such a theorem generally in toric geometry. Alas, that does not seem to be true, though we have no explicit counterexample for this special case. However, it is not true that expansion of ideals in commutative graded rings implies the existence of an injective map as the following example found by David Kazhdan (personal communication) shows.
	
	\begin{ex}
		Let $e_i$ be the standard basis of $X = \R^n$, 
		\[\begin{array}{lccc}
		\upalpha:& \R^n&\ \longrightarrow\ & \R^n \\
		& \upalpha(e_i )&\ {{\xmapsto{\ \ \ \ }}}\ & (-1)^i e_i	\ .
		\end{array}
		\]
		Let \[\uppi\hspace{0.3mm}: X^{[2]} \rightarrow Y\] be the quotient by the span of $e_i \wedge e_j$ where $i - j$ are even and
		\[\updelta {:}\ X\ \times\ X\ \longrightarrow\ Y\] be given by 
		\[\updelta(x' ,x'')\ =\ \uppi\hspace{0.3mm}(x' \wedge \upalpha x'' ).\] The map $\updelta$ is symmetric, therefore defining a multiplication on a commutative graded ring generated in degree one, multiplying two degree one elements by applying $\updelta$. Moreover, for $X'\subset X$ a subspace, the subspace $\updelta(X',X') \subset Y$ is of a larger dimension than $X'$ for $n$ large enough. However, 
		\[\updelta(x,\upalpha x)\ \equiv\ 0,\]
		so the map is never an injection.
	\end{ex}
	
	We will repair this by replacing the expansion condition by a more robust condition on the pairing in the next section.
	
	\subsection{Avocation: Bad Artinian reduction}
	
	One could rightfully ask in the situation of Theorem~\ref{mthm:gl} whether it is necessary to ask for genericity of the Artinian reduction, or whether just any linear system of parameters is enough (for an appropriate Lefschetz element).
	
	Observation~\ref{obs:middle_monotonicity} allows us to answer also this question in the negative. Recall: If $\varDelta$ is a simplicial complex, and $\sigma$ is a face in $\varDelta$ then a \Defn{stellar subdivision} of $\varDelta$ at $\sigma$ is the simplicial complex
	\[\varDelta{\uparrow}\sigma \ \coloneqq \ (\varDelta-\sigma)\ \cup\ \bigcup_{\tau \in \partial \st_\sigma \varDelta} (\{v_\sigma\}\ast \tau)
	\]
	where $v_\sigma$ is the new vertex introduced in place of $\sigma$, and $\ast$ is used to denote the free join of two simplices. If $\varDelta$ is geometric, then $v_\sigma$ will be chosen to lie in the linear span of $\sigma$.
	
	\begin{ex}
		Consider $\varSigma$ the boundary of the tetrahedron, and denote its  vertices by $1,2,3$ and $4$.
		
		Perform stellar subdivisions at every triangle of $\varSigma$, call the newly created vertices $1',2',3'$ and $4'$. Call the resulting complex $\varSigma'$.
		
		Realize the vertices in $\R^{[3]}$ as follows: place the vertices $1,2,3$ and $4$ in general position in~$\R^{[2]}\subset \R^{[3]}$. Place the remaining vertices in $\R^{[3]}{\setminus} \R^{[2]}$. The associated linear system is a linear system of parameters for $\R[\varSigma']$, and therefore $\AR^\ast(\varSigma')$ is a Poincar\'e duality algebra.
		
		However, if we consider the subcomplex $\varDelta=\varSigma \cap \varSigma'$,
		the quotient $\KK_\ast(\varSigma',\varDelta)$ satisfies
		\[\upkappa_1(\varSigma',\varDelta)\ =\ 3\ \quad \text{and}\ \quad \upkappa_2(\varSigma',\varDelta)=2\]
		which violates Observation~\ref{obs:middle_monotonicity}. Hence, the associated ring $\AR^\ast(\varSigma')$ cannot satisfy the Lefschetz theorem with respect to any Picard divisor.
	
	\begin{rem}Strictly speaking, we are no longer considering Picard divisors, as there are no varieties any more. Rather, we use this term as a shorthand for the group of degree one elements in our commutative rings.\end{rem}
	\end{ex}
	
	\subsection{The Gr\"unbaum-Kalai-Sarkaria conjecture}\label{sec:GKS}
	Before we briefly discuss the importance of Observation~\ref{obs:middle_monotonicity}, let us note a simple but important corollary.
	
	\begin{cor}\label{cor:grunbaum}
		If $\varSigma$ is a $2k$-dimensional rational sphere in $\R^{2k+1}$ that satisfies the hard Lefschetz property, then, for every subcomplex $\varDelta$ of $\varSigma$, we have Inequality~\ref{eq:grb}, that is,
		\begin{equation*}		
(k+2)\cdot f_{k-1}(\varDelta)\ \ge\ f_k(\varDelta).
		\end{equation*}
			\end{cor}
	
	\begin{proof}
		Observation~\ref{obs:middle_monotonicity} gives 
		\[\dim \AR_k(\varDelta) \ \ge \ \dim \AR_{k+1}(\varDelta).\]
		Moreover, we have the immediate inequality
		\[\dim \AR_{k}(\varDelta)\ \le\ f_{k-1}(\varDelta),\] 
		which follows by estimating the dimension of a vector space by the size of a generating set, in this case the $(k-1)$-faces and
		\[\dim \AR_{k+1}(\varDelta)\ \ge\ f_{k}(\varDelta) - (k+1)f_{k-1}(\varDelta),\]
		which follows by estimating the number of relations from above: By Proposition~\ref{prp:mw}, there are $k+1$ relations for each $(k-1)$-face, so 
		$(k+1)f_{k-1}(\varDelta)$ is an upper bound for the number of nontrivial relations.
	\end{proof}
	
	\begin{rem}\label{rem:kkk}
	Kalai showed that under the assumptions of the previous corollary, the shifting of $\varDelta$ does not contain the Flores complex $\binom{[2k+3]}{\le d+1}$ as a subcomplex, and is contained in turn in the shifting of the infinite cyclic $(2k+1)$-polytope \cite{KDM,KalaiS}. This allows one to prove Kruskal-Katona type theorems for complexes embeddable in spheres, and in particular sharp bounds for the Gr\"unbaum-Kalai-Sarkaria conjecture, the equality cases for which are attained by compressed collections of $k$-faces in cyclic $(2k+1)$-polytopes. In the end, this effort only improves the bounds a little, in an additive error depending on the dimension, so that we shall content ourselves with stating the simplified Bound~\eqref{eq:grb} here.
	
Finally, these bounds also apply to complexes $\varDelta$ embeddable as subcomplexes into general $2k$-dimensional triangulated rational manifolds $\Mu$, with a small caveat: depending on the $k$-th rational Betti number of $\Mu$, we have to allow for an additive error of $\binom{2k+1}{k}\mr{b}_k(\Mu)$. For instance, the inequality of Corollary~\ref{cor:grunbaum} modifies to 
		\[(k+2)\cdot f_{k-1}(\varDelta)\ +\ \binom{2k+1}{k}\mr{b}_k (\Mu) \ge\ f_k(\varDelta),\]
This extends to Kalai's sharpening via the Kruskal-Katona type bound discussed above, and follows from the inequality
		\[\dim \AR_k(\varDelta) \ \ge \ \dim \AR_{k+1}(\varDelta)-\binom{2k+1}{k}\mr{b}_k (\Mu)\]
		that in turn follows from biased Poincar\'e duality in $\Mu$ that is introduced in the next section. This implies at once another conjecture of K\"uhnel \cite{K2}: if a complete $k$-dimensional complex on $n$ vertices embeds into $\Mu$ sufficiently tamely (so that it extends to a triangulation of $\Mu$), then
		\[\binom{n-k-1}{k+1}\ \le\ \binom{2k+1}{k+1}\mr{b}_k (\Mu),\]
		thereby bounding the number of vertices such a complex can have.	
		\end{rem}
	
	Now, as we will prove the hard Lefschetz theorem (first for PL spheres, then for general manifolds in Section~\ref{sec:mani}), the conclusion of the corollary and the remark following it hold without qualification. However, it is useful to note that a simpler, but nonetheless central, property of the intersection ring already implies Inequality~\eqref{eq:mid}, and therefore also~\eqref{eq:grb}. This is our biased Poincar\'e duality we will examine in Section~\ref{sec:pairings}.
	
	\subsection{Stress subspaces and partitioned subspaces}\label{ssc:sop} Before we continue, it is useful to extend the definition of the modules $\KK$ to rational manifolds (possibly with boundary)~$\Mu$. For subcomplexes $\varDelta$ of $\Mu$, we therefore
	write \[\upkappa_i(\Mu,\varDelta)\ \coloneqq \ \dim \KK^i(\Mu,\varDelta),\]
	where
	\[\KK^\ast(\Mu,\varDelta)\ \coloneqq \ \ker[\BR^{\ast}(\Mu)\rightarrow \BR^{\ast}(\varDelta\subset\Mu)]\]
	and
	\[\BR^{\ast}(\varDelta\subset\Mu)\ \coloneqq\  \bigslant{\AR^{\ast}(\varDelta)}{\kerc\left[\AR^\ast(\Mu)\ \rightarrow\ \bigoplus_{v\in \Mu^{(0)}} \AR^\ast(\st_v \Mu) \right]}.\]
	We will also use the Weil dual spaces, denoted by $\KK_\ast$ and $\BR_\ast$ as quotients and subspaces of~$\AR_\ast$, respectively. 
	
	One reason we shall mostly work with stress groups is that kernels and images of envisioned Lefschetz elements (when they do not connect opposite degrees), enjoy nice descriptions. The kernels are ideally of the form $\BR_{\ast}(\varDelta\subset\Mu)$, called \Defn{stress space}.  The Poincar\'e duals of $\KK_\ast(\Mu,\varDelta)$, and in particular the images of Lefschetz maps as we shall see later, can be understood as well, as they are of the form 
	\[\BR_{\ast}(\Mu)_{|\overline{\varDelta}}\ \coloneqq \ \langle\AR_\ast(\st_\sigma \Mu):\ \sigma \in \Mu,\ \sigma \notin \varDelta \rangle\ \subset\ \BR_{\ast}(\Mu) ,\]
	that is, subspaces generated by partitionable classes in the sense of partition of unity. We will call such a subspace therefore a \Defn{partitioned subspace}. They will be rather hard to work with directly if $\varDelta$ is not a codimension zero manifold in $\Mu$, which will be one of the main obstacles to overcome in the section that follows. 
	
	\section{Biased Poincar\'e duality and the Hall-Laman relations}\label{sec:pairings}
	In this section, we introduce a crucial ingredient of the program by introducing a genericity property of the Poincar\'e pairing replacing Laman's criterion. We then proceed to apply this notion to prove the hard Lefschetz theorem for spheres that allow for "nice" decompositions. For later, the perturbation lemma and pairing properties hold special importance.
	
	\subsection{Poincar\'e pairings} Recall: Let $\Mu$ be a triangulated rational manifold of dimension $d-1$ realized in $\mathbb{R}^d$, and $\BR^\ast(\Mu)$ its face ring modulo Artinian reduction and mobile cycles. Then we have a pairing
	\[\BR^k(\Mu)\ \times\ \BR^{d-k}(\Mu)\ \longrightarrow\ \BR^{d}(\Mu)\ \cong\ \R.\]
	The essential observation we wish to make is that, while the signature of the pairing does not change under small perturbations, it is often convenient to have the possibility to make certain subspaces non-degenerate under the pairing. 
	
	\subsection{Biased Poincar\'e duality in ideals} We say that $\Mu$ satisfies \Defn{biased Poincar\'e duality} in degree $k\le \frac{d}{2}$ if for all nonempty subcomplexes $\varDelta$ of $\Mu$, the pairing
	\begin{equation}\label{eq:ospd}
	\KK^{k}(\Mu,\varDelta)\ \times\ \KK^{d-k}(\Mu,\varDelta)\ \longrightarrow\ \KK^{d}(\Mu,\varDelta)\ \cong\ \R 
	\end{equation}
	is nondegenerate on the first factor. Sometimes we will also say that a specific ideal $\mc{I}$ in a graded Poincar\'e duality algebra over $\R$ satisfies \Defn{biased Poincar\'e duality}, if 
	\begin{equation}\label{eq:ospdI}
	\mc{I}^{k}\ \times\ \mc{I}^{d-k}\ \longrightarrow\ \mc{I}^{d}\ \cong\ \R 
	\end{equation}
is nondegenerate in the first factor, where $d$ is the degree of the fundamental class. An important special case is biased Poincar\'e duality at a subcomplex $\varDelta$, synonymous with biased Poincar\'e duality for the ideal $\KK^\ast(\Mu,\varDelta)$.
	
	 Biased Poincar\'e duality is evidently a weaker form of the Hall-Laman relations, with equivalence in the case $k=d-k$. We keep it separate as on its own it already has interesting consequences, even without having to consider a Lefschetz theorem. The following is a direct consequence of Poincar\'e duality.
	
\begin{prp}
For an ideal $\mathcal{I}$ in $\BR^\ast(\Mu)$
the following are equivalent:
\begin{compactenum}[(1)]
\item The map
\[\mathcal{I}\ \longrightarrow\ \bigslant{\BR^\ast(\Mu)}{\ann_{\BR^\ast(\Mu)} \mathcal{I} }\] is an injection (in degree $k$).
\item $\mathcal{I}$ satisfies biased Poincar\'e duality (in degree $k$).
\end{compactenum}
\end{prp}	

We obtain immediately an instrumental way to prove biased Poincar\'e duality for monomial ideals.

\begin{cor}\label{cor:map}
$\KK^\ast(\Mu,\varDelta)$ satisfies biased Poincar\'e duality in degree $k$ if and only if
\[\KK^k(\Mu,\varDelta)\ \longrightarrow\ \BR^k(\Mu)_{\mid \overline{\varDelta}}\]
is injective.
\end{cor}
	
	\subsection{Avocation: Relation to the Gr\"unbaum-Kalai-Sarkaria conjecture}
	
	Assume that $\Mu$ is a rational manifold of dimension $2k$ realized in $\R^{2k+1}$ so that biased Poincar\'e duality holds with respect to a subcomplex $\varDelta$. Then Inequality~\eqref{eq:mid} follows, that is, we have
	\[ \upkappa_k(\Mu,\varDelta)\ \le\ \upkappa_{k+1}(\Mu,\varDelta) .\]
	In particular, the Gr\"unbaum-Kalai-Sarkaria conjecture holds for $\varDelta$.
	
	\subsection{General properties of biased Poincar\'e duality}
	
	To describe and prove biased Poincar\'e duality in general, we start with a first observation, the following persistence lemma.
	
	\begin{lem}\label{lem:persistence}
		Let $\Mu$ denote a rational $(d-1)$-manifold in $\R^d$, and $k<\frac{d}{2}$. Then biased Poincar\'e duality holds for $\Mu$ in degree $k$ if biased Poincar\'e duality holds for links of all vertices in $\Mu$ and in degree $k$.
	\end{lem}
	
	\begin{proof}
		By partition of unity in $\Mu$, any element of $\KK^{k}(\Mu,\varDelta)$ pairs with some prime divisor $x_v$; by biased Poincar\'e duality of $\Lk_v \varDelta$ in $\AR^{k}(\Lk_v \Mu) \cong x_v\BR^k(\Mu) $ the claim follows. 
	\end{proof}

	This allows us to prove biased Poincar\'e duality by restricting to the case of degree $k$ and $d=2k$.
	
	 Notice that in particular  biased Poincar\'e duality in degree $k$ for $2k$ manifolds, and hence the Gr\"unbaum-Kalai-Sarkaria conjecture and K\"uhnel conjectures for manifolds, only require the discussion of the property in spheres (see Remark~\ref{rem:kkk}). 
	
	We nevertheless stay with manifolds for the moment, and make an additional observation simplifying the proof of the property for $(2k-1)$-manifolds and in degree $k$. The first is a folklore geometric observation:
	
	\begin{prp}\label{prp:env}
		If $\varDelta$ is a $(k-1)$-dimensional subcomplex in a $(2k-1)$-dimensional combinatorial manifold $\Mu$ in $\R^{2k}$, then there exists a subdivision $\Mu'$ of $\Mu$, and a $2(k-1)$-dimensional hypersurface $S$ in $\Mu'$ containing $\varDelta$.
	\end{prp}

A notion of importance in this context is the \Defn{simplicial neighbourhood} $\mr{N}_\varGamma \varDelta$ of an induced subcomplex $\varGamma$ in a simplicial complex $\varDelta$, defined as \[\mr{N}_\varGamma \varDelta\ \coloneqq\ \bigcup_{w \in \varGamma^{(0)}} \St_w \varDelta.\] We call a simplicial neighbourhood \Defn{regular} if it is regular in the sense of PL topology~\cite{RourkeSanders-new}, that is, it is a PL mapping cylinder over the PL manifold given as the boundary of the simplicial neighborhood, mapped to $\varGamma$. We sometimes use the notion of regular neighbourhood without referring to a specific $\varGamma$, and in this case it will simply be a regular simply neighbourhood with respect to some~${\varGamma}$.
	
	\begin{proof}
	There exists a simplicial homeomorphism that embeds $\varDelta$ into the boundary of its regular neighbourhood in a sufficiently fine refinement of $\Mu$ by a classical general position argument. This map can be assumed to be facewise linear on $\varDelta$ by~\cite[Section I.2]{Bing}. A subdivision of $\Mu$ that realizes the regular neighbourhood therefore gives the desired subdivision.
	\end{proof}
	
	\begin{cor}\label{cor:envelope}
		In the situation of the previous lemma, there exists in a refinement $\Mu'$ of $\Mu$ a hypersurface with boundary ${E}_\varDelta$ so that \[\AR^k(\varDelta)\cong\AR^k({E}_\varDelta).\] We can assume that $\partial {E}_\varDelta$ is an induced subcomplex of~${E}_\varDelta$.
	\end{cor}
	
	\begin{proof}
	If ${E}_\varDelta$ supports a $k$-stress outside $\varDelta$ then it must be supported in a $(k-1)$-face outside $\varDelta$. Removing that face may affect the manifold property, but we can refine ${E}_\varDelta$ and $\Mu$ outside $\varDelta$ so that the star of that face is a compact ball in the interior of ${E}_\varDelta$, in which case its removal preserves the manifold property.		
	\end{proof}
	
	We shall need a stronger version of this corollary later, but we leave it here anyway because it serves to illustrate our strategy.
	
	We call such a complex, that is, a supercomplex ${E}$ of $\varDelta$ such that $\AR^k(\varDelta)\cong\AR^k({E})$ holds an \Defn{envelope} in degree $k$ for $\varDelta$. Passing from from a subcomplex of $\Mu$ to its envelope (in degree $k$) does not affect the biased Poincar\'e duality (in degree $k$) as in such case
\[\KK^k(\Mu,\varDelta)\ \cong\ \KK^k(\Mu,{E}_\varDelta).\]
	
	The second useful lemma shows a stability property of biased Poincar\'e duality under PL homeomorphism, and in particular subdivision, and hence justifies the preceding results.
	\begin{lem}\label{lem:PLinv}
		A PL homeomorphism $\upvarphi:\Mu\rightarrow \Mu'$ of rational manifolds $\Mu,\ \Mu'$ in $\R^d$ that restricts to the identity on a common subcomplex $\varDelta$ preserves biased Poincar\'e duality, that is, 
$\KK^\ast(\Mu,\varDelta)$ satisfies biased Poincar\'e duality (in degree $k$) if and only if $\KK^\ast(\Mu',\varDelta)$ does (in degree $k$).	
	\end{lem}

	\begin{proof}
The idea is to construct an isomorphism 
\[\overline{\KK}^\ast(\Mu,\varDelta)\ \cong\ \overline{\KK}^\ast(\Mu',\varDelta)\]
where $\overline{\KK}^\ast(\Mu,\varDelta)$ resp.\ $\overline{\KK}^\ast(\Mu',\varDelta)$ are the orthogonal complements of  ${\KK}^\ast(\Mu,\varDelta)$ resp.\ ${\KK}^\ast(\Mu',\varDelta)$.
	
	To this end, notice that we may define a PL homeomorphism more generally as any transformation obtained by stellar subdivisions and their inverses. Following Alexander's classical work (see~\cite{Licksurv} for a survey), this subsumes classical PL equivalences. Recall that a refinement of geometric simplicial complexes
	$\varDelta$ to a complex $\varDelta'$ induces a map
		\[\AR^\ast(\varDelta)\ \longrightarrow\ \AR^\ast(\varDelta').\]
		This is most easily seen by identifying the face ring with the ring of facewise polynomials (see for instance~\cite{Bapp, Brion}) and using the fact that facewise polynomials on $\varDelta$ define facewise polynomials on $\varDelta'$.

Now,	a stellar subdivision of $\Mu$ at a face $\sigma$, denoted by $\Mu{{{{\uparrow}}}} \sigma$, induces a pullback map
		\[\AR^\ast(\Mu)\ \longhookrightarrow\ \AR^\ast(\Mu{{{{\uparrow}}}} \sigma)\]
		descending to an embedding
		\[\BR^\ast(\Mu)\ \longhookrightarrow\ \BR^\ast(\Mu{{{{\uparrow}}}} \sigma).\]
		
		The splitting this induces in the Poincar\'e pairing is orthogonal (see for instance~\cite[Proposition 2.2]{Petersen}), and therefore does not affect biased Poincar\'e duality of $\KK(\Mu,\varDelta)$ if the subdivision is outside $\varDelta$, that is, we have $\sigma \notin \varDelta$.
	\end{proof}
	
	\begin{rem}
		In fact, one can prove rather easily that
		\[\BR^\ast(\Mu{{{{\uparrow}}}} \sigma)\ \cong\ \BR^\ast(\Mu)\oplus \bigoplus_{i=1}^{|\sigma|-1} x_{v_\sigma}^i \AR^\ast(\lk_\sigma \Mu).\]
		For toric varieties, this is but a special case of the decomposition theorem (see~\cite{BBDG}) and the proof in this generality is a straightforward adaption.
	\end{rem}
	
%	\begin{rem}
%		A more clever proof of the lemma can go as follows: while the PL homeomorphism may change the ideal $\KK(\Mu,\varDelta)$, it leaves its orthogonal complement invariant as witnessed by the invariance of $\AR_\ast(\varDelta)$.
%	\end{rem}

	\subsection{Characterization theorem for biased Poincar\'e duality: spherical envelopes}
	
	Consider now a rational manifold $\Mu$ of dimension $2k-1=d-1$ in $\R^d$; by Corollary~\ref{cor:envelope}, we can assume that $\varDelta$ comes with an envelope ${E}={E}_\varDelta$ that is a hypersurface with boundary. 
	
	We now gradually study such envelopes of increasing topological complexity, until we arrive in the penultimate criterion in Section~\ref{sec:doub}. The three sections are independent, and the hurried reader can skip ahead, but following the gradual thought process of the author may be instructive.
	
	\begin{prp}\label{prp:ospd}
		Assume ${E}$ is a rational sphere. Then $\Mu$ in $\R^d$ satisfies biased Poincar\'e duality in degree $k$ and with respect to $\varDelta$ if and only if
		\[\AR_k({E})\ =\ {0}.\]
	\end{prp}
	
	The proof is very simple, but illustrates a few important principles.
	
	\begin{proof}
		We have a short exact sequence
		\[0\ \longrightarrow\ \AR_k({E})\ \longrightarrow\ \BR_k(\Mu)\ \longrightarrow\ \KK_k(\Mu, {E})\ \longrightarrow\ 0\]
		so we see that \[\BR_k(\Mu)\ =\ \KK_k(\Mu, {E})\ =\ \BR_k(\Mu)_{\mid\overline{{E}}}\] in case the first space is trivial, so the claim simply follows by Poincar\'e duality.
		
			For the converse, we consider the exact sequence
		\[0\ \longrightarrow\	\BR_{k}(E\subset\Mu)\ \longrightarrow\ \BR_k(\Mu)\ \longrightarrow\  \KK_k(\Mu,E)\ \longrightarrow\ 0.
		\]
Notice that because
\[H_{k-1}(E)\ \longrightarrow\ H_{k-1}(\Mu)\]
is the trivial map,
we have
\[\BR_{k}(E\subset\Mu)\ \cong\ \AR_k({E})\]
by Proposition~\ref{prp:part}.
Hence, a nontrivial $\alpha \in \AR_k({E})$ maps to zero in homology of the compactification of $\Mu{\setminus} E$ as a manifold with boundary, so that then $\alpha$ is in the image of $\BR_k(\Mu)_{\mid\overline{{E}}}$, again by Proposition~\ref{prp:part}.
Hence \[\BR_k(\Mu)_{\mid\overline{{E}}}\ \longrightarrow\ \KK_k(\Mu, {E})\] has a nontrivial kernel as $\alpha=0$  in the latter space. Because $\BR_k(\Mu)_{\mid\overline{{E}}}$ and $\KK_k(\Mu, {E})$ are isomorphic as vectorspaces (they are dual under the Poincar\'e pairing) the map also has a nontrivial cokernel, which implies that $\Mu$ does not satisfy biased Poincar\'e duality with respect to $\varDelta$ by Corollary~\ref{cor:map}.
	\end{proof}
	
\begin{ex}\label{ex:smooth}
Consider the case $k=1$, and $\varSigma$ a sphere of dimension $1$, realized in $\R^2$. If $\varDelta$ is a $0$-dimensional sphere in $\varSigma$, then $\KK(\varSigma,\varDelta)$ satisfies biased Poincar\'e duality if and only if $\varDelta$ does not lie on a line through the origin in $\R^2$.

		\begin{figure}[h!tb]
			\begin{center}
				\includegraphics[scale = 1]{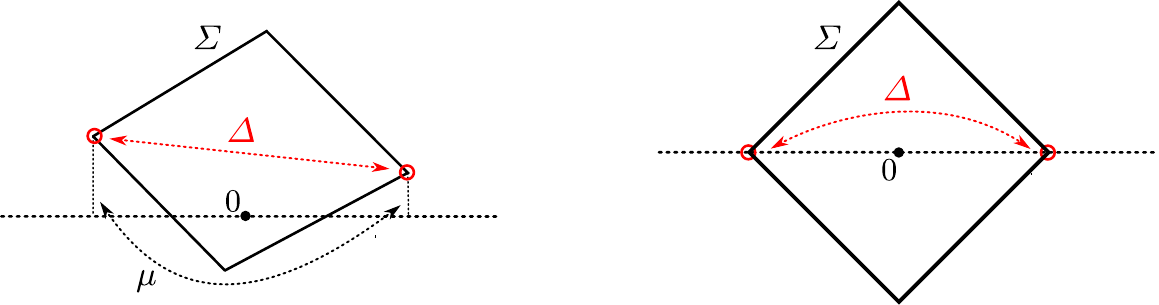}
				\caption{Biased Poincar\'e duality in a sphere $\varSigma$ and with respect to a codimension one sphere $\varDelta$ is related to the Lefschetz theorem on the latter.}
				\label{1sphere}
			\end{center}
		\end{figure}
	
On the other hand, this means that very nice and even smooth varieties do not satisfy biased Poincar\'e duality, for instance products of lower-dimensional varieties like the product of two projective lines $\mathbb{P}^1\times \mathbb{P}^1$, shown in Figure~\ref{1sphere} on the right.
\end{ex}	
	
\subsection{Characterization theorem for closed hypersurfaces}	
	We next want to deal with general closed rational hypersurfaces. Before we do that, we make a small intermission.

We are almost in the same situation as before: $\Mu$ is a rational manifold of dimension $2k-1=d-1$ in $\R^d$, but this time $E$ is a closed rational hypersurface in it that is not necessarily a sphere.
	
For simplicity, we shall assume such a hypersurface to be relatively $(k-2)$-acyclic, that is, 
	the map ${E} \hookrightarrow \Mu$ is injective in rational homology up to dimension $k-2$. We will see that this is sufficiently general for our purposes.
	
	 Indeed, following the proof of Proposition~\ref{prp:env}, this is a valid assumption we can make for combinatorial manifolds: every $\varDelta$ of dimension $k-1$ in $\Mu$ lies in a relatively $(k-2)$-acyclic closed hypersurface after subdivisions of $\Mu$ outside of $\varDelta$. 
	 
	 We shall also assume that the map ${E} \hookrightarrow \Mu$ is the trivial map in homology in dimension $d-2$, and shall treat the remaining case using the culminating methods of the next Section~\ref{sec:doub}.
		
Note that then $\Mu$ is parted into a finite number of compact simplicial manifolds $(\Mu_i)_{i\in I}$ such that the union of their boundaries is the hypersurface ${E}$. 
	
	Note that if $k>1$, then there are at most two of these rational manifolds. Note further that biased Poincar\'e duality with respect to ${E}$ is equivalent to biased Poincar\'e duality with respect to \[\overline{\Mu}_j\ =\ \bigcup_{i\in I{\setminus} \{j\} }\Mu_i\] for all $j\in I$ if $|I|\ge 2$ as the ideals $\KK(\Mu, \overline{\Mu}_j)$ are orthogonal to each other in the Poincar\'e pairing.
	
	We may therefore restrict to examining $\overline{\Mu}_1$ without loss of generality, and may assume that
	\[E\ =\ \partial \Mu_1\ =\ \partial \overline{\Mu}_1.\]
	
	\begin{thm}\label{thm:char_ospd}
		If $|I|\ge 2$, then $\Mu$ in $\R^d$ satisfies biased Poincar\'e duality with respect to $\overline{\Mu}_1$ and degree $k=\nicefrac{d}{2}$ as above if and only if the sequence
		\[0\ \longrightarrow\ \BR_k({E} \subset\Mu)\ \longrightarrow\ (H_{k-1})^{\binom{d}{k}}(\Mu_1)\ \longrightarrow\ (H_{k-1})^{\binom{d}{k}}(\Mu)\]
		is exact.
	\end{thm}
	
	The proof is a direct combination of Proposition~\ref{prp:part} and Corollary~\ref{cor:map}.
	
	\begin{proof}
		The theorem follows at once from the commutative diagram:
		\[\begin{tikzcd}
		&  & 0 \arrow{d}{} &  &\\
		&  & \BR_k({E}\subset\Mu)   \arrow{d}{} \arrow[dashed]{dr}{\mathrm{hom}} &  &\\
		0 \arrow{r}{} & \BR_k(\Mu_1) \arrow{r}{} \arrow[dashed]{dr}{} & \BR_k(\Mu_{1} \subset\Mu) \arrow{r}{} \arrow{d}{} & (H_{k-1})^{\binom{d}{k}}(\Mu_1) \arrow{r}{} &(H_{k-1})^{\binom{d}{k}}(\Mu)\\
		&  & \widetilde{\KK}_k(\Mu_1, {E}) \arrow{d}{} &  &\\
		&  & 0 &  &
		\end{tikzcd}
		\]
		where the bottom space $ \widetilde{\KK}_k(\Mu_1, {E})$ is defined by the vertical exact sequence. It injects into $\KK_k(\Mu, \overline{\Mu}_1)$, which is the Poincar\'e dual to $\BR_k(\Mu_1)$ and therefore isomorphic to it as vectorspaces.
		Hence, the exact sequence of the assumption is equivalent to an isomorphism
		\[\BR_k(\Mu_1)\ \cong\  \widetilde{\KK}_k(\Mu_1, {E})\ \longhookrightarrow\ \KK_k(\Mu, \overline{\Mu}_1),\]
where the spaces on the extreme left and extreme right are of the same dimension. Hence, the first isomorphism is therefore equivalent to an isomorphism between the extreme spaces, which is exactly biased Poincar\'e duality.
	\end{proof}
	
	\subsection{The doubling trick and the general characterization theorem}\label{sec:doub}
	
	To prove biased Poincar\'e duality, the above criterion will not be enough, unfortunately. In particular, we cannot always assume that ${E}$ is a closed hypersurface. Instead, the criterion below will suffice in general by Corollary~\ref{cor:envelope}.

Again, we are almost in the same situation as before, gradually generalizing until we can now formulate the main theorem: $\Mu$ is again a rational manifold of dimension $2k-1=d-1$ in $\R^d$, but this time around we allow for rational hypersurfaces in it to have boundary.
	
	Consider for this purpose the complement $\widetilde{\Mu}$ of ${E}$ in $\Mu$, where  ${E}$ is a relatively $(k-2)$-acyclic rational hypersurface with induced boundary in $\Mu$, and the double $\mathrm{D}{E}$ of ${E}$. Notice that the proof of Corollary~\ref{cor:envelope} guarantees the existence of such an envelope in degree $k$ for any $(k-1)$-dimensional subcomplex of a combinatorial $\Mu$ (potentially having to subdivide $\Mu$ outside of that complex).
	
	 There is a \Defn{folding map} 
	\[\uptau:\mathrm{D}{E}\ \longrightarrow\ {E}\]
	for doubles that will be tremendously useful, which identifies the two copies of ${E}$ that we call the \Defn{charts} of $\mathrm{D}{E}$.

	The open manifold $\widetilde{\Mu}$ can be compactified canonically to a compact manifold $\widetilde{\mr{D}}\Mu$ with boundary $\mathrm{D}{E}$, and the folding map extends to a map
	\[\uptau:\widetilde{\mr{D}}\Mu\ \longrightarrow\ \Mu.\]
	
		Let us also introduce the following useful concept: Let $\uppi$ denote the general position projection to a hyperplane $H$, and ${\tet}$ the height over that projection, so that
	\[\AR_k(X)\ = \ \ker[\AR_{k}(\uppi\hspace{0.3mm} X) \ \xrightarrow{\ \cdot{\tet}\ }\ \AR_{k-1}(\uppi\hspace{0.3mm} X)]\]
	for a complex $X$ in $\R^d$.
	Let us first observe the following helpful lemma that translates the problem to the case of doubles.

	\begin{lem}\label{lem:reductodouble}
		Consider $\uppi$ a general position projection to a hyperplane, and let $\vartheta$ denote the height over that hyperplane. The following are equivalent:
		\begin{compactenum}[(1)]
		\item The map
		\[\BR_k(\uppi\hspace{0.3mm} {E}\subset \uppi\hspace{0.3mm} \Mu)\ \xrightarrow{\ \cdot {\tet}\ }\ \BR_{k-1}(\uppi\hspace{0.3mm} {E},\partial \uppi\hspace{0.3mm} {E}),\]
		which factors as
			\[\begin{tikzcd}
		\BR_k(\uppi\hspace{0.3mm} {E}\subset \uppi\hspace{0.3mm} \Mu)  \arrow{r}{}\arrow{dr}{\ \cdot {\tet}\ }&  \BR_{k-1}(\uppi\hspace{0.3mm} {E},\partial \uppi\hspace{0.3mm} {E})\\
		&\BR_{k-1}(\uppi\hspace{0.3mm} {E}\subset \uppi\hspace{0.3mm} \Mu) \arrow{u}{}
		\end{tikzcd}
		\]
		is a surjection.
\item	The isomorphism
				\[\AR_k({\widetilde{\mr{D}}\Mu}, \mathrm{D}{E} )\ \xrightarrow{\ \sim\ }\ \AR_k({\Mu}, {E})\]
		descends to an isomorphism of the image $\KK(\widetilde{\mr{D}}\Mu, \mathrm{D}{E}) $ of 
		\[\BR_k({\widetilde{\mr{D}}\Mu} \subset \Mu)\ \longrightarrow\ \AR_k({\widetilde{\mr{D}}\Mu}, \mathrm{D}{E} )\] to the image $\KK_k(\Mu,{E})$ of
		\[\BR_k({\Mu})\ \longrightarrow\ \AR_k({\Mu}, {E}).
		\]
		\end{compactenum}
	\end{lem}
	
	Here we read $\BR_k(X\subset\Mu)$ for $X \subset \widetilde{\mr{D}}\Mu$ as the kernel of the composition
	\[\begin{tikzcd}
	\AR_k(X)  \arrow{r}{}\arrow{dr}{}&  (H_{k-1})^{\binom{d}{k}}(\Mu)\\
	& (H_{k-1})^{\binom{d}{k}}(\widetilde{\mr{D}}\Mu)\arrow{u}{\uptau}
	\end{tikzcd}
	\]
	Let us briefly take a moment to explain why we require the above lemma: We wish to prove the existence of a surjection
	\[\BR_{k}(\Mu)_{|\overline{{E}}}\ \longtwoheadrightarrow\ \KK_k(\Mu,{E}).\]
	Unfortunately, the first space, a partitioned subspace of $\BR_{k}(\Mu)$, lacks a nice descripition. Therefore, we look at doubles.
	Notice that \[\AR_k(\st_\sigma \widetilde{\mr{D}}\Mu)\ =\ 0\] unless $\sigma$ is an interior $(k-1)$-face of $\widetilde{\mr{D}}\Mu$, so that
	we have a well-defined map
	\[\bigoplus\limits_{\sigma \in \widetilde{\mr{D}}\Mu^{(k-1)}} \AR_k(\st_\sigma \widetilde{\mr{D}}\Mu) \ \longrightarrow\ \bigoplus\limits_{\substack{\sigma \in \Mu^{(k-1)}, \\ \sigma \notin {{E}}} }  \AR_k(\st_\sigma \Mu)\]
	and in particular a map
		\[\BR_k(\widetilde{\mr{D}}\Mu)\ \longrightarrow\ \BR_{k}(\Mu)_{\mid\overline{{E}}}.\]
	We can now see how the doubling construction helps:
	 Consider the commutative diagram%[row sep=huge, column sep=huge, text height=1.5ex, text depth=0.25ex]
	\[\begin{tikzcd}
	\KK_k(\Mu,{E}) &    \KK_k(\widetilde{\mr{D}}\Mu, \mathrm{D}{E}) \arrow{l}{} \\
	\BR_{k}(\Mu)_{\mid\overline{{E}}} \arrow{u}{} &  \BR_k(\widetilde{\mr{D}}\Mu) \arrow{l}{} \arrow{u}{}\\
	\bigoplus\limits_{\substack{\sigma \in \Mu^{(k-1)}, \\ \sigma \notin {{E}}} }  \AR_k(\st_\sigma \Mu)  \arrow[two heads]{u}{} &  \bigoplus\limits_{\sigma \in \widetilde{\mr{D}}\Mu^{(k-1)}} \AR_k(\st_\sigma \widetilde{\mr{D}}\Mu). \arrow{l}{} \arrow[two heads]{u}{} 
	\end{tikzcd}
	\]
	The previous lemma ensures that the top horizontal map is a surjection. Hence, proving a surjection of the left top vertical map is reduced to proving a surjection on the right top vertical map.

	\begin{proof}[\textbf{Proof of Lemma~\ref{lem:reductodouble}}]
		The map is clearly an injection as it is the specialization of an injective map. But the conditions ensure that the cokernels of
		\[\BR_k(\uppi\hspace{0.3mm} {E}\subset \uppi\hspace{0.3mm} \Mu)\ \xrightarrow{\ \cdot {\tet}\ }\ \BR_{k-1}(\uppi\hspace{0.3mm} {E}\subset \uppi\hspace{0.3mm} \Mu)\]
		and
		\[\BR_k(\uppi\hspace{0.3mm} \mathrm{D}{E}\subset \uppi\hspace{0.3mm}\Mu)\ \xrightarrow{\ \cdot {\tet}\ }\ \BR_{k-1}(\uppi\hspace{0.3mm} \mathrm{D}{E}\subset \uppi\hspace{0.3mm}\Mu)\]
		are generated by elements in $\AR_{k-1}(\partial \uppi\hspace{0.3mm} {E})$, which in particular implies an injection (and in fact an isomorphism)
		\[ \bigslant{\BR_{k-1}(\uppi\hspace{0.3mm} \mathrm{D}{E}\subset \uppi\hspace{0.3mm}\Mu)}{{\tet}\BR_{k}(\uppi\hspace{0.3mm} \mathrm{D}{E} \subset \uppi\hspace{0.3mm}\Mu )}\ \longhookrightarrow\  \bigslant{\BR_{k-1}(\uppi\hspace{0.3mm} {E}\subset \uppi\hspace{0.3mm} \Mu)}{{\tet}\BR_{k}(\uppi\hspace{0.3mm} {E}\subset \uppi\hspace{0.3mm} \Mu)}.\]
		Now we have the following commutative diagram for degrees at most $k$:H
		\[\begin{tikzcd}
		0\arrow{r}{}  & \BR_\ast(\uppi\hspace{0.3mm} \mathrm{D}{E} \subset \uppi\hspace{0.3mm}\Mu) \arrow{r}{} \arrow{d}{} & \BR_\ast(\uppi\hspace{0.3mm}{\widetilde{\mr{D}}\Mu} \subset \uppi\hspace{0.3mm}\Mu ) \arrow{r}{} \arrow{d}{}&  {\AR}_\ast(\uppi\hspace{0.3mm} {\widetilde{\mr{D}}\Mu},\uppi\hspace{0.3mm} \mathrm{D}{E} ) \arrow{r}{} \arrow{d}{} & 0 \\
		0\arrow{r}{} & \BR_\ast(\uppi\hspace{0.3mm} {E} \subset \uppi\hspace{0.3mm} \Mu) \arrow{r}{} & \BR_\ast(\uppi\hspace{0.3mm}\Mu) \arrow{r}{} &  \AR_\ast(\uppi\hspace{0.3mm} \Mu,\uppi\hspace{0.3mm} {E}) \arrow{r}{} & 0 
		\end{tikzcd}
		\]
		which after reduction by $\tet$
		implies a commutative diagram
		\[\begin{tikzcd}
		\BR_k({\widetilde{\mr{D}}\Mu} \subset \Mu ) \arrow{r}{} \arrow{d}{}&  {\AR}_k({\widetilde{\mr{D}}\Mu},\mathrm{D}{E} ) \arrow{r}{} \arrow{d}{} & \bigslant{\BR_{k-1}(\uppi\hspace{0.3mm} \mathrm{D}{E}\subset \uppi\hspace{0.3mm}\Mu)}{{\tet}\BR_{k}(\uppi\hspace{0.3mm} \mathrm{D}{E} \subset \uppi\hspace{0.3mm}\Mu )}\arrow{r}{} \arrow[hook]{d}{} &0 \\
		\BR_k(\Mu) \arrow{r}{} &  \AR_k(\Mu,{E}) \arrow{r}{} &   \bigslant{\BR_{k-1}(\uppi\hspace{0.3mm} {E}\subset \uppi\hspace{0.3mm} \Mu)}{{\tet}\BR_{k}(\uppi\hspace{0.3mm} {E}\subset \uppi\hspace{0.3mm} \Mu)} \arrow{r}{} &0 
		\end{tikzcd}
		\]
		which in turn implies that the desired map is also a surjection of the middle isomorphism
		\[{\AR}_k({\widetilde{\mr{D}}\Mu},\mathrm{D}{E} ) \ \longrightarrow\ \AR_k(\Mu,{E})\]
		when restricted to the images of
		$\BR_k({\widetilde{\mr{D}}\Mu} \subset \Mu )$ and $\BR_k(\Mu)$, respectively.
	\end{proof}
		
	This takes us rather far, as we now have to prove that 
	$\BR({\widetilde{\mr{D}}\Mu})$ surjects onto $\KK(\widetilde{\mr{D}}\Mu, \mathrm{D}{E})$ in degree $k$. We combine this with the arguments for Theorem~\ref{thm:char_ospd} to obtain:
		
	\begin{thm}\label{thm:ospd2}
		In the situation of the previous Lemma~\ref{lem:reductodouble},
		\[\BR_k(\mathrm{D}{E}\subset\Mu)\ \longrightarrow\ (H_{k-1})^{\binom{d}{k}}({\widetilde{\mr{D}}\Mu})\ \longrightarrow\  (H_{k-1})^{\binom{d}{k}}(\Mu)\]
		is exact if and only if 
		\[\BR_k({\widetilde{\mr{D}}\Mu})\ \longrightarrow\ \KK_k(\widetilde{\mr{D}}\Mu, \mathrm{D}{E})\]
		is surjective. 
	\end{thm}
	
	\begin{proof}
		Consider as in the proof of Theorem~\ref{thm:char_ospd} the commutative diagram
		\[\begin{tikzcd}
		&  & 0 \arrow{d}{} &  &\\
		&  & \BR_k(\mathrm{D}{E}\subset\Mu)   \arrow{d}{}\arrow[dashed]{dr}{\mathrm{hom}} &  &\\
		0 \arrow{r}{} & \BR_k({\widetilde{\mr{D}}\Mu}) \arrow{r}{} \arrow[dashed]{dr}{} & \BR_k(\widetilde{\mr{D}}\Mu \subset\Mu) \arrow{r}{} \arrow{d}{} & (H_{k-1})^{\binom{d}{k}}(\widetilde{\mr{D}}\Mu) \arrow{r}{} &(H_{k-1})^{\binom{d}{k}}(\Mu)\\
		&  & \KK_k(\widetilde{\mr{D}}\Mu, \mathrm{D}{E}) \arrow{d}{} &  &\\
		&  & 0 &  &
		\end{tikzcd}
		\]
		The vertical sequence is exact by definition, and the horizontal sequence is exact by Proposition~\ref{prp:part}. The claim follows at once by considering the diagonal maps.
	\end{proof}
		
	\subsection{The Hall-Laman relations}
	
	Finally, biased Poincar\'e duality allows us to formulate a Lefschetz property at ideals. We say that $\Mu$ a rational manifold of dimension $d-1$ in $\R^d$ satisfies the \Defn{Hall-Laman relations} in degree $k\le \frac{d}{2}$ and with respect to an ideal $\mathcal{I}^\ast\subset \BR^\ast(\Mu)$ if there exists an $\ell$ in $\BR^1(\Mu)$, the pairing
	\begin{equation}\label{eq:sl}
	\begin{array}{ccccc}
	\mathcal{I}^k& \times &\mathcal{I}^k & \longrightarrow &\ \mathcal{I}^d\cong \R \\
	a	&		& b& {{\xmapsto{\ \ \ \ }}} &\ \mr{deg}(ab\ell^{d-2k})
	\end{array}
	\end{equation}
	is nondegenerate. Usually interested in face-rings, we will say $\Mu$ satisfies the \Defn{Hall-Laman relations} if it does so at all squarefree monomial ideals, that is, for all $\KK^\ast(\Mu,\varDelta)$, where $\varDelta$ is any subcomplex of $\Mu$, and their annihilators, that is,
\[\overline{\KK}^\ast(\Mu,\varDelta)\ \coloneqq\ \ker\left[\BR^\ast(\Mu)\ \longrightarrow\ \bigoplus_{\substack{\sigma \in \Mu\\ \sigma \notin \varDelta}} \AR^\ast(\st_\sigma \Mu) \right] .\]	
Notice that the two statements are equivalent for the middle pairing, that is,
when $2k~=~d$, because ${\KK}^k(\Mu,\varDelta)$ and $\overline{\KK}^k(\Mu,\varDelta)$ are orthogonal complements in $\BR^k(\Mu)$.
	
	\section{Perturbations via biased Poincar\'e duality}\label{sec:perturbation}
	
	To prove biased Poincar\'e duality in higher degrees, we need to understand the Lefschetz theorem for face rings of manifolds; this in itself is not a hard task, but requires some care as to how we actually construct Lefschetz elements. The idea is provided in the following lemma.
	
	\subsection{The perturbation lemma} The following lemma is, together with biased Poincar\'e duality, the heart of this paper:
	
	\begin{lem}\label{lem:perturbation}
		Consider two linear maps 
		\[\upalpha, \upbeta: \mathcal{X}\ \longrightarrow\ \mathcal{Y}\]
		of two vector spaces $\mathcal{X}$ and $\mathcal{Y}$ over $\R$.
		\begin{compactenum}[(1)]
			\item Assume that $\upbeta$ has image transversal to the image of $\upalpha$, that is,
			\[\upbeta(\ker\upalpha)\ \cap\ \im\upalpha\ =\ {0}\ \subset\ \mathcal{Y} .\]
			Then a generic linear combination 
			$\upalpha ``{+}" \upbeta$ of $\upalpha$ and $\upbeta$
			has kernel 
			\[\ker (
			\upalpha\ ``{+}"\ \upbeta)\ = \ \ker \upalpha\ \cap\ \ker \upbeta.\]
			\item 
			Similarly, if 
			\[\upbeta^\diamondsuit(\ker(\upalpha^\diamondsuit))\ \cap\ \im(\upalpha^\diamondsuit)\  =\ {0}\ \subset\ \mathcal{X}^\diamondsuit \]
			or equivalently
			\[\upbeta^{-1}(\im\upalpha)\ +\ \ker\upalpha\ =\ \mathcal{X},\]
			then
			\[\im (
			\upalpha\ ``{+}"\ \upbeta)\ = \  \im \upalpha + \im \upbeta .\]
		\end{compactenum}
	\end{lem}
	
Here we use $(\cdot)^\diamondsuit$ to denote dual maps and vector spaces in the basic linear-algebra sense of the notion. 
	
	\begin{proof}
		This is an immediate consequence of the implicit function theorem, applied to the map $\upalpha+\varepsilon \upbeta$ at $\varepsilon=0$.
		
		More instructively, let us take a brief moment to visualize Statement~(1).
		Find a subspace $A$ of $\mc{X}$ such that \[\upalpha A\ =\ \upalpha \mc{X}\quad \text{and}\ \quad \mc{X}\ \cong\ A \oplus \ker\upalpha\]
		so that in particular $\upalpha$ is injective on $A$. 
		Then, for $\varepsilon>0$ small enough, the image of
		\[\upalpha\ +\ \varepsilon \upbeta{:}\ \mc{X}\ \longrightarrow\ \mc{Y}\]
		is 
		\[(\upalpha\ +\ \varepsilon \upbeta)(A)\ +\ \upbeta\ker\upalpha\ \subset\ \mathcal{Y}.\]
		But if we norm $\mc{Y}$ in any way, then 
		$(\upalpha+\varepsilon \upbeta)(A)$ approximates $\upalpha A$ as $\varepsilon$ tends to zero, which is linearly independent from $\upbeta\, \ker\upalpha$ by assumption. Claim (1) follows directly, and so does (2) by duality, although we encourage the reader to develop a direct geometric intuition for~it.
	\end{proof}

Note that this simple lemma immediately extends to any infinite field, independent of characteristic, by computing the Hasse derivative of $(\upalpha+\varepsilon \upbeta)(A)$ at $0$. Alternatively, one can use representation theory of the Kronecker quiver \cite{Ringel}. Compare also to Gurvits' Quantum Matching theory \cite{Gurvits} for another approach towards Hall type theorems over vectorspaces.
	
	This is an amazingly innocent and powerful lemma, as it allows us to control the kernel resp.\ the image of a linear combination of two maps. Note that the conditions are naturally dual to each other, so that when we apply the construction in Gorenstein rings associated to simplicial manifolds, we usually get both conclusions for the price of one. Let us also remark that the proofs of the hard Lefschetz theorem via the decomposition principle, such as those of Murai, Laman and Whiteley discussed in Section~\ref{sec:local-move}, can be thought of as employing a very special case of the perturbation lemma.

	\subsection{The approximation lemma}
	
	Before we continue, let us note that we can actually do a little better, and in fact approximate image and kernel of $\upalpha ``{+}" \upbeta$: Note quite simply the following lemma, which follows immediately from the proof of the Perturbation Lemma~\ref{lem:perturbation}.
	
	\begin{lem}\label{lem:approx}
		Consider $\upalpha$ and $\upbeta$ as in the perturbation lemma, and assume $\mathcal{X}$ and $\mathcal{Y}$ are normed. Then
		\begin{compactenum}[(1)]
			\item Assume that $\upbeta$ has image transversal to the image of $\upalpha$, that is,
			\[\upbeta(\ker\upalpha)\ \cap\ \im\upalpha\ =\ {0}\ \subset\ \mathcal{Y}.\]
			Then  
			\[ \lim_{\varepsilon \rightarrow 0} \im (
			\upalpha + \varepsilon \upbeta)\ = \ \langle  \upalpha(\mathcal{X}), \upbeta(\ker \upalpha)\rangle.\]
			\item 
			Similarly, if 
			\[\upbeta^\diamondsuit(\ker(\upalpha^\diamondsuit))\cap \im(\upalpha^\diamondsuit)\   =\ {0} \subset\ \mathcal{X}^\diamondsuit \]
			then
			\[ \lim_{\varepsilon \rightarrow 0} \ker (\upalpha + \varepsilon \upbeta)\ = \ \ker  \upalpha \cap \ker [\upbeta: \mathcal{X} \rightarrow \mathcal{Y}/ \im \upalpha]. \qed\]
		\end{compactenum}
	\end{lem}
	
	Why is this lemma useful? Consider a sphere $\varSigma$ of dimension $d-1$, let $k\le \nicefrac{d}{2}$, and let $i < d-2k$
	then we do not expect
	\[\AR^k(\varSigma)\ \xrightarrow{\ \cdot\ell^{i}\ }\ \AR^{k+i}(\varSigma)\]
	to be an isomorphism, but only an injection at best. Lemma~\ref{lem:approx}(1) allows us to describe the image. Similarly, the second part of the same lemma will allow us to study the kernel of 
	\[\AR^{d-k-i}(\varSigma)\ \xrightarrow{\ \cdot\ell^{i}\ }\ \AR^{d-k}(\varSigma).\]
	Hence, we obtain a quite explicit way to study primitive classes under the Lefschetz map.
	We explore this later in an avocation.
	
	\subsection{Transversal prime property} The strategy for the Lefschetz theorem is to prove the following property:
	
	Let $\Mu$ denote a $(d-1)$-manifold (possibly with boundary) in $\R^d$, and $k\le \frac{d}{2}$. Let $W$ denote a subset of the vertices of $\Mu$. Then $\Mu$ has the \Defn{transversal prime property} in degree $k$ and with respect to $W$ if
	\begin{align*}&\ \ker\left[(``{\sum_{w\in W}}" x_w)^{d-2k}: \BR_{d-k}(\Mu)\ \longrightarrow\ \BR_{k}(\Mu,\partial \Mu)\right]\\
	=&\ \ann_{\BR_{d-k}(\Mu)} \langle x_w| w\in W \subset \Mu^{(0)}\rangle \\
=&\ \bigcap_{w\in W} \ker\left[x_w: \BR_{d-k}(\Mu)\ \longrightarrow\ \BR_{d-k-1}(\Mu)\right] 
	\end{align*}
	We remind the reader that 
	\[``{\sum_{w\in W}}" x_w\]
	stands for a generic linear combination of prime divisors $x_w$ with $w\in W$.
	
	Note that if the transversal prime property holds for $W=\Mu^{(0)}$, then we conclude the hard Lefschetz theorem for $\Mu$. The general strategy for us is to prove the transversal prime property by linearly ordering the vertices in any way, and proceeding by induction on the size of initial segments. 
		
	\subsection{Transversal primes as another version of the Hall-Laman relations}\label{ssc:stable}
	
	Consider a subset of vertices $W$ in the interior $\Mu^\circ$ of $\Mu$, any $(d-1)$-manifold as above. Then $\KK^\ast(\Mu,\Mu-W)$ is generated in degree one, and if $\Mu$ satisfies the transversal prime property with respect to $W$ in degree $k$, then we obtain an isomorphism
	\[\KK_{d-k}(\Mu,\Mu-W)\ \xrightarrow{\ \cdot(``{\sum_{w\in W}}" x_w)^{d-2k}\ }\ \im\left[\bigoplus_{w\in W} \AR_k (\St_w \Mu) \rightarrow \BR_k (\Mu,\partial \Mu) \right].\]
	This in itself forms another Lefschetz theorem, for the ideal $\KK^\ast(\Mu,\Mu-W)$ instead of its annihilator. Note that in particular, the transversal prime property in degree $k$ is equivalent to saying that
	\begin{align*}&\ \im\left[(``{\sum_{w\in W}}" x_w)^{d-2k}: \BR_{d-k}(\Mu)\ \longrightarrow\ \BR_{k}(\Mu,\partial \Mu)\right]\\
	=&\ \im\left[\bigoplus_{w\in W} \AR_k (\st_w \Mu) \rightarrow \BR_k (\Mu,\partial \Mu) \right]\\
		=&\ \im\left[\BR_k (\Mu)_{\mid\overline{\Mu-W}} \rightarrow \BR_k (\Mu,\partial \Mu) \right].
\end{align*}
	Alas, we can prove the statement inductively, and directly so in some cases, for instance for nicely decomposable spheres we discuss now.
	
	\subsection{Avocation: Arranged marriage and Lefschetz theorems for decomposable spheres}\label{sec:vd}

Ideally we want to prove the transversal prime property by induction on the size of $W$. This follows an greedy proof of Hall's marriage theorem: Given two finite sets $X,\ Y$, and a relation $\mc{R}$ between them so that every subset $A$ of $X$ has at least $|A|$ elements $A\mc{R} \subset Y$ related to it. To find the injective map supported by it,  start with an element $x_1$ of $X$, and "marry" it to an element $y_1$ of $x_1\mc{R} \subset Y$. Then proceed to marry another element $x_2$, then the next, and so on. At some point this proof may run into trouble, as all elements of $x_i\mc{R} \subset Y$ may already be married to $x_j$ for some $j<i$. Now, the expansive property has to be employed to reshuffle the previously married element to accommodate $x_i$ by constructing an alternating paths to an element of $\{x_j:j\in [i]\}\mc{R} \subset Y$ not matched. 

We now discuss how this is employed in our context of simplicial spheres and their face rings. Essentially, we want to avoid the reshuffling step, as rings seem to rigid to handle it, and because the perturbation Lemma~\ref{lem:perturbation} is not designed to describe higher order perturbations, or equivalently global properties, immediately. Indeed, reshuffling would correspond to analyzing the higher derivatives (or global properties) of $\upalpha+\varepsilon\upbeta$, which gets increasingly tricky and we want to avoid initially. Instead, we notice that reshuffling is not necessary in the algebraic setting for a simple case, and then describe a homological way to analyse the global structure nevertheless.

Without further ado, let us illustrate our strategy with a toy example. A pure simplicial $d$-manifold is \Defn{$L$-decomposable} if it is a simplex or there exists a vertex whose deletion is $L$-decomposable of dimension $d$, and such that the boundary of the link of the same vertex is empty, or $L$-decomposable of codimension one, compare also~\cite{BP}.

	\begin{thm}\label{thm:vd}
		$L$-decomposable spheres satisfy Theorem~\ref{mthm:gl}(1), the generic Lefschetz property.
	\end{thm}
	
	\begin{proof}
		We prove this theorem by double induction on dimension and the number of vertices, the base case of zero-dimensional spheres being clear. 
		
		We prove stronger that any $L$-decomposable $(d-1)$-ball $\varDelta$ in a $(d-1)$-sphere $\varSigma$ satisfies the hard Lefschetz property with respect to its boundary, that is, with respect to a suitable Artinian reduction and Picard divisor $\ell$, we have an isomorphism
		\[\AR^k(\varDelta, \partial \varDelta)\ \xrightarrow{\ \cdot \ell^{d-2k} \ }\ \AR^{d-k}(\varDelta)\]
which is obtained from the composition
			\[\begin{tikzcd}
		\AR^k(\varDelta, \partial \varDelta)  \arrow{r}{}\arrow{dr}{\ \cdot \ell^{d-2k}\ }&  \AR^{d-k}(\varDelta)\\
		&\AR^{d-k}(\varDelta, \partial \varDelta)  \arrow{u}{}
\end{tikzcd}
\]

		In other words, the sphere $\varSigma$ satisfies the Hall-Laman relations with respect to the monomial ideal generated by the faces in the interior of $\varDelta$.
				\begin{figure}[h!tb]
			\begin{center}
				\includegraphics[scale = 1.2]{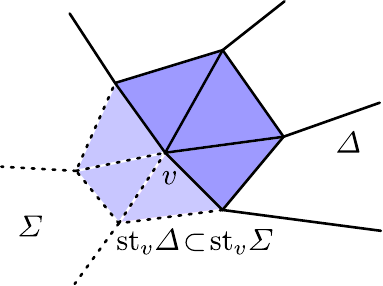}
				\caption{The case of $L$-decomposable spheres.}
				\label{vd}
			\end{center}
		\end{figure}
We want to apply Lemma~\ref{lem:perturbation} to the spaces
\[\mathcal{X}\ = \ \AR^k(\varDelta, \partial \varDelta)\ \quad \text{and}\ \mathcal{Y}\ = \ \AR^{d-k}(\varDelta).\]	
		
		For this purpose we consider such a ball $\varDelta$, and consider $v$ the first vertex removed in a decomposition to a smaller ball. Now, by the cone lemmas, the map
		\begin{equation*}\label{eq:vd2}
\AR^k(\st_v\varDelta,\st_v \partial \varDelta)\ \xrightarrow{\ \cdot x_v^{d-2k}\ }\ \AR^{d-k}(\varDelta, \varDelta-v) \subset \AR^{d-k}(\varDelta)	\end{equation*}
		is an isomorphism if and only if the map
		\[\AR^{k}(\Lk_v\varDelta,\Lk_v \partial \varDelta)\ \xrightarrow{\ \cdot\tet_v^{d-2k-1}\ }\ \AR^{d-k-1}(\Lk_v\varDelta)\]
		is an isomorphism, where $\tet_v$ is the height over the projection along $v$.
But that is just the Hall-Laman relations for $\Lk_v \varSigma$ with respect to the monomial ideal generated by faces interior to $\Lk_v \varDelta$. But, with a sufficiently generic choice of linear system of parameters, we may assume exactly this by induction on the dimension. 

Hence, the image of 		\[\AR^k(\varDelta, \partial \varDelta)\ \xrightarrow{\ \cdot x_v^{d-2k} \ }\ \AR^{d-k}(\varDelta)\]
 is described as $x_v\AR^{d-k-1}(\varDelta)$, and its cokernel is precisely $\AR^{d-k}(\varDelta-v)$. It follows by Poincar\'e duality that the kernel is then
 $\AR^k(\varDelta-v, \partial (\varDelta-v))$.

		We can in addition assume that $\varDelta-v$ satisfies the hard Lefschetz theorem with respect to its boundary already by induction on the size of $\varDelta$ measured in number of vertices, so that by virtue of a Picard divisor $\ell' \in \AR^1(\varDelta)$ we have an isomorphism
		\[\AR^k(\varDelta-v, \partial (\varDelta-v))\ \xrightarrow{\ \cdot (\ell')^{d-2k} \ }\ \AR^{d-k}(\varDelta-v).\]
		Notice now that 
		\[\ker\left[x_v^{d-2k}:\AR^k(\varDelta, \partial \varDelta)\ \rightarrow\ \AR^{d-k}(\varDelta)\right]\ =\ \AR^k(\varDelta-v, \partial (\varDelta-v))\]
		and 
		\[\coker\left[x_v^{d-2k}:\AR^k(\varDelta, \partial \varDelta)\ \rightarrow\ \AR^{d-k}(\varDelta)\right]\ =\ \AR^{d-k}(\varDelta-v).\]
		Hence $x_v=:\upalpha$ and $\ell'=:\upbeta$ satisfies the first condition of Lemma~\ref{lem:perturbation}.
		
		\begin{rem}
		Really, we are considering the $(d-2k)$th powers of the multiplications with these elements here, but the adaption of Lemma~\ref{lem:perturbation} is straightforward. We will later see that the case of $d-2k> 1$ can also be straightforwardly and directly reduced to the case $d-2k=1$, justifying our handwaving in this step.
				\end{rem}

		 In conclusion, the Hall-Laman relations are true at $\varDelta$ in $\varSigma$ if they are true at $\varDelta-v$.		
	\end{proof}
	
	We can actually use Lemma~\ref{lem:approx} to say something more. For this, we look at the Weil dual situation.

Given a geometric simplicial complex $\varGamma$ in $\R^d$, and $\alpha$ an element of $\AR_k(\varGamma)$, where $k\le d$, then we define the $p$-norm of $\alpha$ as 
\[||\alpha||_p\ =\ \left(\sum_{\sigma \in \varDelta^{(k-1)}}  |\mathbf{x}^\sigma \alpha|^p\right)^{\frac{1}{p}}.\]	
For us, the $2$-norm shall generally suffice, and we shall refer to that norm in calculations.

Consider, with the intersection ring as above and $ 1, 2, \dots$ an initial segment of vertices in the $(d-1)$-sphere $\varSigma$ in $\R^d$ that form a vertex-decomposition of the same. We can then consider the divisor
	\[\ell\ \coloneqq \ \sum \varepsilon_i x_{i}\]
	where $|\varepsilon_i| \gg |\varepsilon_{i+1}|$.
	\begin{prp}\label{prp:approx}
		With respect to $\ell$ as above, and for any positive $j\le d$, its kernel in 
		\[\AR_{j}(\varSigma)\ \xrightarrow{\ \cdot\ell\ }\ \AR_{j-1}(\varSigma)\] 
approaches a stress subspace, and its image approaches a partitioned subspace as $\nicefrac{|\varepsilon_i|}{|\varepsilon_{i+1}|}\rightarrow \infty$.
	\end{prp}
	
	This is exceptionally confounding , allowing us to reduce to the discussion of stress and partitioned spaces as announced in Section~\ref{ssc:sop}, and allowing us to even describe and discuss primitive classes in an especially beautiful way. Let us therefore define, given a sequence of indeterminates $(x_1,x_2,\dots)$ a \Defn{decaying} sequence of linear forms as a sequence 
	\[\ell_t\ \coloneqq \ \sum \varepsilon_{t,i} x_{i}\]
	with 
	\[\frac{|\varepsilon_{t,i}|}{|\varepsilon_{t,i+1}|}\ \xrightarrow{\ t\rightarrow \infty\ }\ \infty\]
	sufficiently fast (for instance, faster than any algebraic function).

	\subsection{The general perturbative approach}\label{sec:genperap}
	
	Let us briefly summarize and generalize the observation employed in the proof of Theorem~\ref{thm:vd}.
	
	Consider $\Mu$ a rational manifold with boundary of dimension $(d-1)$ realized in $\R^d$, and assume that we proved the transversal prime property for a set of vertices $W'$ of $\Mu$ for the map between degree $d-k$ and degree $k$. For simplicity of notation, we will assume that $d=2k+1$, the other Lefschetz maps will be reduced to that case.
		 Consider an additional vertex of $\Mu$ not in $W'$.
	
	When we want to prove the transversal prime property for $W=W'\cup \{w\}$, we want to use Lemma~\ref{lem:perturbation}, applied to the maps 
	\[\upalpha\ = \ ``{\sum_{v\in W'}}" x_{v}\ \quad \text{and}\ \quad \upbeta\ =\ x_w\]
	and to the spaces 
	\[\mathcal{X}\ = \ \BR_{k+1}(\Mu)\ \quad \text{and}\ \mathcal{Y}\ = \ \BR_{k}(\Mu,\partial \Mu).\]	
Now, note that if we have already proven the transversal prime property with respect to $W'$, then kernel and image of $\upalpha$ have a particularly nice description:
\[\ker\upalpha\ =\ \bigcap_{v\in W'} \ker\left[x_v: \BR_{k+1}(\Mu)\ \longrightarrow\ \BR_{k}(\Mu)\right]\]
and
\[\im\upalpha\ =\ \im\left[\bigoplus_{v\in W'} \AR_k (\st_v \Mu) \rightarrow \BR_k (\Mu,\partial \Mu) \right].\]
	We have the following central theorem:
	\begin{thm}\label{thm:perturb}
		In the situation above, the following are equivalent.
		\begin{compactenum}[(1)]
			\item The assumptions of Lemma~\ref{lem:perturbation}(1) hold, that is, 
			\[\upbeta(\ker\upalpha)\ \cap\ \im\upalpha\ =\ {0}\ \subset\ \mathcal{Y}.\]
			\item The assumptions of Lemma~\ref{lem:perturbation}(2) hold, that is, 
			\[\upbeta^{-1}(\im\upalpha)\ +\ \ker\upalpha\ =\ \mathcal{X}.\]
			\item The pullback of $\ker \upalpha$ to 
			\[\AR_{k}(\st_w \Mu)\ \cong\ \AR_{k}(\lk_w \Mu)\]
			along $x_w$ satisfies biased Poincar\'e duality in the $(d-2)$-sphere $\lk_w \Mu$.
		\end{compactenum}
	\end{thm}

Statement (3) can be interpreted for stresses using Corollary~\ref{cor:map} or equivalently by defining the product on $\AR_{\ast}(\lk_w \Mu)$ by Poincar\'e duality with $\AR^{\ast}(\lk_w \Mu)$.
	
	\begin{proof}
		We have $(1)\Leftrightarrow (2)$ by Poincar\'e duality. For $(3)\Leftrightarrow (1)$, observe that $\ker \upalpha$ and $\im \upalpha$ are orthogonal complements in $\BR_\ast(\Mu)$. Hence 
		\[\mathcal{I}_W\ \coloneqq\ \left(\im\left[\BR_{k+1}(\Mu)\ \xrightarrow{\ ``{\sum_{v\in W'}}" x_{v}\ }\ \BR_{k}(\Mu) \right]\ +\ \BR_{k}(\partial \Mu\ \subset \Mu)\right)\ \cap\ \AR_{k}(\st_w \Mu)
		\]
		in  $\AR_{k}(\st_w \Mu)$
		and 
		\[\mathcal{K}_W\ \coloneqq\ x_w\upalpha\ \subset\ \AR_{k}(\st_w \Mu)\]
		are orthogonal to each other. It follows that
		 \[x_{w}\ker \upalpha\ \cap\ \im\upalpha\ =\ 0\ \ \ \text{in}\ \BR_k (\Mu,\partial \Mu)\] if and only if the Hall-Laman relations are true for $\mc{K}_W$ in $\AR_{\ast}(\st_w \Mu)$.
		 
		To finally show that also $(1)\Rightarrow (3)$, it suffices to show 
		\begin{lem}\label{lem:orth}
			$\mathcal{I}_W$ and $\mathcal{K}_W$ are orthogonal complements in $\AR_{k}(\lk_w \Mu).$	
		\end{lem} 
		For this, it suffices to argue that their dimensions sum up to $\dim\AR_{k}(\st_w \Mu)$, which follows by the short exact sequence
		\[0\ \longrightarrow\ \BR_{k+1}(\Mu-w\subset \Mu)\ \longrightarrow\ \BR_{k+1}(\Mu)\ \longrightarrow\ \AR_{k}(\st_w \Mu)\ \longrightarrow\ 0\]
		and because \[\ker \upalpha  \ \cap\ \ker x_w\quad\text{and}\quad\im\upalpha \ +\ \im x_w\] are orthogonal complements in the space of stresses $\BR_{k+1}(\Mu-w\subset \Mu)$ and its Poincar\'e dual in $\BR_{k}(\Mu,\partial \Mu)$, the image  \[\im\left[\BR_{k}(\Mu,\partial \Mu)\ \longrightarrow\ \AR_{k}(\Mu,\partial \Mu \cup \st_w \Mu)\right].\qedhere\]
	\end{proof}

We can therefore extend the transversal prime property from $W'$	to $W$ if any of the above conditions hold. Note further that the task to verify the conditions in Theorem~\ref{thm:perturb} is much simpler if we use the transversal prime property applies to $W'$, which we did not assume for its proof (indeed, we could have chosen $\upalpha$ in any other way). In that case, the kernel of $\upalpha$ is a stress subspace, and the image is a partitioned space and described by the partition complex of Sections~\ref{sec:part} and~\ref{sec:part2}.

 This has especially nice description if $\Mu-W'$ and
\[\bigcup_{v\in W'}\st_v \Mu\]
	are codimension zero manifolds, when the full force of Proposition~\ref{prp:part} applies. As we see the pullbacks to the star of $w$  appear directly as $\mathcal{K}_W$ resp.\ $\mathcal{I}_W$ above in the induction step, this is of central importance for our calculations.
\begin{figure}[h!tb]
			\begin{center}
				\includegraphics[scale = 0.7]{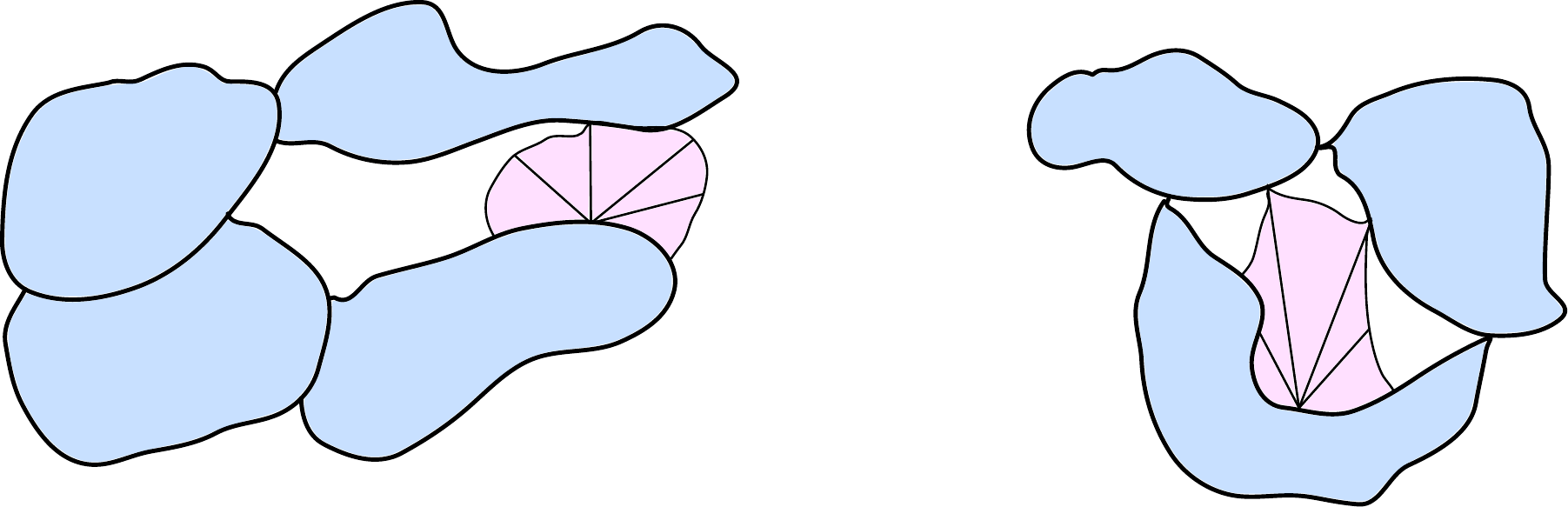}
				\caption{The perturbative approach is simpler if the intermediate complexes are manifolds}
				\label{decomp}
			\end{center}
		\end{figure}
	Unfortunately, this is not always the case. There are then two options: perhaps to use higher order derivatives in Lemma~\ref{lem:perturbation} as alluded to in the beginning of Section~\ref{sec:vd} to realize the algebraic version of reshuffling, or even a direct "homological" version of that Lemma taking into account all prime divisors at once, but we lack such a lemma. Instead, we use another method: we encode global homological information by embeddings into higher-dimensional manifolds, introducing an exterior locus of control and using the fact that in these context, complexes have easier decompositions.
			
	Nevertheless, this observation provides the critical tool for us; it allows us to prove the Lefschetz theorem by induction, having come full circle: We will prove the Hall-Laman relations in lower dimensions, which implies the hard Lefschetz theorem by induction, which in turn implies the Hall-Laman relations using the characterization theorem. The details are the subject of the next two sections. 
	
	\section{Railway constructions, Lefschetz theorems and globality via detours}\label{sec:railway}
	
	The main issue presenting itself to us, if we wish to apply the perturbation lemma in general, is that triangulations even of spheres may be very nasty already in dimension $3$. Indeed, the hardness of distinguishing manifolds and embeddings (as exemplified by the existence of knotted embeddings of circles) in dimension three and higher implies that  in general orders of vertices may not be nice enough to apply the reasoning as we did in the case of shellable spheres. Compare also~\cite{Lickorish} for an explicit construction. The next two sections are largely setup, and we postpone the proofs to Section~\ref{sec:induct}. We also focus on PL spheres, and postpone the discussion of general manifolds to the final section.
	
The approach can briefly be summarized as follows: Ideally, we would like to have a way to apply the perturbation Lemma~\ref{lem:perturbation} globally, ideally in a homological way. Unfortunately, we have no good statement of that kind that allows us to build a homology theory. Indeed, we do not even have a nice symmetric version of the perturbation Lemma. On the other hand, we have no good way to exhaust our triangulated manifolds either. 

Instead, we introduce globality through the backdoor: We consider our sphere $\varSigma$, for which we do not have a good decomposition but want to prove the Lefschetz theorem, into a higher-dimensional sphere $\widetilde{\varSigma}$ where we can ensure a sufficiently nice decomposition, and use topological properties of the embedding \[\varSigma\ \longhookrightarrow\ \widetilde{\varSigma}\] to conclude the originally desired Lefschetz theorem.
	
	\subsection{Railway construction}
			
We need to refine our notions of envelopes, and in particular Corollary~\ref{cor:envelope}. We greedily ask for the best achievable properties for envelopes, which is why the definition is convoluted and subdivided into two parts.
	
We say a an orientable simplicial rational manifold $U$ in $\R^d$ with induced boundary and of dimension $d-1$, and $\varDelta$ is a $(k-1)$-dimensional induced subcomplex of the interior~$U^\circ=(U,\partial U)$ of $U$, and finally a linear order on the vertices of $\varDelta$ is a \Defn{railway} in degree $k\le d$ if the following conditions are satisfied:
	\begin{compactenum}[(1)]
\item For every initial segment $W$ of the vertices of $\varDelta$, the complex
		$\mr{N}_{W} U$	is the {regular neighbourhood} of an induced subcomplex $U[W] \subset U^\circ$ of $U$, that is, 
		\[\mr{N}_W U\ =\ \bigcup_{w \in W} \St_w U\ =\ \bigcup_{v\ \text{vertex of}\ U[W]} \St_v U\]
		and $\mr{N}_W U$ as a rational manifold with boundary strongly deformation retracts to the subcomplex induced by the vertices of $U[W]$
		\[\mr{N}_{W} \varDelta\ =\ \bigcup_{w \in W} \St_w \varDelta,\]
		which itself strongly retracts to $\mr{N}_W \varDelta$ and $U[W]$.
\item The restriction of $\mr{N}_W U$ to the boundary of $U$ is a regular neighbourhood in $\partial U$.
\item If $w$ is the terminal vertex of $W$ and $W'=W{\setminus} w$, then
\[U_w\ \coloneqq\ (\lk_w U)\ \cap\ (\partial U \cup \mr{N}_{W'} U)\] is a regular neighbourhood submanifold in the $(d-2)$-manifold $\lk_w U$ and its complement $\mr{C}{U}_w$ is a regular neighbourhood homotopy equivalent to $(\lk_w U) \cap (\varDelta-W')$.
\item $U$ is an envelope for $\varDelta$ in degree $k$: we have $\AR^{k}(U)=\AR^{k}(\varDelta)$.
	\end{compactenum}
We call $\varDelta$ its \Defn{tabula}. We call $U$ a \Defn{m\'etro} if it satisfies Conditions~(1), (2) and (3). 
	
	It is often useful to impose an additional condition of topological simplicity. We call a railway resp.\ m\'etro \Defn{octavian} if:
	\begin{compactenum}[(1)]
		\setcounter{enumi}{4}
		\item $\mr{N}_{W'} U$ in $U$ is a $(k-2)$-acyclic rational manifold (for every initial segment $W'$) for all initial segments $W'$, and
		\item so is its complement in $U$: the complementary rational manifold to $\mr{N}_{W'} U$ in $U$ is $(k-2)$-acyclic.
	\end{compactenum}

We call a (octavian) m\'etro of degree $k$ \Defn{hereditary} if for every initial segment $W$ with final vertex $w$, the complex $U_w$ is a hereditary (octavian) m\'etro in $\lk_w \varSigma$ of degree $k-1$, or $\varDelta$ is of dimension $0$. Moreover, we demand that $\lk_w \mr{N}_{W{\setminus}\{w\}}U$ is the initial segment of the m\'etro, that is, it is a regular neighbourhood ball whose vertices are initial in the order.

We call a octavian railway \Defn{hereditary} if for every initial segment $W$ with final vertex $w$, the complex $U_w$ is a hereditary octavian m\'etro in $\lk_w \varSigma$ or $\varDelta$ is of dimension $0$. Moreover, we demand that $\lk_w \mr{N}_{W{\setminus}\{w\}}U$ is the initial segment of the railway, that is, it is a regular neighbourhood ball whose vertices are initial in the order.

	\begin{prp}\label{prp:railways} Consider $\varSigma$ a combinatorial rational homology sphere of dimension $d-1$ realized in $\R^{d}$, and $\varGamma$ a subcomplex of dimension $k-1$ in $\varSigma$ with $k\le \frac{d}{2}$. 
	
	Then (assuming Theorem~\ref{mthm:gl}(1) for $(d-3)$-dimensional PL spheres) there exists a hereditary octavian railway $U$ with tabula $\varDelta \supset \varGamma$ that is, up to subdivision of $\varSigma$, a hypersurface in the same, and $\varDelta$ is an envelope for $\varGamma$ in degree $k$.
	\end{prp}
	
	We call $\varGamma$ the \Defn{republic} of the tabula $\varDelta$ and railway $U$.
	
	\begin{proof}
		The proof is an application of the general position principle in high codimension: we start with an arbitrary order on the vertices of $\varGamma$, which we may assume to be $(k-2)$-acyclic. In fact, by adding faces to $\varGamma$ of dimension $k-1$ to obtain a new complex $\varDelta$, we may assume that after removing initial sets of vertices, the complex $\varDelta$ remains $(k-2)$-acyclic and embeddable in $\varSigma$ up to subdivision, and so is the neighbourhood of the initial vertices. Figure~\ref{map} shows how one modifies center to tabula when removing a first vertex. By subdividing $\varDelta$ outside of $\varGamma$ and putting the new vertices in general position, we can assume the former is an envelope in degree $k$ for the former.
		
		\begin{figure}[h!tb]
			\begin{center}
				\includegraphics[scale = 1.1]{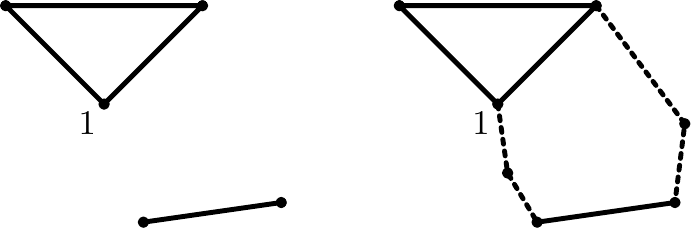}
				\caption{Removing vertices forces us to attach $(k-1)$-faces to keep connectivity conditions intact. It is easiest to visualize in the case of a 3-sphere it with two points removed, to a $2$-sphere, and thinking of the construction as drawing $\varDelta'$ and $\varGamma$ in the plane as an immersed complex using bridges and tunnels as in a knot diagram. }
				\label{map}
			\end{center}
		\end{figure}
		
		It is clear that we can ensure that $\mr{N}_W U$ and $\partial U \cap \mr{N}_W U$ are manifolds at each step, which can be subdivided outside of $\varDelta$ to ensure Conditions~(1) to (3). We can now construct $U$ iteratively around $\varGamma$ as a regular neighbourhood hypersurface in $\varSigma$ using general position of $(k-1)$-dimensional complexes in a $(d-1)$-manifold, as shown in Figure~\ref{railway1}. With the connectivity assumptions on $\varDelta$ we ensured in the beginning, we obtain the octavian property.
		
		\begin{figure}[h!tb]
			\begin{center}
				\includegraphics[scale = 1.0]{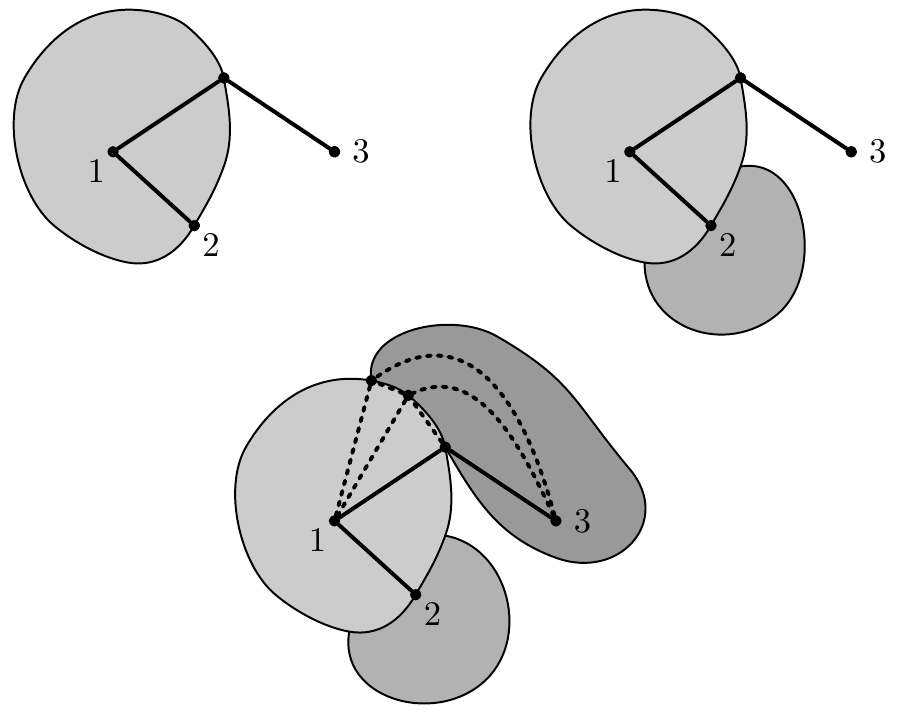}
				\caption{Thickening to a hypersurface, and the intermediate manifolds $U_{\{1\}}$, $U_{\{1,2\}}$ and $U_{\{1,2,3\}}$  }
				\label{railway1}
			\end{center}
		\end{figure} 
		We now turn to making sure that the result is an envelope.			

This is achieved by subdividing the construction in a sufficiently fine way, and applying Theorem~\ref{mthm:gl}: As a regular neighbourhood of $\varDelta$, we may assume $U$ has a simplicial collapse $U\searrow \varDelta$ to $\varDelta$ (see~\cite{Whitehead}). Consider an elementary collapse $U\searrow_e U'$, along a free face~$\sigma$. Stellarly subdivide $\sigma$, and move the new vertices into general position. 

By the Hall-Laman relations for $\lk_{v_\sigma} U{\uparrow }\sigma$, or equivalently the hard Lefschetz theorem for $\lk_{e_\sigma} U{\uparrow }\sigma$ (where $e_\sigma$ is the unique interior edge to $U{\uparrow }\sigma$ incident to $v_\sigma$), we see that after the subdivision, the open star $\st_{v_\sigma}^\circ U{\uparrow }\sigma$ does not support a $k$-stress: If ${\tet}$ is the height over a general position hyperplane in $v^\bot\in \R^d$, $\uppi$ the projection to that hyperplane, then 
\[\AR_{k-1}(\uppi\hspace{0.3mm}\lk_{v_\sigma} U{\uparrow }\sigma)\ \xrightarrow{\ \cdot{\tet}\ }\ \AR_{k-2}(\uppi\hspace{0.3mm}\lk_{v_\sigma} U{\uparrow }\sigma)\]
is an isomorphism. 
 Hence, any $k$-stress of $U{\uparrow} \sigma$ must be supported in $U'$.

		\begin{figure}[h!tb]
			\begin{center}
				\includegraphics[scale = 1.0]{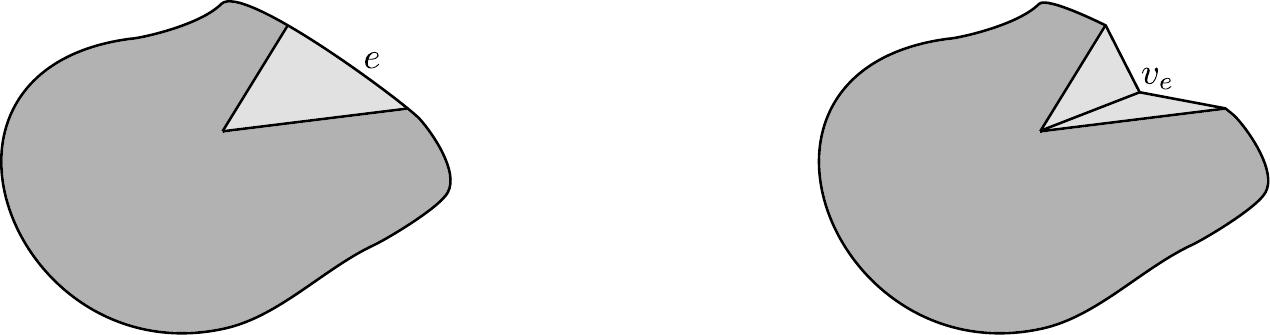}
				\caption{Subdividing a simple homotopy collapse along a free edge $e$, and perturbing it to remove stresses in the neighbourhood of the newly created vertex $v_e$}
				\label{collapse}
			\end{center}
		\end{figure} 

 Repeating this we see that after stellar subdivisions along the collapses, we observe that any $k$-stress of the railway can be assumed to be supported in $\varDelta$. 
		
		Finally, it is straightforward to see that the entire construction can be made hereditary by carefully minding the links in every construction step.
	\end{proof}
	
	\begin{rem}
		It is tempting to draw connections from railways to the somewhat similar notion of dessins d'enfant, at least in the case of surfaces. However, use and critical properties are different, so that we chose not to overemphasize this connection by choosing a different name. 
	\end{rem}
	
	\begin{rem}
	Of course, the proposition is not conditional as we provide a proof of Theorem~\ref{mthm:gl}(1) in the end, but we chose to include the assumption to make the dependencies clear. With a little more work, however, one can even remove that dependency, as one can show that we need Theorem~\ref{mthm:gl}(1) only for sufficiently fine subdivisions, which can be assumed to be polytopal (see \cite{AI}) and satisfy the Lefschetz property following Stanley. 
	\end{rem}
	
	\subsection{Biased Poincar\'e duality via hypersurface envelopes}
	
The critical observation to make is that railways allow us to impose a good order, and therefore show the Lefschetz theorem. Alas, we need to make use of a finer understanding of the map from Chow cohomology to simplicial homology (already encountered in Isomorphism~\eqref{eq:homol}) given by the Map~\eqref{eq:maphom}.
	
Consider a rational sphere $\varSigma$ of dimension $d-1$ realized in $\R^{d}$, where $d=2k$, and $U$ a $(k-2)$-acyclic hypersurface (with boundary) in $\varSigma$. We are interested in the following consequence of Theorem~\ref{thm:ospd2}. 
	
	\begin{thm}\label{thm:env}
		Under the preceding conditions, assume that
		\begin{compactenum}[(1)]
			\item with respect to a projection $\uppi$ to a general position hyperplane, the height ${\tet}$ induces an isomorphism
			\[\BR_{k}(\uppi\hspace{0.3mm} U) \ \xrightarrow{\ \cdot \tet\ }\ \AR_{k-1}(\uppi\hspace{0.3mm} U, \uppi\hspace{0.3mm} \partial U).\]
			\item The map
			\[\AR_k(\mathrm{D}U)\ \longrightarrow\ (H_{k-1})^{\binom{d}{k}}({\widetilde{\mr{D}}\varSigma})\]
			is surjective. 
		\end{compactenum}
		Then we obtain the \emph{biased Poincar\'e duality property for $U$:} Biased Poincar\'e duality holds with respect to the ideal $\KK^{k}(U,\varSigma)$ in $\varSigma$. Conversely, if (1) holds, biased  Poincar\'e duality implies (2).
		\qed
	\end{thm}
	
	It is useful to note a particular corollary of this fact: 
	
	\begin{cor}
		Under the preceding conditions, assume that $\varDelta$ is the $(k-1)$-skeleton of a rational hypersurface and sphere $\check{\varSigma}$ in $\varSigma$ and that $U$ is a hypersurface envelope in degree $k$ for $\varDelta$ in some refinement of $\varSigma$ (that does not affect $\varDelta$). Then we have an isomorphism
		\[\AR^{k-1}(\uppi\hspace{0.3mm} \check{\varSigma})\ \xrightarrow{\ \cdot \tet\ }\ \AR^{k}(\uppi\hspace{0.3mm} \check{\varSigma}) \]
that is, the Picard divisor $\tet$ defines a Lefschetz map in $\uppi\hspace{0.3mm} \check{\varSigma}$ with respect to the middle isomorphism. 	\end{cor}	
	
	\begin{proof}
	Following Theorem~\ref{thm:env}, $\varSigma$ satisfies biased Poincar\'e duality with respect to $\varDelta$. The claim follows by
 Proposition~\ref{prp:ospd}.
	\end{proof}
	
	This observation is of central importance.	
	
	\subsection{Using railways to establish the middle Lefschetz isomorphism}\label{ssc:railways}
	
	We obtain an interesting "bait and switch" trick to simplify the proof of the middle Lefschetz theorem, and the proof of biased Poincar\'e duality with respect to subcomplexes in spheres; indeed, while both may be hard to do for a specific surface resp.\ hypersurface envelope, it is exceedingly easier for railways. In turn, we only have to discuss the Lefschetz properties of the latter to conclude properties of the former.
	
	In particular, we obtain a proof of the middle isomorphism of the hard Lefschetz theorem for a $(2k-2)$-dimensional $\varSigma$ sphere embedded as a hypersurface in a sphere $\varSigma'$ of dimension $(2k-1)$ without having to worry about good orderings in $\varSigma$. The construction of $\varSigma'$ is of no importance, as both biased Poincar\'e duality for $\KK(\varSigma',\varSigma)$ and hard Lefschetz for $\varSigma$ are equivalent (see Proposition~\ref{prp:ospd}), so we may simply take the suspension $\susp\varSigma$ realized in $\R^{2k}$.%, at least provided that $\varSigma$ is a PL sphere. If $\varSigma$ is not a PL sphere, then $\susp\varSigma$ is not a combinatorial sphere, but we can use results of Freedman and Kervaire to realize $\varSigma$ as hypersurface in a combinatorial sphere $\varSigma'$.
%		It is however easier to simply consider $\susp\varSigma$ itself, and notice that the hereditary octavian railways for $\varSigma$ in the former can be chosen not to intersect the PL singular points.
	
	But the former can be proven using a different envelope for the $k$-skeleton of $\varSigma$, in particular we can choose it to be a railway with all the nice properties we desire. The (middle) hard Lefschetz isomorphism for $\varSigma$ follows as a corollary. This means that contrary to the case of $L$-decomposable spheres, we rely on an indirect way to derive the hard Lefschetz theorem.

\subsection{$\dots$ and the Hall-Laman relations}\label{sec:beymid}
We can also use the previous observation to establish the Hall-Laman relations and the general Lefschetz isomorphism. That is comparatively easy to do directly if the sphere is itself a hereditary m\'etro, but if that is not the case, then we circumvent the issue by reducing the problem to biased Poincar\'e duality, which in turn we prove using an octavian and hereditary railway for an appropriate complex.

To this end, assume we wish to prove the Hall-Laman relations for a pair $(\varSigma, \varDelta)$, where $\varDelta$ is a subcomplex of $\varSigma$ a $(d-1)$-sphere, specifically the Hall-Laman relations for
$\KK^\ast(\varSigma,\varDelta)$ or its annihilator. Label the two vertices of the suspension $\mbf{n}$ and $\mbf{s}$ (for north and south).
We then argue using the following observation, which follows as Proposition~\ref{prp:ospd}. Let $\uppi$ denote the projection along $\mbf{n}$, and let $\tet$ denote the height over that projection, and let $A\ast B$ denote the free join of two simplicial complexes $A$ and $B$.
	
	\begin{lem}\label{lem:midred}
Considering $\susp\varSigma$ realized in $\R^{d+1}$, and $k< \frac{d}{2}$, the following two are equivalent:
		\begin{compactenum}[(1)]
					\item The Hall-Laman relations for \[\KK^{k+1}(\susp\varSigma,\susp(\varDelta) \cup \mbf{s}\ast\varSigma)	\quad\text{resp.} \quad \overline{\KK}^{k+1}(\susp\varSigma,\mbf{n}\ast\varDelta)\]
					with respect to $x_{\mbf{n}}$.									
			\item The Hall-Laman relations for 
 \[\KK^{k}(\uppi\hspace{0.3mm}\varSigma,\uppi\hspace{0.3mm} \varDelta)	\quad\text{resp.} \quad \overline{\KK}^{k}(\uppi\hspace{0.3mm}\varSigma,\uppi\hspace{0.3mm} \varDelta)\]			
with respect to $\tet$. 
		\end{compactenum}
	\end{lem}
	
\begin{proof}
	Without loss of generality we have $\tet=x_{\mbf{n}}-x_{\mbf{s}}$ in  $\AR^{\ast}(\susp\varSigma)$.
Consider then the diagram
\[\begin{tikzcd}[column sep=5em]
 \AR^{k}(\uppi\varSigma) \arrow{r}{\ \ \ \ \cdot \tet^{d-2k}\ \ \ \  } \arrow{d}{\sim } & \AR^{d-k}(\uppi\varSigma) \arrow{d}{\sim } \\
\AR^{k+1}(\susp\varSigma,\mbf{s}\ast\varSigma) \arrow{r}{\ \ \ \cdot x_{\mbf{n}}^{d-2k-1}\ \ \ } & \AR^{d-k}(\mbf{n}\ast\varSigma)
\end{tikzcd}
\]
where the first vertical map is defined by the composition of cone lemmas
\[\AR^{k}(\uppi\varSigma)\ \cong\ \AR^{k}(\mbf{n}\ast\varSigma)\ \xrightarrow{\ \cdot x_{\mbf{n}} \ }\ \AR^{k+1}(\susp\varSigma,\mbf{s}\ast\varSigma).\]
and the second vertical map is simply the cone lemma. An isomorphism on the top is then equivalent to an isomorphism of the bottom map, and restricting to ideals and their Poincar\'e duals gives the desired.
\end{proof}
	
Hence, we can reduce the Hall-Laman relations in their entirety to biased Poincar\'e duality.	

Before we continue, let us sum up what happened here:

The case of $L$-decomposable spheres suggest that we prove the Lefschetz theorem by pure induction: we prove the Hall-Laman-relations for lower-dimensional spheres and the perturbation program to prove the Lefschetz theorem in higher dimensions, leaving us to establish the biased Poincar\'e duality as base case of the Hall-Laman relations.

Instead, we ignore the Lefschetz theorem entirely, and prove biased Poincar\'e duality by induction on dimension, and Lefschetz theorem merely falls off as a side-product.

\section{Inductive step via railways and proof of the main theorem}\label{sec:induct}
	
	We are ready to study the crucial inductive step necessary for Theorem~\ref{mthm:gl}. Following the preceding discussion, we need to prove biased Poincar\'e duality for subcomplexes of spheres. For this purpose, we noticed that we can reduce to the case of railways and m\'etros using Theorem~\ref{thm:env} and Proposition~\ref{prp:railways}: 

We have to verify that for every subcomplex $\varDelta$ of dimension $k-1$ in a PL sphere $\varSigma$ of dimension $d-1=2k-1$, we can show biased Poincar\'e duality. To this end, we show that for a suitable hypersurface envelope for degree $k$ (specifically, the octavian hereditary railway with $\varDelta$ as its republic), we can insure that the conditions of Theorem~\ref{thm:env} apply. We do so by induction on the dimension, using biased Poincar\'e duality for spheres of dimension $2k-3$. %We then turn our attention to the other isomorphisms predicted by the generic Lefschetz theorem.
	
	\subsection{Analysis of links, Part (1)} 
	Consider therefore a railway $U$ of dimension $2k-2$ in a sphere $\varSigma$ in $\R^d$ and of dimension $d-1=2k-1$, and initial segments of vertices $W$ and $W'$ with $W'=W-w$. Let us now verify both conditions, using as assumption the following:
	
\begin{compactitem}[$\circ$]
\item The biased Poincar\'e duality in PL spheres of dimension $d-3$, and
\item the middle Lefschetz isomorphism in PL spheres of dimension $d-4$ (which is a special case of the first assumption), and
\item for sufficiently fine codimension zero m\'etros $\Mu$ for $(k-2)$-dimensional complexes in spheres $\varSigma$ of dimension $d-3$, the transversal prime properties: The images of
\begin{equation}\label{eq:trans1}
\AR_{k}(\varSigma,\varSigma-\Mu^\circ)\ \xrightarrow{\ \cdot \ell\ }\ \AR_{k-1}(\varSigma)
\end{equation}
and
\begin{equation}\label{eq:trans2}
\BR_{k}(\Mu,\partial \Mu)\ \xrightarrow{\ \cdot \ell\ }\ \AR_{k-1}(\varSigma)
\end{equation}
satisfy biased Poincar\'e duality in $\varSigma$
where we assume that $\ell$ is supported on interior vertices of $\Mu$. This again is a consequence of the inductive program and the first assumption, and much simpler than proving the Lefschetz theorem for all manifolds as m\'etros allow for simple decomposition. This derivation follows as in Theorem~\ref{thm:vd}, see also the Section~\ref{sec:metro}.
\end{compactitem}
	To verify that $U$ satisfies assumption Theorem~\ref{thm:env}(1) (or rather, that $U$ can be made to satisfy the same by changing its coordinates) using the inductive program, we consider $\uppi$ is a projection to $\R^{d-1}$, and $\tet$ denotes the height with respect to that projection, which we will construct via the inductive approach. This follows essentially directly as in Section~\ref{sec:genperap}. 
	
In particular, following the envelope Theorem~\ref{thm:env}, we then have to understand the pullback \[
	x_w\BR_{k}(\uppi\hspace{0.3mm} U- W')\ \subset\ 
	\AR_{k-1}(\uppi\hspace{0.3mm} \lk_w U)
	\] along the differential
	$x_w$. 
	
	We abbreviate this subspace by $\mc{K}_W$.
	
\begin{rem}
We have to be careful here, as we are changing the linear system of parameters, or equivalently the vertex coordinates of $\varSigma$ in $\R^d$, so we have to verify that property (4) of railways still applies. This can be done indirectly, using the constructed Lefschetz element to perturb the original coordinates, or directly as follows:
The construction of railways only used the combinatorial structure of the same, and Lefschetz properties of edges outside of the republic of $U$. It is easy to see that choosing the Lefschetz element in general position on interior vertices of $U$ therefore does not affect property (4).
\end{rem}	
	
	To show they satisfy the biased Poincar\'e duality property, we note that $\Lk_w U\, \cap\, \mr{N}_W U$ is a manifold itself. To see how this helps us specifically, we shall argue that the octavian and hereditary condition ensure that 
	$\mc{K}_W$ is approximated by the stress complex of a subcomplex following Lemma~\ref{lem:approx}. 
	
	First, we need the following basic consequence of the first condition of railways.
	
	\begin{lem}
		We have
		\[\langle \AR_k(\uppi\hspace{0.3mm} \st_v U): v\in W' \rangle\ =\ \langle \AR_k(\uppi\hspace{0.3mm} \st_v U): v\in (U[W'])^{(0)}\rangle \subset \AR_k(\uppi\hspace{0.3mm} U)\]
		provided $U$ is in general position and the generic Lefschetz theorem is established for $(d-4)$-spheres.
	\end{lem}
	
	\begin{proof}
		We prove this by induction on the size of $W$. Assume we proved the statement for $W'$, we proceed to prove it when we add the next vertex $w$, where $W=W'\cup\{w\}$. Then 
		\[(\mr{N}_{W'} U\cap \st_w U)-w\]
		is a $(d-4)$-manifold and a regular neighbourhood. It therefore does not support any $k$-stress relative to the boundary provided $U$ is in general position.
	\end{proof}
	
	\begin{lem}
		The orthogonal complement $\mc{K}_W^\bot$ of $\mc{K}_W$ is 
		\begin{align*}
		\ker [h:\AR_{k-1}(\uppi\hspace{0.3mm} (\lk_w U) \cap \uppi\hspace{0.3mm}(\partial U \cup \mr{N}_{W'} U))\ \rightarrow&\ \ \AR_{k-2}(\uppi\hspace{0.3mm} (\lk_w U)]\ +\ \\
		&\AR_{k-1}(\uppi\hspace{0.3mm} \lk_w \mr{N}_{W'} U)\ \subset\ \AR_{k-1}(\uppi\hspace{0.3mm} \lk_w U).
		\end{align*}
	\end{lem}
	
	\begin{proof}
		We have that
		\[\bigoplus_{v\in (U[W'])^{(0)}} \AR_k(\uppi\hspace{0.3mm}\st_{v} U)\ \longrightarrow\ \AR_k(\uppi\hspace{0.3mm}\mr{N}_{W'} U)\]
		is a surjection by assumption (5) and Proposition~\ref{prp:part}. 
		By the previous lemma, we therefore have
		\[\bigoplus_{v\in W'} \AR_k(\uppi\hspace{0.3mm}\st_{v} U)\ \longtwoheadrightarrow\ \AR_k(\uppi\hspace{0.3mm}\mr{N}_{W'} U).\]
		Hence, the claim follows by Lemma~\ref{lem:orth}.
	\end{proof}
	
	This is already in the form we could treat it, essentially encoding the Poincar\'e pairing on the primitive part. It is the third assumption at the beginning of the section. Let us see how this is achieved in detail,  using the assumption that the railway is additionally hereditary.
	
	\subsection{Lefschetz for m\'etros}\label{sec:metro}
	
Let us assume the height $\vartheta$ is chosen as decaying for the order on vertices for the railway in the link, and using Lemma~\ref{lem:approx}, which implies that $\mc{K}_W^\bot$ approximates a stress subspace, which satisfies biased Poincar\'e duality by induction on the dimension. Formulated as a lemma for m\'etros, we have the following:
	
\begin{lem}\label{lem:metros}
Consider $\Mu$ a $(2k-3)$-dimensional hereditary octavian m\'etro realized in $\R^{2k-2}$, and $\ell$ a decaying Lefschetz function along the vertices of the tabula, then the kernels of 
\begin{equation*}
\AR_{k-1}(\Mu)\ \xrightarrow{\ \cdot \ell\ }\ \AR_{k-2}(\Mu)
\end{equation*}
and
\begin{equation*}\BR_{k-1}(\Mu)\ \xrightarrow{\ \cdot \ell\ }\ \BR_{k-2}(\Mu)
\end{equation*}
approximate stress subspaces.
\end{lem}
	
\begin{proof}
This follows as in Proposition~\ref{prp:approx}.
\end{proof}

This shows that given sufficiently general position, we can guarantee that Assumption~(1) is satisfied in Theorem~\ref{thm:env}. Notice that we only used the Lefschetz Isomorphism~\eqref{eq:trans1} for the verification of the first condition of Theorem~\ref{thm:env}, the other statement shall be used for the second condition.

\begin{rem}
We reduced the Lefschetz theorem for m\'etros to biased Poincar\'e duality in lower dimensions, but in fact this reduction is unnecessary, and one can argue purely on the level of m\'etros, similar to the proof of Theorem~\ref{thm:vd}.
\end{rem}
	
\subsection{Maps to homology}
	
Now, we have used the perturbation program to establish that we can ensure Assumption~(1) of Theorem~\ref{thm:env}. For Condition~(2), we have to discuss maps to cohomology, which we do in this section that gives the culminating technology for our purpose.

We discuss here the closed hypersurfaces in spheres, and then in the next section we review how this can be achieved in doubles.
	
	Let us assume we are in the situation of Theorem~\ref{thm:char_ospd}, but that $\Mu$ is a rational sphere~$\varSigma$. We have to establish an isomorphism
	\begin{equation}\label{eq:isoospd}
	\AR_k({E})\ \cong\ (H_{k-1})^{\binom{d}{k}}(\Mu_1).
	\end{equation}
	For this, we write $\R^d$ as $\R^{[d]}$, and consider the projection $\uppi$ to $\R^{[d-1]}$. Consider the subspace
	\[\CR_k(\uppi\hspace{0.3mm} {E})\ \coloneqq\ 
	\ker\left[\AR_k(\uppi\hspace{0.3mm}{E})
	\rightarrow (H_{k-1})^{\binom{[d-1]}{k}}(\Mu_1)\right].
	\]
	Then, if ${\tet}$ denotes the height over $\R^{[d-1]}$, then we obtain the following observation:
	\begin{lem}\label{lem:Ciso}
		Assume we have an isomorphism
		\begin{equation}\label{eq:iso2}
		\ker[\CR_k(\uppi\hspace{0.3mm} {E})\ \xrightarrow{\ \cdot {\tet} \ }\ 
		\AR_{k-1}(\uppi\hspace{0.3mm} E)]\ \cong\ (H_{k-1})^{\binom{[d-1]}{k-1}\ast \{d\}}(\Mu_1)
		\end{equation}
		induced by $\mr{hom}_{\tet}$.
		Then we have Isomorphism~\eqref{eq:isoospd}.
	\end{lem}

Here, we index the map $\mr{hom}$ by the element of the linear system 
${\tet}$.

	\begin{proof}
    Provided Isomorphism~\eqref{eq:iso2} holds, we have in particular a surjection
\[\CR_k(\uppi\hspace{0.3mm} {E})\ \xrightarrow{\ \cdot {\tet} \ }\ 
		\AR_{k-1}(\uppi\hspace{0.3mm} E).\]
Hence, the quotient 
\[\bigslant{\AR_k(\uppi\hspace{0.3mm}{E})}{\CR_k(\uppi\hspace{0.3mm} {E})}\]
is generated by kernel elements of the map
\[\AR_k(\uppi\hspace{0.3mm} {E})\ \xrightarrow{\ \cdot {\tet} \ }\ 
		\AR_{k-1}(\uppi\hspace{0.3mm} E).\]
that map injectively to 
		\[\bigslant{(H_{k-1})^{\binom{d}{k}}(\Mu_1)}{(H_{k-1})^{\binom{[d-1]}{k-1}\ast \{d\}}(\Mu_1)}\]
		which in turn is isomorphic to
\[\bigslant{(H_{k-1})^{\binom{d}{k}}(\Mu_1)}{\mr{hom}_\tet\left(\ker[\CR_k(\uppi\hspace{0.3mm} {E})\ \xrightarrow{\ \cdot {\tet} \ }\ 
		\AR_{k-1}(\uppi\hspace{0.3mm} E)] \right)}\]
		as desired.
		\end{proof}	
	
	We can even turn this into a Lefschetz isomorphism with respect to a supremely interesting pairing! We do so as follows: 
	
	First, assume that ${E}$ is induced in $\varSigma$ (which we can do without loss of generality by Lemma~\ref{lem:PLinv}). Second, assume that $\vartheta$ induces a proper embedding of $\varSigma$ even if it may be degenerate on ${E}$. This is easily obtained by choosing ${\tet}$ generic on the vertices of $\varSigma$ not in the ${E}$.
	Now we define a pairing
	\begin{equation}\label{eq:coolpairing}
	\AR_k({E})\ \times\  (H_{k-1})^{\binom{[d]}{k}}(\Mu_1)\ \longrightarrow\ \R
	\end{equation}
	as follows: Every element of $\AR_k(\Mu_1)$ can naturally be written as a sum of an element of
	$(H_{k-1})^{\binom{[d]}{k}}(\Mu_1)$
	and an element of $\BR_k(\Mu_1)$. But the latter is orthogonal to $\AR_k({E})$ in~$\AR(\varSigma)$, so that the natural pairing
\[\AR_k(\varSigma)\ \times\ \AR_k(\varSigma)\ \longrightarrow\ \R\]	
	restricted to
	\[\AR_k({E})\ \times\ \AR_k(\Mu_1)\ \longrightarrow\ \R\]
	induces the desired pairing via the isomorphism
	\[(H_{k-1})^{\binom{[d]}{k}}(\Mu_1)\ \cong\ \bigslant{\AR_k(\Mu_1)}{\BR_k(\Mu_1)}.\]
	Note that Isomorphism~\eqref{eq:isoospd} is equivalent to saying that Pairing~\eqref{eq:coolpairing} is perfect, and that 
	$\CR_k(\uppi\hspace{0.3mm} {E})$ is the orthogonal complement to $(H_{k-1})^{\binom{[d-1]}{k}}(\Mu_1)$ when $\vartheta=0$.

	\begin{rem}
		These relative pairings are very interesting, and one can draw up an entire Hall-Laman theory for them. We will contend ourselves with developing the basic case we need for the Lefschetz theorem for spheres.
	\end{rem}
	
	Now, we have to adapt the inductive program, specifically Theorem~\ref{thm:perturb}. We do so as follows:
	
	Assuming that we have proven the transversal prime property is proven for a set of vertices $W'$, that is, for ${\tet}[W']\coloneqq``{\sum_{v\in W'}}" x_v$, we have
	\[\ker \left[\mr{hom}_{{\tet}[W']}:\ker {\tet}[W']\ \rightarrow\ (H_{k-1})^{\binom{[d-1]}{k-1}\ast \{d\}}(\Mu_1)\right]
	\ =\ \ann_{\CR_k(\uppi\hspace{0.3mm} {E})}\langle x_v| v\in W' \rangle\]
	where we denote the final space as \[\CR_k(\uppi\hspace{0.3mm} {E}-W'\subset \uppi\hspace{0.3mm} {E})\ =\ \AR_k(\uppi\hspace{0.3mm} {E}-W') \cap \CR_k(\uppi\hspace{0.3mm} {E}).\]
	We have the following useful observation:
	\begin{lem}\label{lem:orthC}
		When restricting the Pairing~\eqref{eq:coolpairing} to
		\[\CR_k({E})\ \times\ (H_{k-1})^{\binom{[d-1]}{k-1}\ast \{d\}}(\Mu_1)\ \longrightarrow\ \R,\]
		where
		\[\CR_k({E})\ \coloneqq\ \ker[\CR_k(\uppi\hspace{0.3mm} {E})\ \xrightarrow{\ \cdot {\tet}[W'] \ }\ 
		\AR_{k-1}(\uppi\hspace{0.3mm} E)]\]
		the spaces
		\[\CR_k(\uppi\hspace{0.3mm} {E}-W'\subset \uppi\hspace{0.3mm} {E})\ \subset\ \CR_k({E})\ \subset\ \AR_k(E)\] and \[\mr{hom}_{{\tet}[W']}\CR_k({E})\ \subset\ (H_{k-1})^{\binom{[d-1]}{k-1}\ast \{d\}}(\Mu_1)\ \subset\ (H_{k-1})^{\binom{[d]}{k}}(\Mu_1)\]
		are orthogonal complements.
	\end{lem}
	\begin{proof}
		It is clear from the discussion above that both subspaces are orthogonal to each other. Indeed, 
		as \[\CR_k(\uppi\hspace{0.3mm} {E}-W'\subset \uppi\hspace{0.3mm} {E})\ \longrightarrow\ (H_{k-1})^{\binom{[d]}{k}}(\Mu_1)\] is a trivial map, then
\[\CR_k(\uppi\hspace{0.3mm} {E}-W'\subset \uppi\hspace{0.3mm} {E}) \ \longhookrightarrow\ \BR_{k-1}(\Mu_1)\]
and in fact
\[\CR_k(\uppi\hspace{0.3mm} {E}-W'\subset \uppi\hspace{0.3mm} {E}) \ =\ \BR_{k-1}(\Mu_1)\ \cap\ \AR_k(\uppi\hspace{0.3mm} {E}-W').\]
But $\BR_{k-1}(\Mu_1)$ and $\CR_k({E})$ are orthogonal, and so are in particular the former and the latter spaces image in $(H_{k-1})^{\binom{[d-1]}{k-1}\ast \{d\}}(\Mu_1)$ under $\mr{hom}_{{\tet}[W']}$. It finally follows from a dimension count that they are indeed orthogonal complements.
	\end{proof}
	
	When we want to prove the transversal prime property for $W=W'\cup \{w\}$, we want to use perturbation.
	We map $\CR_{k}(\uppi\hspace{0.3mm} {E})$ to $\AR_{k-1}(\uppi\hspace{0.3mm} \st_w  {E})$. Notice now that
	\[\im [{\tet}[W']:\CR_{k}(\uppi\hspace{0.3mm} {E})\rightarrow \AR_{k-1}(\uppi\hspace{0.3mm} \st_w  {E})]\cap \AR_{k-1}(\uppi\hspace{0.3mm} \st_w  {E})\] generates a subspace of the kernel of
	\[\bigoplus_{v\in W} \AR_{k-1}(\uppi\hspace{0.3mm} \st_v{E})\ \longrightarrow\ 
	\AR_{k-1}(\uppi\hspace{0.3mm} {E})\]
	which following Proposition~\ref{prp:part} maps to $(H_{k-1})^{\binom{[d-1]}{k-1}}({E})$. 
	Consider the composition with   
	\begin{align*}
	(H_{k-1})^{\binom{[d-1]}{k-1}}({E})\ & \cong\ 
	(H_{k-1})^{\binom{[d-1]}{k-1}\ast \{d\}}({E})\ \longrightarrow\ \\ \longrightarrow\  &(H_{k-1})^{\binom{[d-1]}{k-1}\ast \{d\}}(\Mu_1) \ \longrightarrow\ 
	\bigslant{(H_{k-1})^{\binom{[d-1]}{k-1}\ast \{d\}}(\Mu_1)}{\im[\mr{hom}_{{\tet}[W']}]}
	\end{align*}
	where we abbreviated
	\[\im[\mr{hom}_{{\tet}[W']}]\ =\ \im\left[\mr{hom}_{{\tet}[W']}:\ker {\tet}[W']\ \rightarrow\ (H_{k-1})^{\binom{[d-1]}{k-1}\ast \{d\}}(\Mu_1)\right],\]
	an image we can understand equivalently as the image  of
	\[\ker\left[\bigoplus_{v\in W'} \AR_{k-1}(\uppi\hspace{0.3mm} \st_v{E})\ \longrightarrow\ 
	\AR_{k-1}(\uppi\hspace{0.3mm} {E})\right]\ \longrightarrow\ (H_{k-1})^{\binom{[d-1]}{k-1}\ast \{d\}}(\Mu_1)\]
     under the above composition.
	Consider therefore finally the kernel of this composition
	\[\ker\left[\bigoplus_{v\in W} \AR_{k-1}(\st_v\uppi\hspace{0.3mm} {E})\ \longrightarrow\ 
	\AR_{k-1}(\uppi\hspace{0.3mm} {E})\right]\ \longrightarrow\ \bigslant{(H_{k-1})^{\binom{[d-1]}{k-1}\ast \{d\}}(\Mu_1)}{\im[\mr{hom}_{{\tet}[W']}]}
	\]
	and project it to $\AR_{k-1}(\st_w\uppi\hspace{0.3mm} {E})$. Denote the resulting subspace of $\AR_{k-1}(\st_w\uppi\hspace{0.3mm} {E})$ by $\mathcal{I}_W$. Let on the other hand 
	\[\mathcal{K}_W\ \coloneqq\ x_w\CR_{k}(\uppi\hspace{0.3mm} {E}-W'\subset \uppi\hspace{0.3mm} {E}).\]
\begin{lem}\label{lem:orthC3}
The spaces $\mathcal{K}_W$ and $\mathcal{I}_W$ are orthogonal complements in $\AR_{k-1}(\lk_w \uppi\hspace{0.3mm}E)$
\end{lem}		

\begin{proof}
By Lemma~\ref{lem:PLinv}, we may assume the existence of a vertex $u \in \Mu_1^\circ$ so that
\begin{compactitem}[$\circ$]
\item $u$ projects to the origin under $\uppi$, that is, \[\uppi\hspace{0.3mm} u\ =\ 0\] and
\item we have \[(\st_u \Mu_1) \cap E\ =\ \st_w E.\]
\end{compactitem}
Consider then the preimage \[\widetilde{\mc{J}}_W\ \subset\ \AR_{k}(E\cup \st_u \Mu_1)\] of 
\begin{equation*}
\im [{\tet}[W']:\CR_{k}(\uppi\hspace{0.3mm} {E})\rightarrow \AR_{k-1}(\uppi\hspace{0.3mm} \st_w  {E})]\ \subset\ \AR_{k-1}(E\cap \lk_u \Mu_1)
\end{equation*} with respect to the composition
\begin{equation}\label{eq:whopper}
\AR_{k}(E\cup \st_u \Mu_1)\ \xrightarrow{\ \cdot x_u\ }\ \AR_{k-1}(\st_u \Mu) \ \longrightarrow\ \AR_{k-1}(\lk_u \Mu_1).
\end{equation}
Further restrict to the kernel $\widehat{\mc{I}}_W$ of
\begin{align*}
	\widetilde{\mc{I}}_W\  \longhookrightarrow\ &\AR_{k}(E\cup \st_u \Mu_1)\ \longrightarrow\ \\ \longrightarrow\  &(H_{k-1})^{\binom{[d-1]}{k-1}\ast \{d\}}(\Mu_1) \ \longrightarrow\ 
	\bigslant{(H_{k-1})^{\binom{[d-1]}{k-1}\ast \{d\}}(\Mu_1)}{\im[\mr{hom}_{{\tet}[W']}]}.
	\end{align*}
Under Composition	~\eqref{eq:whopper}, the image of $\widehat{\mc{I}}_W$ is the space ${\mc{I}}_W$ we defined before. Hence, the claim follows by Lemma~\ref{lem:orthC}.
\end{proof}

	From this lemma, we immediately obtain a version of Theorem~\ref{thm:perturb}.
	
	\begin{prp}\label{prp:perturb2}
		In the situation above, the following are equivalent.
		\begin{compactenum}[(1)]
			\item The assumptions of Lemma~\ref{lem:perturbation}(1) hold, that is, 
			\[x_w\CR_{k}(\uppi\hspace{0.3mm} {E}-W'\subset \uppi\hspace{0.3mm} {E})\ \cap\ \mathcal{I}_W \ =\ {0}.\]
			\item The assumptions of Lemma~\ref{lem:perturbation}(2) hold, that is,
			\[x_w^{-1} \left(\AR_{k-1}(\lk_w {E}) \cap \mathcal{I}_W\right)\ +\ \CR_k(\uppi\hspace{0.3mm} {E}-W'\subset \uppi\hspace{0.3mm} {E})\ =\ \CR_k(\uppi\hspace{0.3mm}{E}).\]
			\item The pullback $\mathcal{K}_W$ of $\CR_{k}(\uppi\hspace{0.3mm} {E}-W)$ to 
			\[\AR_{k-1}(\st_w {E})\ \cong\ \AR_{k-1}(\lk_w {E})\]
			along $x_w$ satisfies biased Poincar\'e duality in the $(d-2)$-sphere $\lk_w {E}$. \qedhere
		\end{compactenum}
	\end{prp}
	
	Arguing as in Lemma~\ref{lem:perturbation}, we obtain that if the conditions of Proposition~\ref{prp:perturb2} are satisfied, then the transversal prime property holds for $W$.
	
\subsection{Maps to homology in the double}
	
Let us first reformulate the second condition in Theorem~\ref{thm:env}. Denote by $\CR_k(\uppi\hspace{0.3mm} \mr{D}U)$ the kernel of the map
\[\AR_{k}(\uppi\hspace{0.3mm} \mr{D}U)\ \longrightarrow\ (H_{k-1})^{\binom{[d-1]}{k}}(\widetilde{\mr{D}}\varSigma)
\]
where, for context, we recall that we consider the ambient space $\R^{[d]}$ and the projection $\uppi$ to $\R^{[d-1]}$.

\begin{lem}
To obtain surjectivity of
			\[\AR_k(\mathrm{D}U)\ \longrightarrow\ (H_{k-1})^{\binom{d}{k}}({\widetilde{\mr{D}}\varSigma})\]
it is enough to demonstrate an isomorphism 
			\begin{equation}\label{eq:surjint}\BR_{k}(\uppi\hspace{0.3mm} U) \ \xrightarrow{\ \cdot \tet\ }\ \AR_{k-1}(\uppi\hspace{0.3mm} U, \uppi\hspace{0.3mm} \partial U)		
				\end{equation}
			and a surjection
\begin{equation}\label{eq:doubleLL}
\ker\left[\CR_k(\uppi\hspace{0.3mm} \mr{D}U)\ \xrightarrow{\ \cdot {\tet}\ }\ \AR_{k-1}(\uppi\hspace{0.3mm} \mr{D}U)\right] \ \xrightarrow{\ \mr{hom}_\tet\ }\ 
(H_{k-1})^{\binom{[d-1]}{k-1}\ast \{d\}}(\widetilde{\mr{D}}\varSigma).
\end{equation}
\end{lem}

\begin{proof}
Similar to the proof of Lemma~\ref{lem:Ciso}, but we do no longer care for (or expect for that matter, as $\tet$ must be very special to give a double in the preimage of the projection) the map
\begin{equation}\label{eq:doubleL}
\CR_k(\uppi\hspace{0.3mm} \mr{D}U)\ \xrightarrow{\ \cdot {\tet}\ }\ \AR_{k-1}(\uppi\hspace{0.3mm} \mr{D}U)
\end{equation}
 to be a surjection.

An important subspace of $H_{k-1}(\mr{D}U)$ to keep in mind is induced by the folding map:
\[\mc{F}\ \coloneqq\ \ker[\uptau: H_{k-1}(\mr{D}U)\ \longrightarrow\ H_{k-1}(U)].\]

With this, the quotient
\begin{equation}\label{eq:quot}
\bigslant{\AR_k(\uppi\hspace{0.3mm}\mr{D}U)}{\CR_k(\uppi\hspace{0.3mm} \mr{D}U)}
\end{equation}
is generated by preimage of
$\mc{F}^{\binom{[d-1]}{k}}$ in the map
\[\AR_k(\uppi\hspace{0.3mm}\mr{D}U)\ \xrightarrow{\ \mr{hom}\ }
\ (H_{k-1})^{\binom{[d-1]}{k}}(\widetilde{\mr{D}}U).
\]
A basis for the Quotient~\eqref{eq:quot} can therefore be chosen to be in the preimage of the folding map intersected with the kernel of the folding map 
\[\AR_k(\uppi\hspace{0.3mm}\mr{D}U)\ \xrightarrow{\ \uptau\ }\ \AR_k(\uppi\hspace{0.3mm}U).\]
Finally, we use surjectivity of the Map~\eqref{eq:surjint}, and conclude that we can choose the basis to map isomorphically to
		\[\mc{F}^{\binom{[d-1]}{k}}\ \cong\ \bigslant{(H_{k-1})^{\binom{d}{k}}(\widetilde{\mr{D}}\varSigma)}{(H_{k-1})^{\binom{[d-1]}{k-1}\ast \{d\}}(\widetilde{\mr{D}}\varSigma)}\]
		which in turn is isomorphic to
\[\bigslant{(H_{k-1})^{\binom{d}{k}}(\widetilde{\mr{D}}\varSigma)}{\mr{hom}_\tet\left(\ker[\CR_k(\uppi\hspace{0.3mm} \mr{D}U)\ \xrightarrow{\ \cdot {\tet}\ }\ \AR_{k-1}(\uppi\hspace{0.3mm} \mr{D}U)]\right)}\]
		as was claimed.
\end{proof}

Now, let us consider Theorem~\ref{thm:env}(2). To prove surjectivity of the Map~\eqref{eq:doubleLL}, we use inductive program as in the previous section.

%An important subspace of $\CR_k(\uppi\hspace{0.3mm} \mr{D}U)$ is 
%\[\widetilde{\CR}_k(\uppi\hspace{0.3mm} \mr{D}U)\ =\ \CR_k(\uppi\hspace{0.3mm} \mr{D}U) \cap \langle\AR_{k}(\uppi\hspace{0.3mm} \hat{U}),\AR_{k}(\uppi\hspace{0.3mm} \check{U})\rangle,\]
% where $\hat{U}$ and $\check{U}$ are the two charts.
			
 Let $\{\hat{v}, \check{v}\}=\uptau^{-1}v$ denote the preimages in the associated charts for interior vertices $v$ of~$U$. Assume we have proven the appropriate stable annihilator property for $W'$, that is,
\begin{compactenum}[(1)]
\item if 
\[{\tet}[W']\ \coloneqq\ ``{\sum_{v\in W'}}"  (x_{\hat{v}}+ x_{\check{v}})\]
and $\mr{hom}_{{\tet}[W']}$ is the associated map to homology, the image of 
\[\mr{hom}_{{\tet}[W']}{:}\, \ker \left[{\tet}[W']:\CR_k(\uppi\hspace{0.3mm} \mr{D}U)\ \rightarrow\ \AR_{k-1}(\uppi\hspace{0.3mm} \mr{D}U)\right] \rightarrow (H_{k-1})^{\binom{[d-1]}{k-1}\ast \{d\}}(\widetilde{\mr{D}}\varSigma)\]
is as large as possible: it equals the image of \[\ker\left[\bigoplus_{v\in \uptau^{-1}W'} \AR_{k-1}(\uppi\hspace{0.3mm} \st_v \mr{D}U)\ \longrightarrow\ 
\AR_{k-1}(\uppi\hspace{0.3mm} \mr{D}U)\right]\]
in $(H_{k-1})^{\binom{[d-1]}{k-1}\ast \{d\}}(\widetilde{\mr{D}}\varSigma)$. Notice that because 
\[\mr{N}_{\uptau^{-1} W'} \mr{D}U\ =\ \bigcup_{v\in \uptau^{-1}W'} \st_v \mr{D}U\]
is a manifold by Condition (2) for railways, this image can be identified as the image of
\[\left(\mc{F}\cap H_{k-1}(\mr{N}_{\uptau^{-1} W'} \mr{D}U)\right)^{\binom{[d-1]}{k-1}\ast \{d\}}\ \longrightarrow\ (H_{k-1})^{\binom{[d-1]}{k-1}\ast \{d\}}(\widetilde{\mr{D}}\varSigma
).\]
Furthermore,
\item the image
\[\im \left[{\tet}[W'] :\CR_k(\uppi\hspace{0.3mm} \mr{D}U)\ \rightarrow\ \AR_{k-1}(\uppi\hspace{0.3mm} \mr{D}U)\right]\]
contains the images of
\[\im\left[{\ x_{\hat{v}}+ x_{\check{v}}\ }:\BR_k(\uppi\hspace{0.3mm} \hat{U}) \ \oplus\ \BR_k(\uppi\hspace{0.3mm} \check{U})\ \xrightarrow\ \AR_{k-1}(\uppi\hspace{0.3mm} \mr{D}U) \right],\]
for $v\in W'$. This is simply the Lefschetz theorem we proved iteratively for step I.
\end{compactenum} 

Additionally, by the Lefschetz theorem for $(d-4)$-spheres, we may assume that the stress space $\AR_{k}(\uppi\hspace{0.3mm} \hat{U}-\hat{W'})$ and therefore also $\AR_{k}(\uppi\hspace{0.3mm} \check{U}-\check{W'})$
have no degree $k$ stresses outside of their manifold part.

As in the previous section, \[\im [{\tet}[W'] :\CR_k(\uppi\hspace{0.3mm} \mr{D}U)\ \rightarrow\ \AR_{k-1}(\uppi\hspace{0.3mm} \mr{D}U)]\cap \langle \AR_{k-1}(\uppi\hspace{0.3mm} \st_{\hat{w}}  \mr{D}U),\AR_{k-1}( \uppi\hspace{0.3mm} \st_{\check{w}} \mr{D}U)\rangle\] generates a subspace of the kernel of
	\[\bigoplus_{v\in U^{(0)} } \AR_{k-1}(\uppi\hspace{0.3mm} \st_v \mr{D}U)\ \longrightarrow\ 
	\AR_{k-1}(\uppi\hspace{0.3mm} \mr{D}U)\]
	which following Proposition~\ref{prp:part} maps to $(H_{k-1})^{\binom{[d-1]}{k-1}}(\mr{D}U)$. 
	
	Consider the composition with   
	\begin{align*}
	(H_{k-1})^{\binom{[d-1]}{k-1}}( & \mr{D}U)\  \cong\ 
	(H_{k-1})^{\binom{[d-1]}{k-1}\ast \{d\}}(\mr{D}U)\ \longrightarrow\ \\ & \longrightarrow\  (H_{k-1})^{\binom{[d-1]}{k-1}\ast \{d\}}(\widetilde{\mr{D}}\varSigma) \ \longrightarrow\ 
	\bigslant{(H_{k-1})^{\binom{[d-1]}{k-1}\ast \{d\}}(\widetilde{\mr{D}}\varSigma)}{\im[\mr{hom}_{{\tet}[W']}]}
	\end{align*}
	where $\im[\mr{hom}_{{\tet}[W']}]$ has an especially nice description by transversal prime property: It is the image of the $(k-1)$-st homology of the manifold $\mr{N}_{\uptau^{-1} W'} \mr{D}U$ in $\widetilde{\mr{D}}\varSigma$ (to the $\binom{[d-1]}{k-1}\ast \{d\}$-th power).
	
	Consider the projection of the kernel of this composition to
	\[
	\AR_{k-1}(\uppi\hspace{0.3mm} \st_{\hat{w}}  \mr{D}U)\ \oplus\ \AR_{k-1}( \uppi\hspace{0.3mm} \st_{\check{w}} \mr{D}U)\
	\]
	and call it $\mathcal{I}_W$. 

\subsection{Analysis of links, Part (2)} 
We can now finish the proof, as we have reduced the difficult problem of constructing Lefschetz maps to a homological problem. 

If $w$ is the next vertex of the linear order as before, then we have as in Lemma~\ref{lem:orthC}.
\begin{lem}
$\mathcal{I}_W$ is orthogonal to the pullback of
\[\widetilde{\CR}_k(\uppi\hspace{0.3mm} \mr{D}U-\uptau^{-1}W')\] along \[x_{\uptau^{-1}w}\ \coloneqq\ x_{\hat{w}}+x_{\check{w}}\] to 
\[\widetilde{\AR}_{k-1}(\uppi\hspace{0.3mm} \lk_{{w}} U)\ \coloneqq\ \AR_{k-1}(\uppi\hspace{0.3mm} \lk_{\hat{w}} \mr{D}U) \ \oplus\ \AR_{k-1}(\uppi\hspace{0.3mm} \lk_{\check{w}} \mr{D}U)\]
with respect to the antidiagonal fundamental class, that is, the fundamental class of $\mr{D}U$ under pullback along $x_{\uptau^{-1}w}$. \qed
\end{lem}

Here,
$\widetilde{\CR}_k(\uppi\hspace{0.3mm} \mr{D}U-\uptau^{-1}W')$ is the kernel of
\[{\AR}_k(\uppi\hspace{0.3mm} \mr{D}U-\uptau^{-1}W')\ \longrightarrow\ \left(\bigslant{H_{k-1}(\widetilde{\mr{D}}\varSigma)}{ H_{k-1}(\widetilde{\mr{D}}\varSigma) }\right)^{\binom{[d-1]}{k}}.\]
Restrict $x_{\uptau^{-1}w}\widetilde{\CR}_k(\uppi\hspace{0.3mm} \mr{D}U-\uptau^{-1}W')$ to classes that are in the kernel of the folding map $\uptau$, obtaining a subspace $\mc{S}$ of
$\widetilde{\AR}_{k-1}(\uppi\hspace{0.3mm} \lk_{{w}} U)$. This subspace is naturally invariant under the natural involution map $\upiota$ that exchanges the charts. Note that this involution preserves the pairing. 

It suffices to consider this subspace as only these classes give rise to homology cycles in $\mc{F}$. It would be enough to show that this pullback satisfies biased Poincar\'e duality. However, we do something simpler: First observe that, because we only have to take care of the pullback to one chart by symmetry of the pairing. This is already simpler, but we can simplify further.

However, we will work with another decomposition: We shall construct a nice subspace $\mc{T}$ of
\[\widetilde{\CR}_k(\uppi\hspace{0.3mm} \mr{D}U-\uptau^{-1}W')\ \subset \widetilde{\AR}_{k-1}(\uppi\hspace{0.3mm} \lk_{{w}} U)\]
such that $\mc{T}$ maps into $\mc{S}$ under $\mr{id}-\upiota$. 

We do so by exploiting condition (6) of octavian railways and obtain:

\begin{lem}\label{lem:fold2}
The kernel 
\[\mc{H}\ \coloneqq\ \ker\left[H_{k-1}(\uppi\hspace{0.3mm} \mr{D}{U}-\uptau^{-1}{W'})\ \longrightarrow\ H_{k-1}(\uppi\hspace{0.3mm} \mr{D}U)\ \longrightarrow\ \bigslant{H_{k-1}(\widetilde{\mr{D}}\varSigma)}{ H_{k-1}(\widetilde{\mr{D}}\varSigma)}\right]\]
surjects onto $H_{k-2}(\lk_w ({U}-{W'}))$ under the folding map $\uptau$.\qed
\end{lem}

Here, the map to $H_{k-2}(\lk_w ({U}-{W'}))$ factors as
\begin{align*}H_{k-1}(\uppi\hspace{0.3mm} \mr{D}{U}-\uptau^{-1}{W'})\ \longrightarrow\ H_{k-1}(\st_{\uptau^{-1}w}^\circ (\uppi\hspace{0.3mm} &\mr{D}{U}-\uptau^{-1}{W'}) )\ \longrightarrow\ \\ \longrightarrow\  &H_{k-2}(\uppi\hspace{0.3mm}\mr{C}\widetilde{U}_w) \ \xrightarrow{\ \uptau\ }\ 
H_{k-2}(\lk_w ({U}-{W'}))
\end{align*}
where
\[\st_{\uptau^{-1}w} (\uppi\hspace{0.3mm} \mr{D}{U}-\uptau^{-1}{W'})\ \coloneqq\ \st_{\hat{w}} (\uppi\hspace{0.3mm} \mr{D}{U}-\uptau^{-1}{W'}) \cup \st_{\check{w}} (\uppi\hspace{0.3mm} \mr{D}{U}-\uptau^{-1}{W'})\]
and
\[\uppi\hspace{0.3mm}\mr{C}\widetilde{U}_w\ \coloneqq\ (\uppi\hspace{0.3mm} \mr{C}\hat{U}_w) \cup (\uppi\hspace{0.3mm} \mr{C}\check{U}_w)\]
and finally 
\[\st_{\uptau^{-1}w}^\circ( \uppi\hspace{0.3mm} \mr{D}{U}-\uptau^{-1}{W'})\ \coloneqq\ (\st_{\uptau^{-1}w} (\uppi\hspace{0.3mm} \mr{D}{U}-\uptau^{-1}{W'}),\uppi\hspace{0.3mm}\mr{C}\widetilde{U}_w ).\]

We can therefore argue as follows: Consider a subspace \[\widetilde{\mc{H}}\ \subset\ H_{k-1}(\st_{\uptau^{-1}w}^\circ (\uppi\hspace{0.3mm} \mr{D}{U}-\uptau^{-1}{W'}))
\] of
 the image of $\mc{H}$ in $H_{k-1}(\st_{\uptau^{-1}w}^\circ (\uppi\hspace{0.3mm} \mr{D}{U}-\uptau^{-1}{W'}))$ such that the above composition maps $\widetilde{\mc{H}}$ maps isomorphically to $H_{k-2}(\lk_w ({U}-{W'}))$. 
Then we can consider 
\[\AR_{k}( \st_{\uptau^{-1}w}^\circ (\uppi\hspace{0.3mm} \mr{D}{U}-\uptau^{-1}{W'})) \ \longrightarrow\ 
 (H_{k-1})^{\binom{[d-1]}{k}}(\st_{\uptau^{-1}w}^\circ (\uppi\hspace{0.3mm} \mr{D}{U}-\uptau^{-1}{W'}))\]
and the preimage \[\mc{D}_w\subset \AR_{k}( \st_{\uptau^{-1}w}^\circ (\uppi\hspace{0.3mm} \mr{D}{U}-\uptau^{-1}{W'})) \] of 
$\widetilde{\mc{H}}^{\binom{[d-1]}{k}}$. Then 
\[x_{\uptau^{-1}w}\mc{D}_w\ \subset\ x_{\uptau^{-1}w}\widetilde{\CR}_k(\uppi\hspace{0.3mm} \mr{D}U-\uptau^{-1}W')\]
and
\[((\mr{id}-\upiota)\circ x_{\uptau^{-1}w})(\mc{D}_w)\ \subset\ \mc{S}\]
so that we identify $x_{\uptau^{-1}w}\mc{D}_w$ as the desired space $\mc{T}$.

We shall argue the following.
\begin{compactenum}[(1)]
\item Biased Poincar\'e duality for $\mc{T}$ implies the transversal prime property for doubles with respect to the vertex set $W=W'\cup\{w\}$.
\item Biased Poincar\'e duality for $\mc{T}$ holds.
\end{compactenum}

To the first point, we simply notice that biased Poincar\'e duality for $\mc{T}$ implies that 
\[x_{\uptau^{-1}w}\mc{D}_w \cap \AR_{k-1}(\st_{w} \uppi\hspace{0.3mm} \mr{D}{U} \cap \uppi\hspace{0.3mm}(\partial U \cup \mr{N}_{\uptau^{-1}{W'}} \mr{D}U))\]
maps surjectively to
\[(H_{k-2})^{\binom{[d-1]}{k-1}}
\left(\st_{w} \hspace{0.3mm} \mr{D}{U} \cap \hspace{0.3mm}(\partial U \cup \mr{N}_{\uptau^{-1}{W'}} \mr{D}U)\right).\]

But then 
\[(\mc{T}\ \cup\ \upiota\mc{T})\ \cap\ \AR_{k-1}(\st_{w}\uppi\hspace{0.3mm} \mr{D}{U}\ \cap\ \uppi\hspace{0.3mm}(\partial U \cup \mr{N}_{\uptau^{-1}{W'}} \mr{D}U))
\]  
generates 
\[(H_{k-1})^{\binom{[d-1]}{k-1}}
\left(\st_{w} \hspace{0.3mm} \mr{D}{U}\ \cap\ \hspace{0.3mm}(\partial U \cup \mr{N}_{\uptau^{-1}{W'}} \mr{D}U) \right)\]
which implies that 
\begin{align*}
\mr{hom}_{{\tet}[W]}\ &=\  \mr{hom}_{{\tet}[W']}\ \\ &+\ \im\left[(H_{k-1})^{\binom{[d-1]}{k-1}}
\left(\st_{w} \mr{D}{U} \cap \hspace{0.3mm}(\partial U \cup \mr{N}_{\uptau^{-1}{W'}} \mr{D}U\right))\rightarrow (H_{k-1})^{\binom{[d-1]}{k-1}\ast \{d\}}(\widetilde{\mr{D}}\varSigma)\right] \\
&=\ \im\left[\left(\mc{F}\cap H_{k-1}(\mr{N}_{\uptau^{-1} W} \mr{D}U)\right)^{\binom{[d-1]}{k-1}\ast \{d\}} \rightarrow (H_{k-1})^{\binom{[d-1]}{k-1}\ast \{d\}}(\widetilde{\mr{D}}\varSigma
)\right].
\end{align*}

To achieve the second point, notice that we can choose $\widetilde{\mc{H}}$ within some bounds, and in particular can choose find a codimension zero submanifold $\varUpsilon$ of 
 $\uppi\hspace{0.3mm}\mr{C}\widetilde{U}_w$ such that
$\widetilde{\mc{H}}$ defines relative homology classes in 
\[H_{k-1}(\st_{\uptau^{-1}w}^\circ(\mr{D}{U}-\uptau^{-1}{W'}),\varUpsilon)\]
and such that the composition
\[\widetilde{\mc{H}}\ \longhookrightarrow\ 
H_{k-1}(\st_{\uptau^{-1}w}^\circ (\mr{D}{U}-\uptau^{-1}{W'}),\varUpsilon)
\ \longrightarrow\ H_{k-2}(\varUpsilon)\]
is an isomorphism. If $k\ge 3$, we can simply choose $\varUpsilon$ to be the restriction of
 $\uppi\hspace{0.3mm}\mr{C}\widetilde{U}_w$ to one chart, and if $k=2$, then $\varUpsilon$ is naturally chosen so that the folding map defines an isomorphism of manifolds on it, see Figure~\ref{etale}. For $k=1$ we already solved the problem completely in Example~\ref{ex:smooth}. With this choice, the second point follows immediately from Lemma~\ref{lem:metros}. To see this in detail, we argue as follows:

Hence, we can represent the image of $\widetilde{\mc{H}}$ in by codimension zero submanifolds. We obtain a pair $(X,\varUpsilon)$ (see also Figure~\ref{etale}) such that 
\begin{compactitem}[$\circ$]
\item $X$ is a $(k-2)$-acyclic manifold ${X}$ of dimension $2k-2$ that maps simplicially and locally injectively to \[
\upvarphi({X})\ =\ \st_{\uptau^{-1}w} (\uppi\hspace{0.3mm} \mr{D}{U}-\uptau^{-1}{W'})\ \subset\ \uppi\hspace{0.3mm}\mr{D}U.\]
\item The $(k-3)$-connected submanifold of its boundary ${\varUpsilon}$ maps isomorphically to the complement $\uppi\hspace{0.3mm}\mr{C}{U}_w$.% so that $H_{k-1}(X,\varUpsilon)$ maps isomorphically to $\widetilde{\mc{H}}$.
\item Moreover, the relative stresses $\AR_k({X},{\varUpsilon})$ map surjectively to $\mc{S}$
\[\mc{T}\ =\ \left(x_{\uptau^{-1}w}\circ\upvarphi\right) (\AR_k({X},{\varUpsilon})).\]
\end{compactitem}

		\begin{figure}[h!tb]
			\begin{center}
				\includegraphics[scale = 1.2]{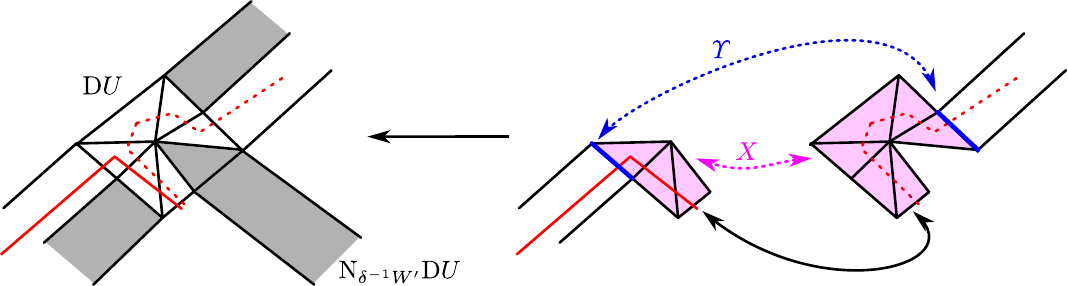}
				\caption{The pullback of ${\AR}_k(\uppi\hspace{0.3mm} \mr{D}U-\uptau^{-1}W')$ to $\st_{\uptau^{-1}w} \uppi\hspace{0.3mm} \mr{D}U$ is decomposed along the relative homology classes of $\im[\ker[H_{k-1}(\mr{D}U)\rightarrow {H_{k-1}(\widetilde{\mr{D}}\varSigma)}/{ H_{k-1}(\widetilde{\mr{D}}\varSigma)}]\rightarrow H_{k-1}(H_{k-1}( \st_v \mr{D}U, \mr{C}{U}_w)]$ by reducing it to a pairing question in a simpler space ${X}$.}
				\label{etale}
			\end{center}
		\end{figure} 
		
This allows us to finish the proof, as simplicial maps descend to morphisms of face rings and we can therefore discuss $\AR_k({X},{\varUpsilon})$ instead of the doubled railway: We denote by $\mr{C}{\varUpsilon}_w$ the complement of ${\varUpsilon}$ in $\lk_{w}^\upvarphi {X}$, and by $U_{\widetilde{w}}$ the copy of the star  $\st_w \uppi\hspace{0.3mm} U$ in the chart corresponding to the vertex $\widetilde{w}\ \in\ (\upvarphi\circ\uptau)^{-1} w$. Then the orthogonal complement to $\AR_k({X},{\varUpsilon})$ in 
\[\lk_{w}^\upvarphi {X}\ \coloneqq\ \bigcup_{\widetilde{w}\in(\upvarphi\circ\uptau)^{-1} w } \lk_{\widetilde{w}} {X}\]
is described as
\begin{align*}
\ker[\vartheta\BR_k(\mr{C}{\varUpsilon}_w)\ \longrightarrow\  & \BR_{k-1}(\mr{C}{\varUpsilon}_w)]\ +\ \\
\ & \bigoplus_{\widetilde{w}\in(\upvarphi\circ\uptau)^{-1} w } \AR_{k-1}(\lk_{\widetilde{w}} {U_{\widetilde{w}} - X^\circ})\ \subset\ \bigoplus_{\widetilde{w}\in(\upvarphi\circ\uptau)^{-1} w } \AR_{k-1}(\lk_{\widetilde{w}} U_{\widetilde{w}})
\end{align*}
which we know for sufficiently general position Artinian reduction and Lefschetz element
satisfies biased Poincar\'e duality by the Lefschetz theorem for m\'etros, see also the second part of Lemma~\ref{lem:metros}. 
Hence, we can ensure that the railway $U$ satisfies the second condition of Theorem~\ref{thm:env} by induction on dimension.
		
	\section{Beyond PL spheres}\label{sec:mani}
	
	For manifolds that are not PL spheres, the same proof goes through with minor caveats. These are exclusively related to the reduction to railways, which are easy construct in combinatorial manifolds using basic general position techniques, but which demands some additional words beyond that case. 
	
The main issue is that in general, given a simplicial complex $\varDelta$ of dimension $k\le\nicefrac{d-1}{2}-1$ contained in a homology sphere of dimension $d-1$, it is not clear that $\varDelta$ is contained in a hypersurface of trivial homology, making the construction of railways in this case tricky.
While it is possible to remedy this by considering Buchsbaum railways instead of ones that are railways, this leads us rather far astray. 	
	
For simplicity, we shall therefore not prove the Hall-Laman relations for general manifolds, but only the hard Lefschetz theorems.
	
	\subsection{Homology spheres}
	
	Consider for instance first the case of a combinatorial rational homology sphere $\varSigma$ of dimension $d-1$. First, let us note that
	the proof of biased Poincar\'e duality for combinatorial rational homology spheres goes through verbatim (as we already established the existence in the case of combinatorial rational homology spheres), completing the proof of Theorem~\ref{mthm:opd}, so it remains to discuss the Lefschetz theorem.
	
	If $\varSigma$ is not a PL sphere, then $\susp\varSigma$  is not a PL sphere, but we can use results of Freedman and Kervaire~\cite{Freedman, Kervaire} to instead realize $\varSigma$ as hypersurface in a PL sphere $\varSigma'$. Following the same program as above, we obtain immediately the middle hard Lefschetz theorem for $\varSigma$ by proving biased Poincar\'e duality for $\varSigma'$ with respect to the subcomplex $\varSigma$. To obtain the higher Lefschetz isomorphisms, we want to use Lemma~\ref{lem:midred}:

If we want to prove to the Lefschetz theorem, say, from degree $k$ to degree $d-k$, then Lemma~\ref{lem:midred} asks us to prove biased Poincar\'e duality for the $(d-k-1)$-skeleton of $\partial \Delta_{i} \ast \varSigma$, where  $i=d-2k$, in the $i$-fold suspension $\susp^i \varSigma$ of $\varSigma$. However that in general is not a combinatorial manifold. However, we can still use the observation: The suspension vertices in the $i$-fold suspension form a $i$-dimensional crosspolytope $\blacklozenge^i$. Using the Freedman-Kervaire construction, we can modify \[\susp^i \varSigma\ =\ \blacklozenge^i \ast \varSigma\]
outside of 
\[ (\blacklozenge^i)^{(\le i-2)} \ast \varSigma,\] where $(\cdot)^{(\le j)}$ denotes the $j$-skeleton of a simplicial complex, to a triangulated $(d+i-1)$-sphere. In this modification, the subcomplex $\partial \Delta_{i} \ast \varSigma$ is not affected, hence we are reduced back to proving biased Poincar\'e duality in a PL sphere.

However, this still only works for combinatorial rational homology spheres, but not rational spheres. For this, we can be slightly more clever about the octavian railway construction, and exploit the fact that suspensions are always simply connected. Consider again
 $\partial \Delta_i \ast \varSigma$, with $i=d-2k$, in the $i$-fold suspension $\susp^i \varSigma$ of $\varSigma$. Now, we can construct a octavian railway for this complex by starting with exactly those PL singular poles, circumventing the issue as, of course, the link of a singular vertex in $\partial \Delta_i \ast \varSigma$ is already a rational homology sphere. Then, we turn to the vertices of $\varSigma$ itself, whose links are simply connected.
   This can then be exploited to complete the railway construction in the desired generality, which finishes the proof.
   
Let us summarize:
   
\begin{thm}\label{mthm:h}
		Consider a triangulated rational $(d-1)$-sphere $\varSigma$, and the associated graded commutative ring $\R[\varSigma]$. Then
		there exists an open dense subset of the Artinian/socle reductions $\mathcal{R}$ of $\R[\Mu]$ and an open dense subset $\mathcal{L} \subset \R^1[\varSigma]$, where $\AR(\varSigma)\in \mc{R}$, such that for every $k\le \nicefrac{d}{2}$, we have the
\begin{compactenum}[(1)]
\item \emph{Generic Lefschetz theorem:}  For every $\AR(\varSigma)\in \mathcal{R}$ and every $\ell \in \mathcal{L}$, we have an isomorphism 
			\[\AR^k(\varSigma)\ \xrightarrow{\ \cdot \ell^{d-2k} \ }\ \AR^{d-k}(\varSigma).\]
\item \emph{Biased Poincar\'e duality:} If $\varSigma$ is combinatorial, then in addition
the perfect pairing
			\[\begin{array}{rccccc}
\AR^k(\varSigma)& \times &\AR^{d-k}(\varSigma)& \longrightarrow &\ \AR^d(\varSigma)\cong \R \\
			a		&	& b& {\xmapsto{\ \ \ \ }} &\ \mr{deg}(ab)
			\end{array}\]
		is non-degenerate in the first factor when restricted to any squarefree monomial ideal $\mc{I}$.
\end{compactenum}
\end{thm}

	\subsection{The main theorem for manifolds}
	
For manifolds, we again restrict to Lefschetz theorem. Anticipating the final result, we obtain

\begin{thm}\label{mthm:gl2}
		Consider a triangulated rational $(d-1)$-manifold $\Mu$, and the associated graded commutative ring $\R[\Mu]$. Then
		there exists an open dense subset of the Artinian/socle reductions $\mathcal{R}$ of $\R[\Mu]$ and an open dense subset $\mathcal{L} \subset \R^1[\Mu]$, where $\AR(\Mu)\in \mc{R}$, such that for every $k\le \nicefrac{d}{2}$, we have the
\begin{compactenum}[(1)]
\item \emph{Generic Lefschetz theorem:}  For every $\AR(\Mu)\in \mathcal{R}$ and every $\ell \in \mathcal{L}$, we have an isomorphism 
			\[\BR^k(\Mu)\ \xrightarrow{\ \cdot \ell^{d-2k} \ }\ \BR^{d-k}(\Mu).\]
\item \emph{Biased Poincar\'e duality:} If $\Mu$ is combinatorial, then in addition
the perfect pairing
			\[\begin{array}{rccccc}
\BR^k(\Mu)& \times &\BR^{d-k}(\Mu)& \longrightarrow &\ \BR^d(\Mu)\cong \R \\
			a		&	& b& {\xmapsto{\ \ \ \ }} &\ \mr{deg}(ab)
			\end{array}\]
		is non-degenerate in the first factor when restricted to any squarefree monomial ideal $\mc{I}$.
\end{compactenum}
\end{thm}
	
	Only few modifications are necessary, and we sketch the proof of the above theorem by surveying the necessary modifications. The first is the definition of railways we need to prove the second statement.
	
	\subsection{Railways and m\'etros in manifolds}
	We fix this as follows. We keep Conditions~(1)-(4) as is, and modify the last conditions using the ambient manifold $\Mu$.
	
	\begin{compactenum}[(1')]
		\setcounter{enumi}{4}
		\item $\mr{N}_{W'} U\hookrightarrow \Mu$ induces an injection in homology up to dimension $(k-2)$ for all initial segments $W'$.
		\item The complementary manifold to $\mr{N}_{W'} U$ in $U$ mapped into $\Mu$ induces an injection in homology up to dimension $(k-2)$.
	\end{compactenum}
	
	A (octavian) m\'etro in $\Mu$ is again a railway of codimension zero and satisfying Conditions~(1) to (3), (5') and (6'). It is again an easy exercise to show hereditary octavian railways in combinatorial $(d-1)$-manifolds $\Mu$ exist for every subcomplex $\Delta$ of dimension at most $\frac{d}{2}-1$.

	\subsection{Using railways, revisited}
	
	The proof of Theorem~\ref{mthm:gl2} for works as the proof for spheres, reducing everything to the case of the middle pairing using the following version of Lemma~\ref{lem:midred}. This is straightforward if one is only interested in the middle Lefschetz isomorphism, but requires some definitions in general.
	
	Let $\susp^i$ denote the $i$-fold suspension, let $\mbf{n}_i$ denote the $i$-th north pole, and let $\Delta_{i-1}$ denote the simplex on the first $(i-1)$ north poles.
	
	\begin{lem}\label{lem:midred2}
Considering $\susp^i\Mu$ realized in $\R^{d+i}$, and $k+i\le \frac{d}{2}$, the following two are equivalent:
		\begin{compactenum}[(1)]
					\item The Hall-Laman relations for \[\overline{\KK}^{k+1}(\susp^i\Mu,\Delta_{i}\ast\Mu)\]
					with respect to the map $x_{\mbf{n}_i}$.									
			\item The Hall-Laman relations for 
 \[\overline{\KK}^{k}(\uppi\hspace{0.3mm} \susp^{i-1}\Mu,\uppi\hspace{0.3mm} \Delta_{i-1}\ast \Mu)\]			
with respect to $\tet$. \qed
		\end{compactenum}
	\end{lem}
		
	Here, we define $\BR^\ast(\susp^i(\Mu))$ as the quotient of $\AR(\susp(\Mu))$ by the interior socle ideal, that is, the ideal generated by socle elements of degree less than $d+i$. Notice that 
	\[
	(H^{k-1})^{\binom{d}{k}}(\Mu))\ \cong\ (\Socl)^k(\susp^i(\Mu))\]
	following the partition double complex of Lemma~\ref{lem:partyyyyy}. The lemma follows as before.
	
	Now, we only need to prove the biased Poincar\'e duality of Lemma~\ref{lem:midred2}. For this, we again need the notion of railways. We define this as follows:
	
	A railway in a subdivision of $\susp^i\Mu$ is a singular hypersurface with boundary that is a rational manifold outside the suspension points, and such that at a face $\sigma$ of $\Delta_i$, the link of $\sigma$ in the railway is up to subdivision coinciding 
	with $\partial ((\Delta_{i}-\sigma)\ast\Mu)$. The octavian condition is modified by 
	considering the radial projection 
	\[\uppsi:\susp^i\Mu{\setminus}\{\text{singular points}\} \ \longrightarrow\ \Mu\]
	the conditions (5) and (6) are modified as
	\begin{compactenum}[(1')]
		\setcounter{enumi}{4}
		\item $\uppsi:\mr{N}_{W'} U\rightarrow \Mu$ induces an injection in homology up to dimension $(k-2)$ for all initial segments $W'$.
		\item The complementary manifold to $\mr{N}_{W'} U$ in $U$, mapped to $\Mu$ along $\uppsi$ induces an injection in homology up to dimension $(k-2)$.
	\end{compactenum}

\subsection{Positive characteristic}

What happens when we do no longer restrict to rational homology spheres, but homology spheres, or even manifolds, with respect to any other field? What works, and what has to be modified, and what does apparently not work?

One thing we rely on is that the field is infinite. That leaves still the possibility that the characteristic is positive.

Immediately, we then see that the approximation results, that is, Lemma~\ref{lem:approx} and its relatives, make less sense in positive characteristic. However, the perturbation lemma, and in particular the Lefschetz theorems, still apply as stated.

There are two ways to see the latter: we worked with stress spaces, the Weil dual for the face ring most of the time, and it is perhaps not immediately clear how to define these in a suitable way in positive characteristic. For this, we replace the differential operator by the Hasse derivative and obtain the appropriate Weil dual in positive characteristic. With this, and apart from the aforementioned approximation lemma, the results apply as usual.

Alternatively, we should note however that we worked with stress spaces precisely because of the nice approximations of primitive subspaces, and they were at no point critical to our investigations. Indeed, we could have worked with face rings throughout, would have to modify our perspective slightly, but then otherwise work as before, and in particular obtain the Lefschetz theorems and Hall-Laman relations.

	{\small
		\bibliographystyle{myamsalpha}
		\bibliography{ref}}

\end{document}